\documentclass[10pt]{amsart}

\usepackage{amsmath,amssymb,amsthm,amscd,mathrsfs,bm,enumitem,mathtools,comment,needspace}
\usepackage[linkcolor=red,citecolor=green]{hyperref}
\usepackage[T1]{fontenc}
\usepackage{indentfirst, xcolor}
\usepackage{graphicx}
\usepackage{tikz-cd}
\usepackage{longtable,array}

\def\p{\partial}
\def\D{\mathfrak{D}}
\def\C{\mathbb{C}}
\def\R{\mathbb{R}}
\def\Z{\mathbb{Z}}

\def\kah{K\"{a}hler }
\def\B{\mathcal{B}}
\def\H{\mathcal{H}}
\def\T{\mathcal{T}}
\def\F{\mathcal{F}}

\newcommand{\detA}{\det A\mkern2mu}
\newcommand{\dist}{\operatorname{dist}}
\newcommand{\bF}{\boldsymbol{F}}
\begin{document}

	\newtheorem{claim}{Claim}
	\newtheorem{theorem}{Theorem}[section]
	\newtheorem{lemma}[theorem]{Lemma}
	\newtheorem{corollary}[theorem]{Corollary}
	\newtheorem{proposition}[theorem]{Proposition}
	\newtheorem{question}{question}[section]
	\newtheorem{definition}[theorem]{Definition}
	\newtheorem{remark}[theorem]{Remark}

	\newtheoremstyle{plainnormal}
	{\topsep}{\topsep}{\normalfont}{}{\bfseries}{.}{ }{}
	\theoremstyle{plainnormal}
	
	\newtheorem{example}[theorem]{Example}

	\numberwithin{equation}{section}

	\title[Short title for running heads]{New Calabi--Yau Metrics of Taub--NUT Type on $\C^{N+1}$}
	\author{Tengfei Ma}
	\address{School of Mathematics, Nanjing University, Nanjing, 210093, P.R. China}
	\email{math\_mtf@163.com}

	\begin{abstract}
		We construct a class of complete non-flat Calabi--Yau metrics on $\mathbb{C}^{N+1}$ for every $N\geq 3$, which generalize the Taub--NUT metrics from $\mathbb{C}^{2}$ and $\mathbb{C}^{3}$ and whose tangent cone at infinity is $\mathbb{R}^{N}$.  
		The construction relies on the generalized Gibbons--Hawking ansatz.  
		A key obstacle is that the volume-form defect of the ansatz fails to decay near certain components of the discriminant locus, producing singularities more severe than those encountered in dimension three, we resolve this by a gluing procedure.
	\end{abstract}
	\maketitle
	\tableofcontents

	\section{Introduction}
	\subsection{Motivation and History}
	The study of Calabi--Yau metrics originates from Calabi's extremal-metric programme \cite{calabi1979metriques} and Yau's subsequent proof of the Calabi conjecture for {compact} \kah manifolds \cite{yau1978ricci}.
	Since then, the existence of complete, non-compact Calabi--Yau metrics has become a central theme in \kah geometry.  
	Seminal results were obtained by Tian–-Yau~\cite{tian1987existence,tian1990complete,tian1991complete}, who constructed such metrics on the complement of a smooth anti-canonical divisor in a projective manifold. Readers may also consult \cite{hein2010gravitational,conlon2013asymptotically,conlon2015asymptotically} and the references therein for further related developments.
	
	In the present article we focus on complete Calabi--Yau metrics on the complex affine space $\C^{N+1}$.  The flat Euclidean metric is a trivial example.  When $N=1$, i.e.,\ in complex dimension two, the Taub--NUT metric supplies a distinguished example.  
	It first appeared in Taub's 1951 cosmological model~\cite{taub1951empty}, was extended by Newman--Tamburino--Unti~\cite{newman1963empty} to incorporate the now-called NUT charge, and was reinterpreted by Gibbons--Hawking~\cite{gibbons1978gravitational} as an ALF gravitational instanton.
	LeBrun~\cite{lebrun1991complete} observed that the underlying manifold of these metrics is diffeomorphic to~$\mathbb{C}^{2}$
	and thus that the  Taub--NUT metric is a complete Calabi--Yau metric on~$\mathbb{C}^{2}$. In particular, the tangent cone at infinity of these metrics is~$\mathbb{R}^{3}$
	and their geodesic balls have non-maximal volume growth.
	\[
	\operatorname{Vol}\bigl(B_{p}(R)\bigr)\sim R^{3}.
	\]
	
	Very recently, Li~\cite{li2023syz} produced a new family of complete Calabi--Yau metrics on~$\mathbb{C}^{3}$ whose tangent cone at infinity is~$\mathbb{R}^{4}$ and whose volume growth rate is again non-maximal:
	\[
	\operatorname{Vol}\bigl(B_{p}(R)\bigr)\sim R^{4}.
	\]
	These metrics are naturally regarded as higher-dimensional analogues of the classical Taub--NUT metric.  
	Li's strategy is to build an ansatz on a model singular~$\mathbb{T}^{2}$-bundle near infinity and then deform it to a genuine Calabi--Yau metric by means of the Tian--Yau--Hein package~\cite{tian1990complete,hein2010gravitational}.  
	The main difficulty is that, near the discriminant locus, the volume-form error of the ansatz decays only at order $1$ with respect to the distance function; whereas the theory of~\cite{hein2010gravitational} requires a decay rate strictly greater than quadratic.
	
	The main result of the present article (Theorem~\ref{maintheorem}) extends this Taub--NUT-type construction to~$\mathbb{C}^{N+1}$ for every~$N\ge 3$, yielding complete Calabi--Yau metrics whose tangent cone at infinity is the flat~$\mathbb{R}^{N+2}$ and whose volume growth is
	\[
	\operatorname{Vol}\bigl(B_{p}(R)\bigr)\sim R^{N+2},
	\]
	hence again non-maximal.  
	Compared with the~$\mathbb{C}^{3}$ case treated in~\cite{li2023syz}, the principal difficulty in higher dimensions is that the volume-form error exhibits substantially worse decay behaviour. Indeed, it can fail to decay at all near certain components of the discriminant locus (see the discussion after Theorem~\ref{maintheorem}).  
	We overcome this obstruction by a careful gluing technique.  
	Our metrics have bounded curvature ($\lVert\mathrm{Rm}\rVert_{L^{\infty}}<\infty$), but they are not $L^{2}$-integrable.
	
	We emphasize that the metrics constructed here do \emph{not} exhibit maximal volume growth.  
	At present, two families of complete Calabi--Yau metrics on~$\mathbb{C}^{N+1}$ with non-maximal growth are known.  
	First, Apostolov--Cifarelli~\cite{apostolov2025hamiltonian} recently produced examples whose volume behaves like
	\[
	\operatorname{Vol}\bigl(B_{p}(R)\bigr)\sim R^{2N+1}.
	\]
	When~$N=1$, their construction recovers the classical Taub--NUT metric on~$\mathbb{C}^{2}$, but for~$N\ge 2$ their growth rate differs from both Li's examples~\cite{li2023syz} and ours.  
	Second, in even dimensions, the so-called Taubian--Calabi metrics on~$\mathbb{C}^{2N}$ ($N\ge 1$) were constructed in~\cite{lee1996massive} and generalised in~\cite{min2023construction}. According to~\cite{gibbons1997hyperkahler}, the terminology was suggested by Ro\v{c}ek~\cite{rovcek1985supersymmetry}.  
	These metrics satisfy
	\[
	\operatorname{Vol}\bigl(B_{p}(R)\bigr)\sim R^{4N-1},
	\]
	and they therefore coincide with the present growth rate only when ~$N=1$.
	
	Finally, we remark that Li~\cite{li2019new}, Conlon--Rochon~\cite{conlon2021new}, Sz\'{e}kelyhidi~\cite{szekelyhidi2019degenerations}, and more recently Firester~\cite{firester2024complete}, have constructed complete Calabi--Yau metrics on~$\mathbb{C}^{N+1}$ whose tangent cones are of the form~$V\times\mathbb{C}$ with~$V$ an~$N$-dimensional Calabi--Yau cone having smooth cross-section.  
	All of these examples enjoy maximal volume growth and are therefore distinct from the metrics constructed in this article.
	
	\subsection{Main Theorem}
	We introduce a holomorphic Lie-group action on $\mathbb{C}^{N+1}$ that will be referred to as the \emph{diagonal $\mathbb{T}^{N}$-action}, defined by
	\begin{equation*}
		(e^{i\theta_{1}},\dots,e^{i\theta_{N}})\cdot(z_{0},z_{1},\dots,z_{N})=\bigl(e^{-i\theta_{\Sigma}}z_{0},\,e^{i\theta_{1}}z_{1},\dots,e^{i\theta_{N}}z_{N}\bigr),
	\end{equation*}
	where $\theta_{\Sigma}\coloneqq\sum_{j=1}^{N}\theta_j$. This action evidently preserves the standard holomorphic volume form
	\[
	\sqrt{-1}^{N}\,\mathrm{d}z_{0}\wedge\dots\wedge\mathrm{d}z_{N}.
	\]
	If a tensor is unchanged under this action, we say that the tensor is $\mathbb{T}^{N}$-invariant. The main result of this paper is as follows:
	
	\begin{theorem}\label{maintheorem}
		There exists a family of complete, non-flat, $\mathbb{T}^{N}$-invariant Calabi--Yau metrics on $\C^{N+1}$  whose tangent cone at infinity is $\R^{N+2}$.
	\end{theorem}
	
	The toric symmetry of the metrics gives rise to $\mathbb{T}^{N}$-fibrations via the corresponding moment maps and \kah quotient coordinates. The base of each fibration is $\mathbb{R}^{N}\times\mathbb{C}$, and the discriminant locus $\D\subset\mathbb{R}^{N}\times\mathbb{C}$ (see Definition~\ref{discriminateloucsdef}) is an unbounded $(N-1)$-dimensional subset.
	
	These metrics are parametrised by positive-definite real symmetric matrices $A=(A_{ij})$ of rank $N$.  Away from a tubular neighbourhood of $\D$, the metric is asymptotically a flat $\mathbb{T}^{N}$-fibration over an open subset of $\R^{N+2}$.  Transverse to $\D$, it is modelled on a fibration of Taub--NUT-type metrics on $\C^{n+1}$ $(1\le n<N)$, where $n$ may vary along different components of the locus.
	
	The volume growth of geodesic balls satisfies
	\[
	\operatorname{Vol}\bigl(B_{p}(R)\bigr)\sim R^{N+2}\quad\text{as }R\to\infty,
	\]
	so these metrics do not have maximal volume growth.  The curvature decays cubically with respect to the distance to $\D$. In particular, it is bounded but does not decay near $\D$.
	
	When $N=1$ or $2$, the Calabi--Yau metrics of Theorem~\ref{maintheorem} correspond to the classical Taub--NUT metrics and the Taub--NUT-type metrics on $\C^{3}$ constructed in \cite{li2023syz}, respectively.  Thus Theorem~\ref{maintheorem} provides a higher-dimensional generalisation of the Taub--NUT metric.

	\subsection{Obstructions and Strategy}
	Compared with the construction of the Taub--NUT metric on $\C^{3}$ in \cite{li2023syz}, the principal difficulty in proving Theorem~\ref{maintheorem} is the deterioration in the volume-form error decay for the first-order asymptotic metric in higher dimensions. This phenomenon essentially arises because, as the dimension increases, the discriminant locus $\D$ acquires more intricate topological features.
	
	Indeed, one can construct a first-order asymptotic metric $g^{(1)}$ using the generalized Gibbons--Hawking ansatz, analogous to Section~2.1 of \cite{li2023syz}; see Subsection~\ref{subsectionfirstorder}. 
	In \cite{li2023syz} the volume-form error decays at least at  order $1$ along the discriminant locus, whereas in our setting it may fail to decay at any positive order along some part of the discriminant locus; see the discussion at the end of Subsection~\ref{subsectionfirstorder}. 
	Note that in the classical theory of complete non-compact \kah manifolds, the solvability of the complex Monge--Amp\`ere equation typically requires the volume-form error to decay by more than two orders; see Tian--Yau--Hein package \cite{hein2010gravitational}.

	To overcome the lack of sufficient decay, we employ a gluing technique. Roughly speaking, our main idea is to approximate the first-order asymptotic metric $g^{(1)}$ along the discriminant locus by a model metric. This model metric is defined as a product metric on a certain product space; see Subsection~\ref{subsectionstructureonmodel}. We will show that this product space is isomorphic to $\C^{n} \times \R^{2(N - n)}$, where all topological gaps are concentrated in the $\C^{n}$ factor. 
	
	Next, we apply the gluing technique in a conical region of $\C^{n}\times\R^{2(N-n)}$ to deform $g^{(1)}$ into a Calabi--Yau metric; see Subsection~\ref{subsectiongluingcalabiyau}.  After that we add back the error produced when the model metric was used to approximate $g^{(1)}$.  This procedure is referred to as a \emph{surgery}.  We prove that the volume-form error of the metric obtained after surgery decays at least at order $1$ along the locus; see Subsection~\ref{subsectionpde}.  We then improve this decay rate by solving a Laplace-type equation, following Section~2.9 of \cite{li2023syz}, and finally apply the Tian--Yau--Hein package to solve the complex Monge--Amp\`ere equation and construct the desired Calabi--Yau metric; see Subsection~\ref{subsectionpde}.

	When constructing the Taub--NUT metric on $\C^{N+1}$ with $N \geq 4$, at least two surgeries are required. Since the output of one surgery influences the next, the analysis becomes highly entangled. To address this, our proof of Theorem~\ref{maintheorem} is based on mathematical induction. Roughly speaking, in Section~\ref{sectioninduction}, we assume that after a certain surgery, the resulting Gibbons--Hawking coefficients satisfy the desired properties. These assumptions are satisfied by the coefficients corresponding to $g^{(1)}$. We then use these properties to perform the next surgery, and verify that the new coefficients again satisfy the same assumptions. After finitely many such steps, we obtain a proof of Theorem~\ref{maintheorem}. Note that after each surgery, the resulting Gibbons--Hawking coefficients are not defined over the entire base space, yet the induction procedure remains valid.
	
	\subsection{Outline of This Paper}
	This paper is organized as follows.
	
	Section~\ref{sectionprel} reviews the generalized Gibbons--Hawking ansatz following~\cite{zharkov2004limiting}.
	In Subsection~\ref{examplefalt} we fix notation for the discriminant locus adapted to our setting, and in Subsection~\ref{subsectioncompacteqation} we record the distributional equation that compactifies the principal~$\mathbb{T}^{N}$-bundle.
	This equation already appears in~\cite{zharkov2004limiting} in greater generality, here we specialise it to the case required by our main theorem.
	
	Section~\ref{sectionlinearization} constructs the first-order asymptotic metric~$g^{(1)}$ from~\cite{li2023syz} and observes that its volume-form error fails to decay along certain components of the discriminant locus.
	Subsection~\ref{subsectionthegeometry} exploits the symmetries of the locus to simplify later arguments, while Subsection~\ref{subsectionjianjing1} derives the asymptotics of~$g^{(1)}$ nearby. Lemma~\ref{alphajianjin} provides the key estimate for the gluing procedure.
	
	Section~\ref{sectioninduction} formulates the inductive hypotheses (Proposition~\ref{inductionhyp}) needed in the proof of Theorem~\ref{maintheorem}.
	Propositions~\ref{prop:smooth-structure-DK} and~\ref{buchongassumation} in Subsection~\ref{subsectionsmoothstructure} clarify and supplement these hypotheses: the former shows how the GH coefficients induce smooth structures near the discriminant locus, and the latter imposes additional geometric constraints on the resulting GH metrics.
	All assumptions are satisfied by the initial metric~$g^{(1)}$. Example~\ref{biaozhunmuxing} defines the model geometry that serves as the local model for our higher-dimensional Taub--NUT metric along the locus.
	
	Section~\ref{sectionholpoint} produces a \kah structure from the Gibbons--Hawking ansatz under the inductive hypotheses.
	Subsections~\ref{subsectionholomorphicmap}--\ref{subsectiontheimageofthehol} prove that this structure is biholomorphic to the model metric of Example~\ref{biaozhunmuxing}.
	Our metric is, in a suitable sense, defined only ``outside a compact set'', the proof follows~\cite[Section~2.4]{li2023syz}, but the preceding surgeries create extra technicalities.
	We carefully analyse the relation between the complex structure already present along the locus and the one induced by the inductive hypothesis (Subsection~\ref{subsectionlogarithmicgrowth}). Lemma~\ref{zzengzhangxing} is crucial for the subsequent gluing.
	
	Section~\ref{sectionsurgery} carries out the gluing.
	Subsection~\ref{subsctionsurgeryoriginal} extends the \kah metric from Section~\ref{sectionholpoint} smoothly to the whole model space, using toric symmetry and Lemma~\ref{zzengzhangxing}.
	Subsection~\ref{subsectionasymptoticprop} studies the asymptotic behavior of the extended metric and introduces global weighted Sobolev norms needed for Tian--Yau--Hein package.
	Subsection~\ref{subsectiongluingcalabiyau} performs a fibrewise gluing of Calabi--Yau metrics along the conical region of the locus, producing new GH coefficients that satisfy all assumptions of Section~\ref{sectioninduction}; induction then gives Theorem~\ref{maintheorem}.
	Finally, Subsection~\ref{subsectionpde} solves the complex Monge--Amp\`ere equation (Theorem~\ref{MAkejiexing}) with sharp estimates by approximating Green kernels region-by-region, solving the Laplace equation and improving the decay of the volume-form error, following~\cite[Section~2.8]{li2023syz}.
	
	\subsection{Acknowledgments} 
	The author thanks his advisor, Professor Gang Tian, for his continuous guidance and for suggesting the investigation of metric problems in \kah geometry that motivated the present work.

	\section{Preliminaries}\label{sectionprel}
	
	\subsection{Generalised Gibbons--Hawking Ansatz}
	The Gibbons--Hawking ansatz was introduced in 1978 to construct circle-invariant ALF gravitational instantons \cite{gibbons1978gravitational}.  
	Pedersen--Poon reformulated the four-dimensional Gibbons--Hawking monopole equations as a moment-map condition for a tri-hamiltonian circle action \cite{pedersen1991hamiltonian} and, in a companion paper, extended the ansatz to arbitrary $4n$ dimensions with an $m$-torus fibre, obtaining toric hyper-K\"ahler metrics governed by generalized abelian monopole equations \cite{pedersen1988hyper}.  
	Zharkov \cite{zharkov2004limiting} later adapted the framework to local Calabi--Yau geometry, interpreting it as a semi-flat SYZ model and analysing its large-complex-structure limit.
	
	We follow the notation of \cite{zharkov2004limiting}, whose Theorem~2.1 gives the generalized Gibbons--Hawking ansatz in greater generality. Here we record only the special form needed in the present paper.  
	Suppose $M$ is a complex $(N+1)$-dimensional K\"ahler manifold endowed with a non-vanishing holomorphic volume form and admitting a holomorphic, isometric, free $T^{N}$-action.  
	The \emph{generalized Gibbons--Hawking ansatz} expresses the K\"ahler and Calabi--Yau conditions in terms of the $N$ moment-map coordinates and the holomorphic coordinates on the K\"ahler quotient.
	
	Let $\mathfrak{t}^{N}$ denote the Lie algebra of $T^{N}$, and let $\mathfrak{t}^{N}_{\mathbb{Z}}$ be the natural integral lattice in $\mathfrak{t}^{N}$.  
	A choice of basis for $\mathfrak{t}^{N}_{\mathbb{Z}}$ defines linear coordinates $\mu_{i}$ on the dual space $(\mathfrak{t}^{N})^{*}\simeq\mathbb{R}^{N}$.  
	Let $Y$ be either $\mathbb{C}$ or $\mathbb{C}^{*}$, and let $\eta$ denote the standard complex coordinate on $\mathbb{C}$ or the logarithmic coordinate on $\mathbb{C}^{*}$ with period $1$ (so that $e^{2\pi\sqrt{-1}\eta}$ are the standard coordinates on $\mathbb{C}^{*}$).  
	We consider a principal $T^{N}$-bundle $\pi\colon M\to\mathcal{B}^{0}$ over an open subset $\mathcal{B}^{0}$ of $(\mathfrak{t}^{N})^{*}\times Y$ whose first Chern class is an element $c_{1}\in H^{2}(\mathcal{B}^{0},\mathfrak{t}^{N}_{\mathbb{Z}})$.  
	Later we shall partially compactify $M$ into a singular $T^{N}$-bundle.  
	The summation convention is used throughout.

	\begin{theorem}\label{GibbonsHawking}
		(cf.\ Theorem 2.1 in \cite{zharkov2004limiting})
		Let $V_{ij}$ be smooth real symmetric positive-definite matrix functions and let $W$ be a smooth positive real function on $\mathcal{B}^{0}$, locally given by a potential function $\Phi$:
		\begin{equation}\label{GibbonsHawkingpotential}
			V_{ij}=\frac{\partial^{2}\Phi}{\partial\mu_{i}\partial\mu_{j}},\quad
			W=-4\frac{\partial^{2}\Phi}{\partial\eta\partial\bar\eta},\quad
			1\le i,j\le N.
		\end{equation}
		Then the following $\mathfrak{t}^{N}$-valued real $2$-form is closed:
		\begin{align}\label{GibbonsHawkingcurvature}
			F_{j}=\sqrt{-1}\left(\frac{1}{2}\frac{\partial W}{\partial\mu_{j}}\mathrm{d}\eta\wedge \mathrm{d}\bar\eta
			+\frac{\partial V_{ij}}{\partial\eta}\mathrm{d}\mu_{i}\wedge \mathrm{d}\eta
			-\frac{\partial V_{ij}}{\partial\bar\eta}\mathrm{d}\mu_{i}\wedge \mathrm{d}\bar\eta
			\right).
		\end{align}
		Suppose further that $\frac{1}{2\pi}(F_{1},\dots,F_{N})$ lies in the cohomology class $c_{1}\in H^{2}(\mathcal{B}^{0},\mathfrak{t}^{N}_{\mathbb{Z}})$.  
		Then there exists a connection $\vartheta$ on the principal bundle $M\to\mathcal{B}^{0}$ with curvature $\mathrm{d}\vartheta_{i}=F_{i}$ for $i=1,\dots,N$, such that $M$ is a K\"ahler manifold with metric tensor
		\begin{equation}\label{GibbonsHawkingmetrictensor}
			h=V^{-1}_{ij}\zeta_{i}\otimes\bar\zeta_{j}+W\,\mathrm{d}\eta\otimes \mathrm{d}\bar\eta,\quad
			\omega=\mathrm{d}\mu_{j}\wedge\vartheta_{j}+\frac{\sqrt{-1}}{2}W\,\mathrm{d}\eta\wedge \mathrm{d}\bar\eta,
		\end{equation}
		where $\zeta_{j}=V_{ij}\,\mathrm{d}\mu_{i}+\sqrt{-1}\,\vartheta_{j}$ and $\mathrm{d}\eta$ form a basis of $(1,0)$-forms defining an integrable complex structure.  
		There is a nowhere-vanishing holomorphic form on $M$:
		\begin{equation}\label{GibbonsHawkingholomorphicvolume}
			\Omega=\wedge_{j=1}^{N}(-\sqrt{-1}\zeta_{j})\wedge \mathrm{d}\eta.
		\end{equation}
		The Calabi--Yau condition $\omega^{N}=\frac{N!}{2^{N}}\sqrt{-1}^{N^{2}}\Omega\wedge\overline{\Omega}$ is equivalent to
		\begin{equation}\label{GibbonsHawkingCY}
			\det(V_{ij})=\det(W).
		\end{equation}
	\end{theorem}
	
	\begin{remark}
		By the Frobenius theorem, \eqref{GibbonsHawkingpotential} is equivalent to the following system, which we shall refer to as the integrability condition:
		\begin{equation}\label{GibbonsHawkingintegrability}
			\frac{\partial V_{ij}}{\partial\mu_{k}}=\frac{\partial V_{ik}}{\partial\mu_{j}},\quad
			\frac{\partial^{2}W}{\partial\mu_{i}\partial\mu_{j}}
			+4\frac{\partial^{2}V_{ij}}{\partial\eta\partial\bar\eta}=0.
		\end{equation}
	\end{remark}
	
	For the reader's convenience we briefly recall the proof given in \cite{zharkov2004limiting,li2023syz}, it may be regarded as an additional explanation of the theorem.
	
	\begin{proof}
		Using \eqref{GibbonsHawkingcurvature} and the integrability condition \eqref{GibbonsHawkingintegrability},
		\begin{equation}\label{GibbonsHawkingdF}
			\mathrm{d} F_{j}=\frac{\sqrt{-1}}{2}\left(
			\frac{\partial^{2}W}{\partial\mu_{i}\partial\mu_{j}}
			+4\frac{\partial^{2}V_{ij}}{\partial\eta\partial\bar\eta}
			\right)\mathrm{d}\mu_{i}\wedge \mathrm{d}\eta\wedge \mathrm{d}\bar\eta,
		\end{equation}
		so closedness of $F_{j}$ is equivalent to \eqref{GibbonsHawkingintegrability}.  
		Since $\frac{1}{2\pi}F$ represents the appropriate first Chern class, $F$ must be the curvature of a $T^{N}$-connection $\vartheta$.  
		Moreover, if $\theta_{j}$ denote coordinates on the torus fibre such that $X_{j}$ are the Hamiltonian vector fields, then the connection $1$-forms can be written, up to exact forms on $\mathcal{B}^{0}$, in terms of the local potential $\Phi$:
		\begin{equation}
			\vartheta_{j}=\mathrm{d}\theta_{j}+\sqrt{-1}\left(
			\frac{\partial^{2}\Phi}{\partial\mu_{j}\partial\eta}\,\mathrm{d}\eta
			-\frac{\partial^{2}\Phi}{\partial\mu_{j}\partial\bar\eta}\,\mathrm{d}\bar\eta
			\right).
		\end{equation}
		One verifies explicitly that $F_{j}=\mathrm{d}\vartheta_{j}$.  
		Gauge-equivalent choices of the connection define the structures on $M$ up to holomorphic isometry.
		
		Integrability of the complex structure follows from the fact that the differential ideal generated by the $(1,0)$-forms is closed:
		\begin{equation}\label{GibbonsHawkingholomorphicdifferential}
			\mathrm{d}\zeta_{j}=\left(
			\frac{1}{2}\frac{\partial W}{\partial\mu_{j}}\,\mathrm{d}\bar\eta
			-2\frac{\partial V_{ij}}{\partial\eta}\,\mathrm{d}\mu_{i}
			\right)\wedge \mathrm{d}\eta,
		\end{equation}
		where we have used \eqref{GibbonsHawkingintegrability}, \eqref{GibbonsHawkingcurvature} and the definition of $\zeta_{j}$.
		
		Note that $\mathrm{d}\eta=\Omega\!\left(X_{1},\dots,X_{N},\cdot\right)$, accordingly we refer to $\eta$ as the \emph{holomorphic moment coordinate}.  
		The K\"ahler condition $\mathrm{d}\omega=0$ follows from \eqref{GibbonsHawkingintegrability} and \eqref{GibbonsHawkingcurvature}.  
		The Calabi--Yau condition follows from the more general formula
		\[
		\omega^{N}=W\det(V_{ij})^{-1}\frac{N!}{2^{N}}\sqrt{-1}^{N^{2}}\Omega\wedge\overline{\Omega}.
		\]
		This completes the proof.
	\end{proof}
	
	In the remainder of the paper we abbreviate ``Gibbons--Hawking'' to GH for brevity.
	We refer to the functions $V_{ij}$ and $W$ appearing in the theorem as the \emph{GH coefficients}.
	Thus the generalized Gibbons--Hawking ansatz asserts that, given GH coefficients satisfying the positivity assumption, the integrability condition \eqref{GibbonsHawkingintegrability}, and the requirement that the curvature defined by \eqref{GibbonsHawkingcurvature} lies in the cohomology class $2\pi c_{1}$ of the principal bundle, one obtains a GH structure (an integrable complex structure, a holomorphic volume form, and a metric). In particular, if the GH coefficients also satisfy the Calabi--Yau condition \eqref{GibbonsHawkingCY}, then the induced GH metric is Calabi--Yau.
	
	\subsection{The Flat Example}\label{examplefalt}
	In this section we write down the moment map of the flat metric on $\mathbb{C}^{N+1}$ together with its associated GH coefficients.  
	Our goal is to describe the shape of the discriminant locus and to introduce notation that will be used throughout the paper.

	The motivation comes from Theorem~\ref{maintheorem}: the Taub--NUT metric on $\mathbb{C}^{N+1}$ constructed there and the flat metric share the same discriminant locus when viewed as singular principal $\mathbb{T}^{N}$-bundles under their respective moment maps.  
	More generally, for any complete $\mathbb{T}^{N}$-invariant \kah metric on $\mathbb{C}^{N+1}$ endowed with the standard holomorphic volume form, the discriminant locus is a union of codimension-$3$ affine half-spaces. However, not every metric gives rise to a locus that extends to infinity.
	
	The affine space $\C^{N+1}$ with the standard Euclidean metric
	\[
	\omega=\frac{\sqrt{-1}}{2}\sum_{i=0}^{N}\mathrm{d} z_{i}\wedge d\bar z_{i}
	\]
	and holomorphic volume form
	\[
	\Omega=\sqrt{-1}^{N}\mathrm{d} z_{0}\wedge \mathrm{d} z_{1}\wedge\cdots\wedge \mathrm{d} z_{N}
	\]
	admits the diagonal $\mathbb{T}^{N}$-action whose $k$-th circle factor acts by
	\[
	e^{\sqrt{-1}\theta_{k}}\cdot(z_{0},z_{1},\dots,z_{N})
	=\bigl(e^{-\sqrt{-1}\theta_{k}}z_{0},z_{1},\dots,e^{\sqrt{-1}\theta_{k}}z_{k},z_{k+1},\dots,z_{N}\bigr).
	\]
	The corresponding moment coordinates are
	\[
	\mu_{i}=\frac{1}{2}\bigl(|z_{i}|^{2}-|z_{0}|^{2}\bigr),\quad i=1,\dots,N,\qquad
	\eta=z_{0}z_{1}\cdots z_{N}.
	\]
	This defines a $\mathbb{T}^{N}$-bundle away from the singular locus $\bigcup_{i<j}\{z_{i}=z_{j}=0\}$.  
	In moment coordinates we have
	\[
	V^{-1}_{ij}=|z_{0}|^{2}+\delta_{ij}|z_{i}|^{2},\qquad
	W^{-1}=|z_{0}z_{1}\cdots z_{N}|^{2}\Bigl(\frac{1}{|z_{0}|^{2}}+\cdots+\frac{1}{|z_{N}|^{2}}\Bigr),
	\]
	viewed as functions of $\mu_{i}$ and $\eta$.  
	(Note that once the holomorphic volume form is fixed, the coordinate $\eta$ is independent of the metric.)
	
	Next we describe the discriminant locus.
	
	\begin{definition}\label{discriminateloucsdef}
		Define the singular $\mathbb{T}^{N}$-bundle
		\[
		\begin{split}
			\pi\colon\C^{N+1}&\longrightarrow\R^{N}\times\C,\\
			(z_{0},z_{1},\dots,z_{N})&\longmapsto
			\Bigl(
			\tfrac{1}{2}(|z_{1}|^{2}-|z_{0}|^{2}),\,
			\dots,\,
			\tfrac{1}{2}(|z_{N}|^{2}-|z_{0}|^{2}),\,
			z_{0}z_{1}\cdots z_{N}
			\Bigr).
		\end{split}
		\]
		For any subset
		\[
		I=\{i_{1},\dots,i_{k}\}\subseteq\{0,1,\dots,N\}\quad(k\ge 2)
		\]
		define
		\[
		\D_{I}=\pi\bigl(\{z_{i_{1}}=\cdots=z_{i_{k}}=0\ \text{and}\ z_{j}\ne 0\ \text{if}\ j\notin I\}\bigr).
		\]
		The discriminant locus is $\bigcup_{|I|\ge 2}\D_{I}$.
		
		Notice that $\D_{I}$ is an $(N+1-k)$-dimensional Hausdorff open subset of $\R^{N}\times\C$, we set
		\[
		\p\D_{I}=\bigcup_{I\subsetneqq J}\D_{J},
		\]
		its topological boundary.
	\end{definition}
	The set $\D$ is contained in $\{\eta=0\}$, the regions $\D_{I}$ are unbounded.  
	For instance,
	\[
	\D_{01}=\{\mu_{1}=0,\ \mu_{j}>0,\ \eta=0\},\qquad
	\D_{12}=\{\mu_{1}=\mu_{2}=s<0,\ \mu_{j}>s,\ \eta=0\}.
	\]
	
	\subsection{Compactification and Distributional Equation}\label{subsectioncompacteqation}
	In the GH ansatz \eqref{GibbonsHawking} the curvature $F$ defined by \eqref{GibbonsHawkingcurvature} is required to lie in the cohomology class $2\pi c_{1}$.  
	This constraint reveals the topological gap for the existence of a GH structure, namely the non-triviality of the first Chern class.

	Zharkov \cite{zharkov2004limiting} translated this requirement into a distributional equation, a form that is more convenient for our purposes.  
	Although the discussion in \cite{zharkov2004limiting} applies to more general situations, we restrict ourselves to the map $\pi$ in Definition~\ref{discriminateloucsdef} and write $\B$ for the base $\R^{N}\times\C$ and $\B^{0}$ for $\R^{N}\times\C\setminus\D$.  
	Recall that the loci $\D_{ij}$ introduced in Definition~\ref{discriminateloucsdef} are $(N-1)$-dimensional Hausdorff open subsets of $\B$.  
	In the language of currents the equation reads
	
	\begin{equation}\label{cherncon}
		\begin{aligned}
			\frac{\sqrt{-1}}{4\pi}
			\sum_{1\le i,j\le N}
			\Bigl(
			\frac{\partial^{2}W}{\partial\mu_{i}\partial\mu_{j}}
			+4\frac{\partial^{2}V_{ij}}{\partial\eta\partial\bar{\eta}}
			\Bigr)
			\,\mathrm{d}\mu_{i}\wedge\mathrm{d}\eta\wedge\mathrm{d}\bar{\eta}\otimes e_{j}\\[4pt]
			=\sum_{i=1}^{N}\D_{0i}\otimes e_{i}
			+\sum_{1\le i<j\le N}\D_{ij}\otimes(e_{j}-e_{i}).
		\end{aligned}
	\end{equation}
	
	Here $\D_{ij}$ denotes the current associated with the locus $\D_{ij}$, explicitly, in the sense of currents,
	
	\[
	\D_{0i}
	=-\mathcal{H}^{N-1}|_{\D_{0i}}
	\;\mathrm{d}\mu_{i}\wedge\mathrm{d}\operatorname{Re}\eta
	\wedge\mathrm{d}\operatorname{Im}\eta,
	\]
	where the measure $\mathcal{H}^{N-1}|_{\D_{0i}}$ is given by
	\[
	\mathcal{H}^{N-1}|_{\D_{0i}}(f)
	=\idotsint_{\mathbb{R}_{+}^{N-1}}
	f(t_{1},\dots,\widehat{t_{i}},\dots,t_{N},0)\,
	\mathrm{d}t_{1}\cdots\widehat{\mathrm{d}t_{i}}\cdots\mathrm{d}t_{N}.
	\]
	where  $f\in C^{\infty}_{c}(\R^{N}\times\C)$. Similarly,
	\[
	\D_{ij}
	=-\mathcal{H}^{N-1}|_{\D_{ij}}
	\;(\mathrm{d}\mu_{j}-\mathrm{d}\mu_{i})\wedge\mathrm{d}\operatorname{Re}\eta
	\wedge\mathrm{d}\operatorname{Im}\eta
	\]
	with
	\[
	\mathcal{H}^{N-1}|_{\D_{ij}}(f)
	=\idotsint_{\mathbb{R}_{+}^{N-1}}
	f(t_{1},\dots,\widehat{t_{i}},\dots,\widehat{t_{j}},\dots,t_{N},0)\,
	\mathrm{d}t_{1}\cdots\widehat{\mathrm{d}t_{i}}\widehat{\mathrm{d}t_{j}}\cdots\mathrm{d}t_{N}.
	\]
	
	Whenever the GH coefficients $V_{ij}$ and $W$ satisfy equation \eqref{cherncon} in the sense of currents, the principal $\mathbb{T}^{N}$-bundle over $\B^{0}$ can be compactified to a singular $\mathbb{T}^{N}$-bundle over $\B\supset\B^{0}$ by allowing the torus fibres to degenerate.  
	Henceforth we refer to equation \eqref{cherncon} as the \emph{Chern-class condition} for the GH ansatz.

	\section{Linearization Method}\label{sectionlinearization}
	\subsection{First-Order Asymptotic Metric near Infinity}\label{subsectionfirstorder}
	It is evident that constructing a Calabi--Yau metric directly via the GH ansatz is exceedingly difficult.  
	Not only does the Calabi--Yau condition \eqref{GibbonsHawkingCY} impose a nonlinear equation on the GH coefficients, but one must also satisfy \eqref{GibbonsHawkingintegrability} and \eqref{cherncon}.  
	Solving these equations simultaneously for the coefficients~$V_{ij}$ and~$W$
	amounts to resolving a singular, fully nonlinear system---an undertaking that is intractable.
	
	Following the strategy of \cite[Section~2.1]{li2023syz}, we therefore linearise the problem around a constant solution.  
	We perturb the GH coefficients so that the resulting deviations solve the linearised Calabi--Yau equation derived from \eqref{GibbonsHawkingCY}, in the linearisation we may further require that the perturbation satisfy both \eqref{GibbonsHawkingintegrability} and \eqref{cherncon}.
	The perturbed coefficients give rise to a first-order asymptotic metric $g^{(1)}$, at the end of this section we analyse the decay of its volume-form error.
	
	Henceforth we fix $\pi\colon M\to\B\setminus\D$ to be a principal $\mathbb{T}^{N}$-bundle whose topological configuration coincides with that of the bundle described in Example~\ref{examplefalt}. Here $\B=\R^{N}\times\C$ and $\D$ is the discriminant locus introduced in Definition~\ref{discriminateloucsdef}.

	The basic idea is to perturb the constant solution in a way that encodes the underlying topology.
	The constant solution is given by
	\[
	g_{A}=A_{ij}\,\mathrm{d}\mu_{i}\otimes \mathrm{d}\mu_{j}+\detA\,|d\eta|^{2}.
	\]
	Its associated volume form is
	\[
	\mathrm{d}\operatorname{Vol}_A=(\detA)^{3/2}\,\mathrm{d}\mu_{1}\wedge\dots\wedge \mathrm{d}\mu_{N}\wedge \mathrm{d}\operatorname{Re}\eta\wedge d\operatorname{Im}\eta.
	\]
	For a point $\vec{\mu}=(\mu_{1},\dots,\mu_{N},\eta)\in\R^{N}\times\C$ we set
	\[
	|\vec{\mu}|_{A}=\sqrt{A_{ij}\mu_{i}\mu_{j}+\detA|\eta|^{2}}.
	\]
	
	Written in terms of the local potential $\Phi$, the Calabi--Yau condition \eqref{GibbonsHawkingCY} becomes
	\[
	\det\Bigl(\frac{\partial^{2}\Phi}{\partial\mu_{i}\partial\mu_{j}}\Bigr)
	=-4\frac{\partial^{2}\Phi}{\partial\eta\partial\bar{\eta}},
	\]
	whose linearisation at the constant solution is the Laplace equation
	\[
	\Delta_A\phi
	=A^{-1}_{ij}\frac{\partial^{2}\phi}{\partial\mu_{i}\partial\mu_{j}}
	+4(\detA)^{-1}\frac{\partial^{2}\phi}{\partial\eta\partial\bar{\eta}}
	=0.
	\]
	Here $\Delta_{A}$ denotes the Laplacian of the metric $g_{A}$.  
	Away from the discriminant locus, the first-order corrections
	\begin{equation}\label{potentialvij}
		v_{ij}=\frac{\partial^{2}\phi}{\partial\mu_{i}\partial\mu_{j}},
		\qquad
		w=-4\frac{\partial^{2}\phi}{\partial\eta\partial\bar{\eta}}
	\end{equation}
	to $V_{ij}$ and $W$ should therefore be $\Delta_{A}$-harmonic, i.e.
	\begin{equation}\label{CYconvij}
		\Delta_{A}v_{ij}=0,\qquad
		\Delta_{A}w=0,\qquad
		\detA A^{-1}_{ij}v_{ij}=w.
	\end{equation}
	
	To encode the topology we recall the distributional equation \eqref{cherncon} satisfied by $V_{ij}$ and $W$.  
	Linearising yields the current equation
	\begin{equation}\label{chernvij}
		\begin{aligned}
			&\frac{\sqrt{-1}}{4\pi}
			\biggl(\frac{\partial^{2}w}{\partial\mu_{i}\partial\mu_{j}}
			+4\frac{\partial^{2}v_{ij}}{\partial\eta\partial\bar{\eta}}\biggr)
			\mathrm{d}\mu_{i}\wedge \mathrm{d}\eta\wedge \mathrm{d}\bar{\eta}\otimes e_{j}\\[4pt]
			&\qquad=\sum_{i=1}^{N}\D_{0i}\otimes e_{i}
			+\sum_{1\le i<j\le N}\D_{ij}\otimes(e_{j}-e_{i}),
		\end{aligned}
	\end{equation}
	where $v_{ij}$ and $w$ are globally defined, whereas $\phi$ is only a local potential.  
	Regarding $v_{ij}$ and $w$ as the unknowns, the existence of a local function $\phi$ satisfying \eqref{potentialvij} is equivalent to the integrability conditions
	\begin{equation}\label{jiaohuanvij}
		\frac{\partial v_{ij}}{\partial\mu_{k}}
		=\frac{\partial v_{ik}}{\partial\mu_{j}},\qquad
		\frac{\partial^{2}w}{\partial\mu_{i}\partial\mu_{j}}
		=-4\frac{\partial^{2}v_{ij}}{\partial\eta\partial\bar{\eta}}
	\end{equation}
	away from $\D$. These conditions translate directly into the analogous constraints on $V_{ij}$ and $W$ in \eqref{GibbonsHawkingintegrability}.

	\begin{lemma}
		Let $\alpha(N)$ denote the volume of the unit ball in $\R^{N}$. For $1\le i < j \le N$ define  
		\begin{align*}
			\alpha_{0i}
			= \frac{2\pi\sqrt{\detA}}{N(N+2)\alpha(N+2)}
			\idotsint_{\mathbb{R}_+^{N-1}}
			\frac{1}{\lvert\vec{\mu\vphantom{t}}-\vec{t}^{\mkern2mu\scriptscriptstyle(i)}\rvert_{A}^{N}}\,
			{\rm d}t_{1}\!\cdots\widehat{{\rm d}t_{i}}\!\cdots{\rm d}t_{N},
		\end{align*}
		where
		\begin{align*}
			\vec{t}^{\mkern2mu\scriptscriptstyle(i)}=(t_{1},\dots,t_{i-1},\overset{\raisebox{-1pt}{$\scriptstyle i$}}{0},t_{i+1},\dots,t_{N},0)\in\mathbb{R}^{N}\times\mathbb{C},
		\end{align*}
		and
		\begin{align*}
			\alpha_{ij}
			= \frac{2\pi\sqrt{\detA}}{N(N+2)\alpha(N+2)}
			\idotsint_{\mathbb{R}_+^{N-1}}
			\frac{1}{\lvert\vec{\mu\vphantom{t}}-\vec{t}^{\mkern2mu\scriptscriptstyle(ij)}+s\bm{\vec{1}}\rvert_{A}^{N}}\,
			{\rm d}s\,
			{\rm d}t_{1}\!\cdots\!\widehat{{\rm d}t_{i}}\widehat{{\rm d}t_{j}}\!\cdots{\rm d}t_{N},
		\end{align*}
		with
		$$\vec{t}^{\mkern2mu\scriptscriptstyle(ij)}=
		(t_{1},\dots,t_{i-1},\overset{\raisebox{-1pt}{$\scriptstyle i$}}{0},t_{i+1},\dots,
		t_{j-1},\overset{\raisebox{-1pt}{$\scriptstyle j$}}{0},t_{j+1},\dots,t_{N},0)
		\in\mathbb{R}^{N}\times\mathbb{C},$$
		and
		\begin{align*}
			\bm{\vec{1}}=(1,\dots,1,0)\in \R^{N}\times \C. 
		\end{align*}
		Then, in the sense of distributions,
		\begin{align*}
			\Delta_A\alpha_{0i}\,\mathrm{d}\!\operatorname{Vol}_{A}
			&=-2\pi\sqrt{\det A}\,\mathcal{H}^{N-1}|_{\D_{0i}},\\[2pt]
			\Delta_A\alpha_{ij}\,\mathrm{d}\!\operatorname{Vol}_{A}
			&=-2\pi\sqrt{\det A}\,\mathcal{H}^{N-1}|_{\D_{ij}}.
		\end{align*}
		Setting formally $\alpha_{0i}=\alpha_{i0}$, $\alpha_{ij}=\alpha_{ji}$ and $\alpha_{ii}=0$, one further has
		\begin{equation}\label{alphaijjiaohuan}
			\begin{aligned}
				\frac{\p \alpha_{ij}}{\p\mu_{k}}=\frac{\p \alpha_{ik}}{\p\mu_{j}},\quad
				\frac{\p \alpha_{0i}}{\p\mu_{j}}=\frac{\p \alpha_{0j}}{\p\mu_{i}}=-\sum_{t=1}^{N} \frac{\p  \alpha_{ij}}{\p\mu_{t}}
			\end{aligned}
		\end{equation}
		for any distinct indices $1\le i,j,k\le N$.
	\end{lemma}
	\begin{proof}
		The Green function on $\mathbb{R}^{N+2}$ is given by
		\[
		G(x,y)=-\frac{1}{N(N+2)\alpha(N+2)}\frac{1}{|x-y|_{A}^{N}},
		\]
		so the coefficients $\alpha_{ij}$ can be obtained explicitly by integration.
		
		To verify the commutativity relations, start with
		\[
		\alpha_{01}=\int_{-\infty}^{\mu_2}\cdots\int_{-\infty}^{\mu_n}
		\frac{1}{|(\mu_{1},t_{2},\dots,t_{n},\eta)|_{A}^{N}}\,
		\mathrm{d}t_{2}\cdots\mathrm{d}t_{n}.
		\]
		Differentiating with respect to $\mu_{2}$ yields
		\[
		\frac{\partial\alpha_{01}}{\partial\mu_{2}}
		=\int_{-\infty}^{\mu_3}\cdots\int_{-\infty}^{\mu_n}
		\frac{1}{|(\mu_{1},\mu_{2},\dots,\mu_{1},\eta)|_{A}^{N}}\,
		\mathrm{d}t_{3}\cdots\mathrm{d}t_{n}
		=\frac{\partial\alpha_{02}}{\partial\mu_{1}}.
		\]
		Similar arguments show that
		\[
		\frac{\partial\alpha_{0i}}{\partial\mu_{j}}
		=\frac{\partial\alpha_{0j}}{\partial\mu_{i}}
		\quad\text{and}\quad
		\frac{\partial\alpha_{ij}}{\partial\mu_{k}}
		=\frac{\partial\alpha_{ik}}{\partial\mu_{j}}
		\]
		for any distinct indices $1\le i,j,k\le N$.
		
		Next consider
		\[
		\alpha_{12}
		=\int_{\mu_{1}}^{+\infty}\!\int_{\mu_1-s}^{+\infty}\!\cdots\!\int_{\mu_n-s}^{+\infty}
		\frac{1}{|(s,\mu_{2}-\mu_{1}+s,\mu_{3}-t_{3},\dots,\mu_{n}-t_{n},\eta)|_{A}^{N}}\,
		\mathrm{d}s\,\mathrm{d}t_{3}\cdots\mathrm{d}t_{n}.
		\]
		Differentiating with respect to $\mu_{1}$ gives
		\begin{align*}
			\frac{\partial\alpha_{12}}{\partial\mu_{1}}
			&=-\frac{\partial\alpha_{12}}{\partial\mu_{2}}+\idotsint\limits_{\mathbb{R}_+^{N-2}}
			\frac{1}{\,\bigl|\vec{\mu}-\vec{t}^{\,(12)}\bigr|_{A}^{N}\,}\,
			\mathrm{d}t_{3}\cdots\mathrm{d}t_{N}\\[4pt]
			&\quad -\sum_{i=3}^{N}\;
			\idotsint\limits_{\mathbb{R}_+^{N-2}}
			\frac{1}{\,\bigl|\vec{\mu}-\vec{t}^{\,(12i)}+s\mathbf{1}\bigr|_{A}^{N}\,}\,
			\mathrm{d}s\,\mathrm{d}t_{3}\cdots\widehat{\mathrm{d}t_{i}}\cdots\mathrm{d}t_{N}\\[4pt]
			&=-\frac{\partial\alpha_{12}}{\partial\mu_{2}}-\frac{\partial\alpha_{01}}{\partial\mu_{2}}
			-\sum_{i=3}^{N}\frac{\partial\alpha_{12}}{\partial\mu_{i}}
			.
		\end{align*}
		An analogous computation establishes
		\[
		\frac{\partial\alpha_{0i}}{\partial\mu_{j}}
		=\frac{\partial\alpha_{0j}}{\partial\mu_{i}}
		=-\sum_{t=1}^{N}\frac{\partial\alpha_{ij}}{\partial\mu_{t}}
		\quad\text{for }i\ne j.
		\]
		This completes the proof.
	\end{proof}
	
	Building on the preceding lemma we readily obtain the following proposition, whose integrability condition is supplied by equation \eqref{alphaijjiaohuan}.
	\begin{proposition}
		The quantities defined by
		\[
		v_{ii}=\alpha_{0i}+\sum_{k\ne i}\alpha_{ik},\qquad
		v_{ij}=-\alpha_{ij}\ \ (i\ne j),\qquad
		w=\detA\,A^{-1}_{ij}v_{ij}
		\]
		solve the integrability condition \eqref{jiaohuanvij} and the harmonicity condition \eqref{CYconvij} away from $\D$, and satisfy the distributional equation \eqref{chernvij} globally.
	\end{proposition}
	We now introduce the K\"ahler ansatz $(g^{(1)},\omega^{(1)},J^{(1)},\Omega^{(1)})$ via the generalized Gibbons--Hawking construction by setting
	\[
	V_{ij}^{(1)}=A_{ij}+v_{ij},\qquad W^{(1)}=\det A+w.
	\]
	The superscript ``(1)'' indicates the first-order asymptotic metric and signifies that no surgery has yet been performed.
	Positive-definiteness of the matrix $\bigl(V_{ij}^{(1)}\bigr)$ is ensured by the pointwise inequality $\alpha_{ij}>0$, which can be read off from the explicit integral formulae.
	A caveat is that smoothness of the metric across the discriminant locus $\D$ is not a priori guaranteed, this issue will be addressed in the next stage of the construction.
	
	The quality of the approximation to the Calabi--Yau condition is measured by the volume-form error
	\[
	E^{(1)}=\frac{W^{(1)}}{\det\!\bigl(V_{ij}^{(1)}\bigr)}-1.
	\]
	Expanding the determinant we obtain
	\[
	\begin{split}
		E^{(1)}
		&=\frac{\det A+w}{\det A+w+\sum_{k=2}^{N}\sigma_{k}\!\bigl(A^{-1}(v_{ij})\bigr)}-1\\[4mm]
		&=-\frac{(\det A)\sum_{k=2}^{N}\sigma_{k}\!\bigl(A^{-1}(v_{ij})\bigr)}
		{\det A+w+(\det A)\sum_{k=2}^{N}\sigma_{k}\!\bigl(A^{-1}(v_{ij})\bigr)},
	\end{split}
	\]
	where $\sigma_{k}\!\bigl(A^{-1}(v_{ij})\bigr)$ denotes the elementary symmetric polynomial of degree $k$ in the eigenvalues of $A^{-1}(v_{ij})$. The term $\sigma_k$ can be regarded as a degree-$k$ polynomial in the variables $\alpha_{ij}$, while $w$ can be viewed as a linear combination of the $\alpha_{ij}$. Consequently, along a stratum $\D_{I}$ with $|I|\geq 3$ the numerator does not decay, so $E^{(1)}$ remains $O(1)$.
	This is a new obstruction that does not appear in the three-dimensional case treated in \cite{li2023syz}.
	
	In summary, away from the discriminant locus we have
	\[
	E^{(1)}=O\!\left(\frac{1}{|\vec\mu|_{A}^{2}}\right)
	\qquad\text{as }|\vec\mu|_{A}\to\infty,
	\]
	while along the simple-intersection region of $\D$ (codimension-one strata) the decay drops to
	\[
	E^{(1)}=O\!\left(\frac{1}{|\vec\mu|_{A}}\right),
	\]
	and along the deep-intersection region (i.e.\ intersection multiplicity at least three) the error is  
	\[
	E^{(1)}=O\!\left(1\right).
	\]
	
	Standard analytic machineries \cite{hein2010gravitational} that produce Ricci-flat metrics from approximate solutions require a volume-form error decay strictly faster than quadratic.
	Our ansatz fails this requirement, so the error must be corrected before those tools can be applied.
	In \cite{li2023syz} Yang Li overcame the analogous difficulty on $\mathbb C^{3}$ by exploiting the Green-function behaviour of $g^{(1)}$ in different regions.
	Notice, however, that outside the origin the discriminant locus in $\mathbb C^{3}$ has at most simple intersection; the problem can then be reduced to solving a Laplace-type equation after a Ricci-curvature correction.
	In higher dimensions the volume-form error may not decay along certain directions, so the three-dimensional strategy does not carry over.
	Instead, we shall perform a surgery along the deep-intersection strata $\D_{I}$ with $|I|\geq 3$ to force the error to decay at the required rate.
	
	\subsection{The Geometry of the Discriminant Locus}\label{subsectionthegeometry}
	In the sequel, whenever we analyse properties near \(\D_{I}\) we routinely reduce to the case \(I=\{0,1,\dots,n\}\).  
	This reduction is justified by the natural symmetry of the discriminant locus of the singular \(\mathbb{T}^{n}\)-bundle \(M\):
	
	Fix an integral basis \(e_{1},\dots,e_{N}\) of the Lie algebra \(\mathfrak t\) of \(\mathbb{T}^{n}\) and let \(\D\) be the corresponding discriminant locus.  
	Given \(I=\{i_{1},\dots,i_{n}\}\), two cases occur.
	
	\begin{enumerate}
		\item If \(0\in I\), we reorder the basis so that \(I=\{0,1,\dots,N\}\).
		
		\item If \(0\notin I\), we replace the basis by the new integral basis
		\[
		-e_{i_{1}},\qquad e_{k}-e_{i_{1}}\quad(k\neq i_{1}).
		\]
	\end{enumerate}
	
	The second change induces an affine automorphism of the base \(\B\), equivalently, we obtain a new projection \(\pi'\colon M\to \B'\) under which \(\D_{I}\) is mapped to \(\D_{J}\) with \(J=\{0,i_{2},\dots,i_{n}\}\).  
	This affine transformation alters the metric \(g_{A}\) on \(\B\), that is, it changes the positive-definite matrix \(A\) but keeps its eigenvalues within a uniform compact set.  
	Moreover, our ansatz metric \(g^{(1)}\) is compatible with this change of \(A\).

	To streamline later discussions we fix a notational convention.  
	Let 
	$$
	I=\{i_{1},\dots,i_{n}\}\subseteq\{0,1,\dots,N\}
	$$ 
	be an index set.  
	Whenever a subscript \(I\) is used, if \(0\in I\) and the object labelled by \(I\) is defined only on \(\{1,\dots,N\}\), the index \(0\) is automatically dropped.  
	For example, if \(I=\{0,1,2\}\), then \(\mu_{I}\) denotes the column vector \((\mu_{1},\mu_{2})^{\top}\).

	Below we assume \(0\in I\), so that \(I=\{0,i_{1},\dots,i_{n}\}\); the case \(0\notin I\) is completely analogous.  
	We write  
	\[
	\mu_{I}=(\mu_{i_{1}},\dots,\mu_{i_{n}})^{\top},\qquad 
	\vec{\mu}_{I}=(\mu_{i_{1}},\dots,\mu_{i_{n}},\eta)^{\top}
	\]  
	for the corresponding column vectors.
	
	Define  
	\[
	A_{I}=(A_{i_{t}i_{s}})_{1\le t,s\le n}
	\]  
	to be the principal sub-matrix of \(A\) whose rows and columns are indexed by \(I\), and set  
	\[
	A^{-1}_{I}:=(A_{I})^{-1}.
	\]  
	Notice that in general \((A^{-1})_{I}\neq(A_{I})^{-1}\).
	
	If \(J=\{j_{1},\dots,j_{k}\}\) is another index set disjoint from \(I\), we define  
	\[
	A_{IJ}=(A_{i_{t}j_{s}})_{\substack{1\le t\le n\\ 1\le s\le k}}
	\]  
	so that \(A_{IJ}^{\top}=A_{JI}\).
	
	For every index set \(I\) let \(I^{\prime}:=\{0,1,\dots,N\}\setminus I\).  Introduce the matrix  
	\[
	G_{I}:=A_{I}-A_{II^{\prime}}A_{I^{\prime}}^{-1}A_{I^{\prime}I}.
	\]  
	The eigenvalues of \(G_{I}\) lie in the interval  
	\[
	\Bigl[\lambda\Bigl(\frac{\lambda}{\Lambda}\Bigr)^{N-1},\;\Lambda\Bigr].
	\]

	Let $I=\{0,1,\dots,n\}$. For any point $p\in\mathcal B$ we write $p_{I}=\operatorname{pr}_{I}(p)$ for its projection onto $\D_{I}$, i.e.
	\begin{equation*}
		\operatorname{dist}_{A}(p,p_{I})=\min_{q\in\D_{I}}\operatorname{dist}_{A}(p,q).
	\end{equation*}
	Observe the identity
	\begin{align*}
		\begin{aligned}
			\mu^{\mathstrut\top}\!A\mu
			&=\begin{pmatrix}\mu_{I}^{\top}&\mu_{I'}^{\top}\end{pmatrix}
			\begin{pmatrix}A_{I}&A_{II'}\\A_{I'I}&A_{I'}\end{pmatrix}
			\begin{pmatrix}\mu_{I}\\\mu_{I'}\end{pmatrix}\\[4pt]
			&=\mu_{I}^{\top}G_{I}\mu_{I}
			+(\mu_{I'}+A_{I'}^{-1}A_{I'I}\mu_{I})^{\top}\!A_{I'}(\mu_{I'}+A_{I'}^{-1}A_{I'I}\mu_{I}).
		\end{aligned}
	\end{align*}
	Since $\D_{I}$ is a Hausdorff open subset, $p_{I}$ lies in the interior of $\D_{I}$ if and only if every entry of
	\begin{equation*}
		\nu_{I,I'}=(\nu_{I,i})_{n+1\le i\le N}=\mu_{I'}+A_{I'}^{-1}A_{I'I}\mu_{I}
	\end{equation*}
	is strictly positive. In that case $p_{I}=(0,\dots,0,\nu_{I}^{\top},0)$ and
	\begin{equation*}
		\operatorname{dist}_{A}^{2}(p,\D_{I})=\mu_{I}^{\top}G_{I}\mu_{I}+\det A\,|\eta|^{2}.
	\end{equation*}
	In general, $p_{I}$ may lie on the boundary of $\D_{I}$, then the distance from $p$ to $\D_{I}$ equals the distance from $p$ to $\D_{I\cup J}$ for some index set $J$. Henceforth we assume $p_{I}\in\D_{I}$ and investigate the relation between the projections $\operatorname{pr}_{I\cup J}(p_{I})$ and $\operatorname{pr}_{I\cup J}(p)$.
	
	Let $J$ be another index set disjoint from $I$, and set $K=\{0,1,\dots,N\}\setminus(I\cup J)$. Then
	\begin{align*}
		\begin{aligned}
			\begin{pmatrix}\nu_{I,J}\\\nu_{I,K}\end{pmatrix}
			&=\begin{pmatrix}\mu_{J}\\\mu_{K}\end{pmatrix}
			+\begin{pmatrix}A_{J}&A_{JK}\\A_{KJ}&A_{K}\end{pmatrix}^{-1}
			\begin{pmatrix}A_{JI}\\A_{KI}\end{pmatrix}\mu_{I}\\[4pt]
			&=\begin{pmatrix}
				\mu_{J}+s_{J}^{-1}s_{JI}\mu_{I}\\[4pt]
				\mu_{K}+A_{K}^{-1}\bigl(A_{KI}-A_{KJ}s_{J}^{-1}s_{JI}\bigr)\mu_{I}
			\end{pmatrix},
		\end{aligned}
	\end{align*}
	where
	\begin{equation*}
		s_{J}=A_{J}-A_{JK}A_{K}^{-1}A_{KJ},\qquad
		s_{JI}=A_{JI}-A_{JK}A_{K}^{-1}A_{KI}.
	\end{equation*}
	Consequently,
	\begin{equation*}
		\nu_{I,K}+A_{K}^{-1}A_{KJ}\nu_{I,J}=\mu_{K}+A_{K}^{-1}A_{KI}\mu_{I}=\nu_{I\cup J,K},
	\end{equation*}
	which implies $\operatorname{pr}_{I\cup J}(p)=\operatorname{pr}_{I\cup J}(p_{I})$. Thus, if both projections lie in the interior of $\D_{I\cup J}$, we obtain
	\begin{equation}\label{affinejulihengdengshi}
		\operatorname{dist}_{A}^{2}(p,\D_{I\cup J})
		=\operatorname{dist}_{A}^{2}(p,\D_{I})
		+\operatorname{dist}_{A}^{2}(p_{I},\D_{I\cup J}).
	\end{equation}
	If $\operatorname{pr}_{I\cup J}(p_{I})$ is not in the interior of $\D_{I\cup J}$, then $\operatorname{pr}_{I\cup J}(p)$ is not either. Hence there exists an index set $L\supset J$ such that $\operatorname{pr}_{I\cup J}(p_{I})\in\D_{I\cup L}$, and then $\operatorname{pr}_{I\cup L}(p)\in\D_{I\cup L}$. Therefore
	\begin{equation*}
		\operatorname{dist}_{A}^{2}(p,\D_{I})
		+\operatorname{dist}_{A}^{2}(p_{I},\D_{I\cup J})
		=\operatorname{dist}_{A}^{2}(p,\D_{I\cup L})
		\ge\operatorname{dist}_{A}^{2}(p,\D_{I\cup J}).
	\end{equation*}
	Conversely, there also exists an index set $L'\supset J$ such that $\operatorname{pr}_{I\cup J}(p)\in\D_{I\cup L'}$, and thus
	\begin{align*}
		\operatorname{dist}_{A}^{2}(p,\D_{I\cup L'})
		&=\operatorname{dist}_{A}^{2}(p,\D_{I\cup J})
		=\operatorname{dist}_{A}^{2}(p,\D_{I})
		+\operatorname{dist}_{A}^{2}(p_{I},\D_{I\cup L'})\\[2pt]
		&\ge\operatorname{dist}_{A}^{2}(p,\D_{I})
		+\operatorname{dist}_{A}^{2}(p_{I},\D_{I\cup J}).
	\end{align*}
	Hence identity~\eqref{affinejulihengdengshi} holds whenever $p_{I}\in\D_{I}$. Moreover, when $p_{I}\in\partial\D_{I}$ we have
	\begin{equation*}
		\operatorname{dist}_{A}^{2}(p,\D_{I\cup J})
		\le\operatorname{dist}_{A}^{2}(p,\D_{I})
		+\operatorname{dist}_{A}^{2}(p_{I},\D_{I\cup J}).
	\end{equation*}

	\subsection{Asymptotic Properties of \texorpdfstring{$g^{(1)}$}{g(1)}}\label{subsectionjianjing1}
	
	In this section we analyze the asymptotic behavior of the ansatz metric $g^{(1)}$ near the discriminant locus, this is equivalent to studying the functions $\alpha_{ij}$.  
	First observe that for $i,j\in I$ the function $\alpha_{ij}$ becomes singular along $\D_{I}$ and is a smooth harmonic function away from $\D_{ij}$.  
	Moreover, there exists a uniform constant $C$ such that
	\begin{equation*}
		0<\alpha_{ij}<\frac{C}{\operatorname{dist}_{A}(\,\cdot\,,\D_{ij})}.
	\end{equation*}
	For $2\le|I|\le N$ we introduce the following region in $\B$:
	\begin{equation}\label{defineBI}
		\B_{I}:=\bigl\{p\in\B\bigm| C_{0}\operatorname{dist}_{A}(p,\D_{I})<\operatorname{dist}_{A}(p,\partial\D_{I})\bigr\},
	\end{equation}
	where $C_{0}\gg 1$ are uniform constants. 
	\begin{align*}
		\B_{a}=\bigl\{p\in\B \bigm| 2 C_{0}^{N-1}\operatorname{dist}_{A}(p,\D)>\operatorname{dist}_{A}(p,O)\bigr\},
	\end{align*}
	It is not hard to verify that
	\[
	\B=\B_{a}\cup\bigcup_{2\le|I|\le N}\B_{I}.
	\]
	When $i,j\in I$ we decompose $\alpha_{ij}$ on $\B_{I}$ into a singular leading term and a smooth decaying term, this decomposition will be the cornerstone of the surgery construction in later sections.
	
	\begin{lemma}\label{alphajianjin}
		Let $2\le|I|\le N+1$ and $i,j\in I$.  Then we have the following decomposition:
		\begin{equation*}
			\alpha_{ij}=\alpha_{I,ij}+\beta_{I,ij},
		\end{equation*}
		where $\alpha_{I,ij}$ depends only on $\mu_{I}$ and $\eta$, while $\beta_{I,ij}$ is a smooth $\Delta_{A}$-harmonic function in $\B_{I}$ satisfying
		\begin{equation*}
			|\beta_{I,ij}|\le\frac{C}{\operatorname{dist}(\,\cdot\,,\partial\D_{I})}.
		\end{equation*}
		If $I\subset J$ and $|J|=|I|+1$, then in $\B_{I}$ we have
		\begin{equation*}
			|\beta_{I,ij}-\beta_{J,ij}|=|\alpha_{I,ij}-\alpha_{J,ij}|\le\frac{C}{\operatorname{dist}(\,\cdot\,,\D_{J})}.
		\end{equation*}
		More generally, if $I\subset J$, then in $\B_{I}$ we have
		\begin{equation*}
			|\beta_{I,ij}-\beta_{J,ij}|=|\alpha_{I,ij}-\alpha_{J,ij}|\le\max\limits_{\substack{I\subset K\subset J\\|K|=|J|+1}}\frac{C}{\operatorname{dist}(\,\cdot\,,\D_{K})}.
		\end{equation*}
		The constant $C$ is uniform.  For notational convenience we set $\alpha_{I,ij}=0$ and $\beta_{I,ij}=\alpha_{ij}$ whenever $i\notin I$ or $j\notin I$.
	\end{lemma}
	
	\begin{proof}
		We set $I=\{0,1,\dots,n\}$ and $i=0$, $j=1$.  Then
		\begin{equation*}
			\alpha_{01}
			=\frac{2\pi\sqrt{\detA}}{N(N+2)\alpha(N+2)}
			\idotsint_{\mathbb{R}_{+}^{N-1}}
			\frac{1}{\bigl\lvert(\mu_{1},\mu_{i}-t_{i},\eta)\bigr\rvert_{A}^{N}}\,
			\mathrm{d}t_{2}\dotsm\mathrm{d}t_{N}.
		\end{equation*}
		Define $\alpha_{I,01}$ by integrating out the last $N-n$ variables:
		\begin{equation*}
			\alpha_{I,01}
			=\frac{2\pi\sqrt{\det A}}{N(N+2)\alpha(N+2)}
			\idotsint_{\mathbb{R}_{+}^{n-1}\times\mathbb{R}^{N-n}}
			\frac{1}{\bigl\lvert(\mu_{1},\mu_{i}-t_{i},\eta)\bigr\rvert_{A}^{N}}\,
			\mathrm{d}t_{2}\dotsm\mathrm{d}t_{N}.
		\end{equation*}
		For a positive-definite matrix $M$ of order $k$ and any $R>0$ we have
		\begin{equation*}
			\idotsint_{\mathbb{R}^{k}}
			\frac{1}{\bigl(x^{\top}Mx+R^{2}\bigr)^{N/2}}\,
			\mathrm{d}x_{1}\dotsm\mathrm{d}x_{k}
			=\frac{1}{\sqrt{\det M}\,R^{N-k}}\,
			\frac{\alpha(N-2)}{\alpha(N-k-2)}.
		\end{equation*}
		Since
		\begin{equation*}
			\frac{\alpha(N-2)}{\alpha(n-2)}
			=\frac{N(N+2)\alpha(N+2)}{n(n+2)\alpha(n+2)},
		\end{equation*}
		we obtain
		\begin{equation*}
			\alpha_{I,01}
			=\frac{2\pi\sqrt{\det G_{I}}}{n(n+2)\alpha(n+2)}
			\idotsint_{\mathbb{R}_{+}^{n-1}}
			\frac{1}{\bigl\lvert(\mu_{1},\mu_{I\setminus\{1\}}-t_{I\setminus\{1\}},\eta)\bigr\rvert_{G_{I}}^{n}}\,
			\mathrm{d}t_{2}\dotsm\mathrm{d}t_{n}.
		\end{equation*}
		In particular, when $n=1$,
		\begin{equation*}
			\alpha_{\{0,1\},01}
			=\frac{1}{2\sqrt{\mu_{1}^{2}+\det(A_{\{2,\dots,N\}})|\eta|^{2}}}.
		\end{equation*}
		Clearly $\alpha_{I,01}$ depends only on $\mu_{I},\eta$ and satisfies
		\begin{equation*}
			\Delta_{A}\alpha_{I,01}\,\mathrm{d}\!\operatorname{Vol}_{A}
			=-2\pi\sqrt{\det A}\,\mathcal{H}^{N-1}\big|_{S},
		\end{equation*}
		where $S$ is the $(N-3)$-dimensional Hausdorff set
		\begin{equation*}
			S=\bigl\{p\in\B:\mu_{1}=\eta=0,\; \mu_{i}>0\ (2\le i\le n)\bigr\}.
		\end{equation*}
		Set
		\begin{equation*}
			\beta_{I,01}=\alpha_{01}-\alpha_{I,01}.
		\end{equation*}
		Then $\beta_{I,01}$ satisfies
		\begin{equation*}
			\Delta_{A}\beta_{I,01}\,\mathrm{d}\!\operatorname{Vol}_{A}
			=2\pi\sqrt{\det A}\,\mathcal{H}^{N-1}\big|_{S'},
		\end{equation*}
		with $S'=\{\mu_{1}=\eta=0\}\setminus\overline{S}$ also of Hausdorff dimension $N-3$; in particular $\beta_{I,01}$ is smooth on $\B_{I}$.
		
		Next let $J=\{0,1,\dots,n+1\}$.  One computes
		\begin{equation*}
			\Delta_{A}(\alpha_{I,01}-\alpha_{J,01})\,\mathrm{d}\!\operatorname{Vol}_{A}
			=-2\pi\sqrt{\det A}\,\mathcal{H}^{N-1}\big|_{Y},
		\end{equation*}
		where
		\begin{equation*}
			Y=\bigl\{p\in\B:\mu_{1}=\eta=0,\; \mu_{i}>0\ (2\le i\le n),\ \mu_{n+1}<0\bigr\}.
		\end{equation*}
		Hence
		\begin{equation*}
			|\alpha_{I,01}-\alpha_{J,01}|\le\frac{C}{\operatorname{dist}_{A}(\,\cdot\,,Y)}.
		\end{equation*}
		We now verify that for $p\in\B_{I}$,
		\begin{equation}\label{dist-Y-DJ}
			\operatorname{dist}_{A}(p,Y)\ge C\operatorname{dist}_{A}(p,\D_{J}).
		\end{equation}
		Since the projection $p_{I}$ of $p$ onto $\D_{I}$ lies in $\D_{I}$, we have
		\begin{equation*}
			\operatorname{dist}_{A}(p,Y)
			\ge\operatorname{dist}_{A}(p_{I},Y)-\operatorname{dist}_{A}(p,p_{I}).
		\end{equation*}
		Write
		\begin{equation*}
			p_{I}=(0,\dots,0,w_{n+1},\dots,w_{N},0),\qquad w_{i}>0.
		\end{equation*}
		For any $q\in Y$ of the form
		\begin{equation*}
			q=(0,t_{2},\dots,t_{n},-t_{n+1},s_{n+2},\dots,s_{N},0),\qquad t_{i}>0,\ s_{i}\in\mathbb{R},
		\end{equation*}
		we obtain
		\begin{equation*}
			\operatorname{dist}_{A}^{2}(p_{I},q)
			\ge\lambda\Bigl(\sum_{i=2}^{n}t_{i}^{2}+(w_{n+1}+t_{n+1})^{2}
			+\sum_{j=n+2}^{N}(w_{j}-s_{j})^{2}\Bigr),
		\end{equation*}
		so that
		\begin{equation*}
			\operatorname{dist}_{A}^{2}(p_{I},Y)
			\ge\lambda w_{n+1}^{2}
			\ge\frac{\lambda}{\Lambda}\operatorname{dist}_{A}^{2}(p_{I},\D_{J}).
		\end{equation*}
		By the definition of $\B_{I}$ and identity~\eqref{affinejulihengdengshi},
		\begin{equation*}
			\operatorname{dist}_{A}^{2}(p_{I},\D_{J})
			=\operatorname{dist}_{A}^{2}(p,\D_{J})-\operatorname{dist}_{A}^{2}(p,p_{I})
			\ge\Bigl(1-\frac{1}{C_{0}^{2}}\Bigr)\operatorname{dist}_{A}^{2}(p,\D_{J}).
		\end{equation*}
		Thus
		\begin{equation*}
			\operatorname{dist}_{A}(p,Y)
			\ge\frac{1}{C_{0}}\Bigl(\sqrt{\frac{\lambda}{\Lambda}}\sqrt{C_{0}^{2}-1}-1\Bigr)
			\operatorname{dist}_{A}(p,\D_{J}).
		\end{equation*}
		Choose \(C_{0}\) sufficiently large so that the right-hand side of the above expression is positive, hence
		\begin{equation*}
			|\alpha_{I,01}-\alpha_{J,01}|
			\le\frac{C}{\operatorname{dist}(\,\cdot\,,\D_{J})}.
		\end{equation*}
		The general case $I\subset J$ follows similarly.  
		Finally, since $\beta_{I,01}=\beta_{I,01}-\beta_{\{0,1,\dots,N\},01}$, we conclude
		\begin{equation*}
			|\beta_{I,01}|
			\le\frac{C}{\operatorname{dist}(\,\cdot\,,\partial\D_{I})}.\\[-1.5em]
		\end{equation*}
	\end{proof}
	
	Let $I=\{0,1,\dots,n\}$.  
	On the subspace $\mathbb{R}^{n}_{\mu_{I}}\times\mathbb{C}_{\eta}$ we introduce the metric
	\begin{equation*}
		g_{G_{I}}=\sum_{i,j\in I}G_{I,ij}\,\mathrm{d}\mu_{i}\otimes \mathrm{d}\mu_{j}
		+\detA |\mathrm{d}\eta|^{2}
	\end{equation*}
	and the volume form
	\begin{equation*}
		\mathrm{d}\operatorname{Vol}_{G_{I}}
		=\bigl(\det G_{I}\bigr)^{3/2}\mathrm{d}\mu_{1}\wedge\dots\wedge \mathrm{d}\mu_{n}.
	\end{equation*}
	Denote by $\Delta_{G_{I}}$ the Laplace--Beltrami operator associated with $g_{G_{I}}$:
	\begin{equation*}
		\Delta_{G_{I}}u
		=\sum_{i,j\in I}G_{I,{ij}}^{-1}\frac{\partial^{2}u}{\partial\mu_{i}\partial\mu_{j}}
		+4 (\det A)^{-1}\frac{\partial^{2}u}{\partial\eta\partial\bar{\eta}}
		\operatorname{Re}\mathrm{d}\eta\wedge\operatorname{Im}\mathrm{d}\eta,
	\end{equation*}
	If a function $u$ on $\B$ depends only on $\mu_{I},\eta$, then
	\begin{equation*}
		\Delta_{A}u=\Delta_{G_{I}}u.
	\end{equation*}
	
	Regarding $\D_{ij}\cap(\mathbb{R}^{n}_{\mu_{I}}\times\mathbb{C}_{\eta})$ as an $(n-3)$-dimensional Hausdorff open subset of $\mathbb{R}^{n}_{\mu_{I}}\times\mathbb{C}_{\eta}$, we have, in the sense of distributions,
	\begin{equation*}
		\Delta_{G_{I}}\alpha_{I,ij}\,\mathrm{d}\operatorname{Vol}_{G_{I}}
		=-2\pi\sqrt{\det G_{I}}\,\mathcal{H}^{n-1}\big|_{\D_{ij}}.
	\end{equation*}
	
	\begin{proposition}
		The quantities
		\[
		v_{I,ii}=\alpha_{I,0i}+\sum_{k\ne i}\alpha_{I,ik},\qquad
		v_{I,ij}=-\alpha_{I,ij}\ \ (i\ne j),\qquad
		w_{I}=\detA\,A^{-1}_{ij}v_{I,ij}
		\]
		solve the integrability condition \eqref{jiaohuanvij} and the harmonicity condition \eqref{CYconvij} away from $\D$, and satisfy the distributional equation \eqref{chernvij} in $\B_{I}$.
	\end{proposition}
	
	\begin{proof}
		Harmonicity is clear, each $\alpha_{I,ij}$ is a first-order linear ansatz on $\mathbb{R}^{n}_{\mu_{I}}\times\mathbb{C}_{\eta}$ with constant coefficient matrix $G_{I}$, so the stated distributional equation holds.  
		By definition, $v_{I,ij}=0$ whenever $i\notin I$ or $j\notin I$, hence the integrability condition \eqref{jiaohuanvij} is automatically satisfied.
	\end{proof}
	
	\begin{remark}
		The functions $v_{I,ii}$ and $w_{I}$ actually satisfy the distributional equation \eqref{chernvij} on a larger domain, we shall exploit this fact in later sections.
	\end{remark}
	
	Similarly, we introduce the corresponding linear combinations of the $\beta_{I,ij}$:
	
	\begin{definition}\label{definehIij}
		Set
		\[
		h_{I,ij}=v_{ij}-v_{I,ij},\qquad h_{I}=w-w_{I}.
		\]
		Then $h_{I,ij}$ and $h_{I}$ are smooth harmonic functions on $\B_{I}$ and satisfy the integrability condition \eqref{jiaohuanvij}.
	\end{definition}
	
	These quantities $h_{I,ij}$ and $h_{I}$ will be used below to measure the error between the ansatz metric and the model metric near the discriminant locus, and play a key role in our surgery procedure.
	
	\section{Induction Hypothesis}\label{sectioninduction}
	We now consider the GH ansatz coefficients $V_{ij}^{(1)}, W^{(1)}$ defined on $\B\setminus \D$, which induce the GH metric $g^{(1)}$.  
	As explained at the end of Section~\ref{subsectionfirstorder}, the volume-form error $E^{(1)}$ associated with $g^{(1)}$ does \emph{not} decay as we approach the strata $\D_{I}$ with $\lvert I\rvert\geq 3$.
	
	To overcome this difficulty, we proceed by induction, performing surgery near the locus to refine the Gibbons–Hawking coefficients. 
	First, we construct a new set of coefficients $V_{ij}^{(2)}, W^{(2)}$ whose induced GH metric $g^{(2)}$ agrees with $g^{(1)}$ outside the regions near $\D_{I}$ with $\lvert I\rvert=3$, but approximates the Calabi-–Yau metric more closely on the subsets $\B_{I}$ (defined in Section~\ref{subsectionjianjing1}).  
	Consequently, the volume-form error $E^{(2)}$ associated with $g^{(2)}$ exhibits decay along the locus $\D_{I}$ with $\lvert I\rvert\le 3$.
	
	Building upon $V_{ij}^{(2)}, W^{(2)}$, we next construct $V_{ij}^{(3)}, W^{(3)}$, ensuring that the corresponding volume-form error $E^{(3)}$ decays along the locus $\D_{I}$ with $\lvert I\rvert\le 4$.  
	This process continues inductively.
	
	Thus our proof is inherently inductive.  
	We begin by establishing a set of inductive hypotheses on the GH coefficients $V_{ij}^{(s)}, W^{(s)}$, which are satisfied by the initial data $V_{ij}^{(1)}, W^{(1)}$.  
	In the remainder of the paper we will construct $V_{ij}^{(s+1)}, W^{(s+1)}$ satisfying these same hypotheses, thereby completing the induction and proving Theorem~\ref{maintheorem}. We define the following subsets of $\B_{I}$ for $|I|\ge 2$:
	\begin{equation}\label{defineBIprime}
		\begin{aligned}
			\B'_{I}&:=\bigl\{\,p\in\B\bigm|
			\ \ 2C_{0}\dist_{A}(p,\D_{I})<\dist_{A}(p,\partial\D_{I})\,\bigr\},\\[2pt]
			\B''_{I}&:=\bigl\{\,p\in\B\bigm|
			4\hat{C}C_{0}\dist_{A}(p,\D_{I})<\dist_{A}(p,\partial\D_{I})\,\bigr\}.
		\end{aligned}
	\end{equation}
	Here $\hat{C}>1$ depends only on $N,\lambda,\Lambda$, the significance of this constant will be clarified in Lemma~\ref{juliheaffinedengjia}. We now state our inductive hypotheses, recall that $h_{I,ij}$ and $h_{I}$ are the error terms associated with the index set $I$ introduced in Definition~\ref{definehIij}.
	
	\begin{proposition}[Induction hypothesis]\label{inductionhyp}
		For $1\le s\le N-1$ define the subdomain of $\B$
		\begin{equation}\label{assumationdomain}
			\F_{s}:=\bigcap_{\lvert I\rvert=s+2}
			\bigl\{\,p\in\B\bigm|\dist_{A}(p,\D_{I})>C_{s}\bigr\},
		\end{equation}
		where the constants $C_{s}$ satisfy $C_{s}\gg C_{s-1}$, in particular $C_{s}\gg C_{0}$ with $C_{0}$ the constant used in \eqref{defineBI} to define $\B_{I}$.
		We assume there exist smooth functions $V^{(s)}_{ij}$, $W^{(s)}$ on $\F_{s}\setminus\D$ satisfying the following conditions.
		
		\smallskip
		\noindent
		\textup{(I)} On $\F_{s}\setminus\D$ the matrix $\bigl(V^{(s)}_{ij}\bigr)$ is positive-definite, $W^{(s)}>0$, and the commutativity relations
		\begin{equation}\label{assumationjiaohuan}
			\frac{\partial V^{(s)}_{ij}}{\partial\mu_{k}}
			=\frac{\partial V^{(s)}_{ik}}{\partial\mu_{j}},\qquad
			\frac{\partial^{2}W^{(s)}}{\partial\mu_{i}\partial\mu_{j}}
			+4\frac{\partial^{2}V^{(s)}_{ij}}{\partial\eta\partial\bar\eta}=0
		\end{equation}
		hold.
		Moreover, in the sense of currents on $\F_{s}$,
		\begin{equation}\label{assumationcherncon}
			\begin{split}
				\frac{\sqrt{-1}}{4\pi}\Bigl(
				\frac{\partial^{2}W^{(s)}}{\partial\mu_{i}\partial\mu_{j}}
				+4\frac{\partial^{2}V^{(s)}_{ij}}{\partial\eta\partial\bar\eta}
				\Bigr)\,\mathrm{d}\mu_{i}&\wedge \mathrm{d}\eta\wedge \mathrm{d}\bar\eta\otimes e_{j}
				\\
				&=\sum_{i=1}^{N}\D_{0i}\otimes e_{i}
				+\!\!\sum_{1\le i<j\le N}\D_{ij}\otimes(e_{j}-e_{i}).
			\end{split}
		\end{equation}
		
		\smallskip
		\noindent
		\textup{(II)} For any $K$ with $|K|=s+1$, inside $(\F_{s}\setminus\D)\cap\B_{K}$ the matrix $(V^{(s)}_{ij}-h_{K,ij})$ is positive-definite and $W^{(s)}-h_{K}>0$; they induce a smooth \kah structure on $\pi^{-1}(\B_{K})$.
		The detailed construction will be given in Subsection~\ref{subsectionsmoothstructure}, the smooth \kah structure is induced by the map $F_{K}$ in Proposition~\ref{prop:smooth-structure-DK}, and it satisfies all the additional properties listed in Proposition~\ref{buchongassumation}.
		Moreover, if $J=K\cup\{j\}$ for some $j\notin K$, then on $\B_{J}\cap\F_{s}$
		\[
		|V^{(s)}_{tj}-h_{K,tj}-A_{tj}|\le\frac{C}{\operatorname{dist}_{A}(\,\cdot\,,\D_{J})}
		\quad\text{for all } t\in J,
		\]
		where $C$ is a uniform constant.
		
		\smallskip
		\noindent
		\textup{(III)} For any $I$ with $|I|\ge s+2$, inside $(\F_{s}\setminus\D)\cap\B_{J}$ the matrix $(V^{(s)}_{ij}-h_{I,ij})$ is positive-definite and $W^{(s)}-h_{I}>0$; both depend only on $\mu_{I}$ and $\eta$.
		Moreover, $V^{(s)}_{tl}-h_{I,tl}=A_{tl}$ whenever $t\notin I$ or $l\notin I$.
		In addition,
		\[
		V^{(s)}_{ij}=V^{(s-1)}_{ij},\qquad
		W^{(s)}=W^{(s-1)}
		\quad\text{on }\,
		(\F_{s}\cap\F_{s-1})\setminus
		\bigcup_{\lvert J\rvert=s+2}\B'_{J}.
		\]
		
		\smallskip
		\noindent
		\textup{(IV)} On $\F_{s}\cap \{\operatorname{dist}_{A}(\cdot,\D)>1 \}$ we have
		\begin{gather*}
			\lVert V^{(s)}_{ij}-A_{ij}\rVert
			\le C\dist_{A}(\,\cdot\,,\D)^{-1},\qquad
			\lVert W^{(s)}-\det A\rVert
			\le C\dist_{A}(\,\cdot\,,\D)^{-1},\\[2pt]
			\lVert\nabla^{k}V^{(s)}_{ij}\rVert
			\le C\dist_{A}(\,\cdot\,,\D)^{-k},\qquad
			\lVert\nabla^{k}W^{(s)}\rVert
			\le C\dist_{A}(\,\cdot\,,\D)^{-k},\quad k\ge 1,
		\end{gather*}
		where the affine connection and norms are induced by $g_{A}$ and $C>0$ is a uniform constant.
	\end{proposition}
	
	\begin{remark}
		The precise formulation of \textup{(II)} will be given in Section~\ref{subsectionsmoothstructure}, since the way $V^{(s)}_{ij}$, $W^{(s)}$ induce the smooth structure is rather involved and would overload the present statement.
	\end{remark}
	
	\begin{remark}\label{HKdefine}
		Let $|I|=s+2$, let $K\subsetneqq I$, $|K|\ge 2$, define
		\begin{align}\label{defineHk}
			\H_{K}=\B_{I}\cap\Bigl(\B_{K}\setminus
			\bigcup_{K\subsetneqq J\subsetneqq I}\B^{\prime}_{J}\Bigr).
		\end{align}
		By item \textup{(III)} above when $|K|=t+1\ge 2$, we have
		\[
		V^{(s)}_{ij}=V^{(t)}_{ij},\qquad W^{(s)}=W^{(t)}
		\qquad\text{on } \F_{t}\cap\H_{K}.
		\]
		Consequently, property \textup{(II)} implies that $V^{(s)}_{ij}$ and $W^{(s)}$ induce a smooth structure on $\F_{s}\cap\F_{t}\cap\D_{K}$.
	\end{remark}
	
	For the initial coefficients $V^{(1)}_{ij}$, $W^{(1)}$ the properties \textup{(I)}, \textup{(III)} and \textup{(IV)} are clearly satisfied.  
	As shown in the next subsection, they also induce a smooth structure near $\D_{ij}$, thereby completing the verification of \textup{(II)}.

	\subsection{Smooth Structure along Loci}\label{subsectionsmoothstructure}
	In this subsection we fix
	\[
	K=\{0,1,\dots,k\}
	\]
	and explain how $V^{(k)}_{ij}$ and $W^{(k)}$ induce a smooth structure on $\D_{K}$.  
	By the symmetry of the discriminant locus described at the beginning of Subsection~\ref{subsectionthegeometry}, the discussion for a general $K$ is completely analogous.
	
	Whenever we speak of $\B_{K}$ below, we always mean its intersection with the domain $\F_{k}$ defined in Proposition~\ref{inductionhyp}.  
	On $\B_{K}$ set
	\[
	V_{K,ij}:=V^{(k)}_{ij}-h_{K,ij},\qquad
	W_{K}:=W^{(k)}-h_{K}.
	\]
	Since $V_{K,ij}$ and $W_{K}$ satisfy the integrability condition and the Chern-class condition of the singular $\mathbb{T}^{N}$-bundle $\pi\colon\pi^{-1}(\B_{K})\to\B_{K}$, and $h_{K,ij}$ and $h_{K}$ are smooth harmonic functions on $\B_{K}$ satisfying the integrability condition as well, the differences $V_{K,ij}$ and $W_{K}$ also satisfy both the integrability condition and the Chern-class condition.
	
	By condition~(II) of Proposition~\ref{inductionhyp}, the matrix $(V_{K,ij})$ is positive definite and $W_{K}>0$.  
	For $k=1$ (i.e.\ $|K|=2$) the definitions in Subsection~\ref{subsectionjianjing1} give
	\[
	V_{K,ij}=A_{ij}+v_{K,ij},\qquad
	W_{K}=\det A+w_{K},
	\]
	so positive-definiteness is immediate.
	
	By Theorem~\ref{GibbonsHawking} there exists a complex structure $J_{K}$ on $\pi^{-1}(\B_{K}\setminus\D)$ together with a \kah metric $g_{K}$, denote by $\omega_{K}$ its \kah form and by $\Omega_{K}$ the holomorphic volume form.  
	We therefore obtain the GH structure
	\[
	\bigl(\pi^{-1}(\B_{K}\setminus\D),\,
	J_{K},\,\Omega_{K},\,g_{K},\,\omega_{K}\bigr).
	\]
	In the sequel we will show that this GH structure is biholomorphically isomorphic to a certain model geometry, to that end we first describe the standard space.
	
	\begin{example}[Model geometry]\label{biaozhunmuxing}
		Define the manifold
		\[
		M^{n}:=\C^{n+1}\times\R^{2(N-n)}
		\]
		equipped with the standard complex structure
		\[
		J_{\text{Nor}}:=J_{\C^{n+1}}\oplus J_{\R^{2(N-n)}}.
		\] 
		Using the global coordinates
		\[
		(z_{0},\dots,z_{n},x_{1},\dots,x_{N-n},y_{1},\dots,y_{N-n}),
		\]
		a basis of holomorphic $(1,0)$-forms is given by
		\[
		\mathrm{d} z_{0},\dots,\mathrm{d} z_{n},\quad \mathrm{d} x_{j}+\sqrt{-1}\,\mathrm{d} y_{j}\ (1\le j\le N-n).
		\]
		The standard holomorphic volume form is
		\[
		\Omega_{\text{Nor}}:=(\sqrt{-1})^{2N-n}\bigwedge_{i=0}^{n}\mathrm{d} z_{i}\wedge
		\bigwedge_{j=1}^{N-n}(\mathrm{d} x_{j}+\sqrt{-1}\,\mathrm{d} y_{j}).
		\]
		
		There is a natural action of the group $\mathbb{T}^{n}\times\R^{N-n}$ on $M^{n}$:
		\begin{itemize}
			\item
			$\mathbb{T}^{n}$ acts on $\C^{n+1}$ by
			\[
			(\theta_{1},\dots,\theta_{n})
			\cdot(z_{0},\dots,z_{n})
			=(e^{-\sqrt{-1}\sum\theta_{i}}z_{0},e^{\sqrt{-1}\theta_{1}}z_{1},\dots,e^{\sqrt{-1}\theta_{n}}z_{n}).
			\]
			
			\item
			$\R^{N-n}$ acts on the $\R^{2(N-n)}$-factor by
			\[
			(t_{1},\dots,t_{N-n})\cdot
			\bigl(x_{j}+\sqrt{-1}\,y_{j}\bigr)_{j=1}^{N-n}
			=\bigl(x_{j}+\sqrt{-1}\,(y_{j}+t_{j})\bigr)_{j=1}^{N-n}.
			\]
		\end{itemize}
		
		Define the holomorphic projection
		\[
		\pi^{n}_{\eta}\colon M^{n}\longrightarrow\C,\qquad
		\pi^{n}_{\eta}(z,x,y):=z_{0}\cdots z_{n}.
		\]
		This map equips $M^{n}$ with the structure of a holomorphic fibre bundle over $\C$.
		
		Endow $M^{n}$ with the standard \kah form
		\[
		\omega_{\text{Nor}}:=
		-\frac{\sqrt{-1}}{2}\sum_{i=0}^{n}\mathrm{d} z_{i}\wedge d\bar z_{i}
		+\sum_{j=1}^{N-n}\mathrm{d} x_{j}\wedge \mathrm{d} y_{j}.
		\]
		The moment map $\mu^{n}=(\mu^{n}_{1},\dots,\mu^{n}_{N})\colon M^{n}\to\R^{N}$ for the above $\mathbb{T}^{n}\times\R^{N-n}$-action is given by
		\begin{alignat*}{2}
			\mu^{n}_{i}&=\frac{1}{2}\bigl(|z_{i}|^{2}-|z_{0}|^{2}\bigr),
			&\quad &i=1,\dots,n,\\[2pt]
			\mu^{n}_{j}&=x_{j-n},
			&\quad &j=n+1,\dots,N.
		\end{alignat*}
		Setting $\pi^{n}:=(\mu^{n}_{1},\dots,\mu^{n}_{N},\pi^{n}_{\eta})\colon M^{n}\to\R^{N}\times\C$, one verifies that $\pi^{n}$ is a diffeomorphism onto its image; denote the image by
		\[
		\B^{n}:=\pi^{n}(M^{n})=\R^{N}\times\C.
		\]
		
		Inside $\B^{n}$ we may speak of the discriminant locus of the group action: for any index set $I\subset\{0,1,\dots,n\}$ put
		\[
		\D^{n}_{I}:=\bigcap_{i\in I}\pi^{n}\bigl(\{z_{i}=0\}\bigr),
		\qquad
		\D^{n}:=\bigcup_{I}\D^{n}_{I}.
		\]
		When $n=N$ this reduces to the flat example discussed in Example~\ref{examplefalt}.
	\end{example}

	As a detailed formulation of condition~(II) in Proposition~\ref{inductionhyp} we have
	
	\begin{proposition}\label{prop:smooth-structure-DK}
		There exists a continuous bundle map $\bF_{K}$ fitting into the commutative diagram
		\[
		\begin{tikzcd}[column sep=large, row sep=large]
			\widetilde{\pi^{-1}(\B_{K})}
			\arrow[rr,"\bF_{K}"]
			\arrow[dr,"\eta"']
			&&
			\C^{k+1}\times\R^{2(N-k)}
			\arrow[dl,"\pi^{k}_{\eta}"]
			\\
			&
			\C
		\end{tikzcd}
		\]
		where $\widetilde{\pi^{-1}(\B_{K})}$ denotes the universal cover of $\pi^{-1}(\B_{K})$, and $\pi_{\eta}$ is the projection defined in Example~\ref{biaozhunmuxing}. The map $\bF_{K}$ enjoys the following properties:
		
		\begin{enumerate}[label={(\roman*)}]
			\item 
			$\bF_{K}$ is injective and holomorphic on $\widetilde{\pi^{-1}(\B_{K}\setminus\D)}$, and satisfies  
			\[
			(\bF_{K})^{*}(\Omega_{\text{Nor}})=\Omega_{K}.
			\]  
			Consequently, we may identify $\widetilde{\pi^{-1}(\B_{K})}$ with an open subset of the model space $\C^{k+1}\times\R^{2(N-k)}$.
			\item
			Let $\{e_{i}\}_{i=1}^{N}$ be the standard basis of the Lie algebra $\mathfrak{t}^{N}$ of the original $\mathbb{T}^{N}$-action.  
			Introduce the affine basis
			\begin{equation}\label{Kfangsheji}
				e_{K,K}:=e_{K},\qquad
				e_{K,K'}:=e_{K'}+A_{K'}^{-1}A_{K'K}e_{K}.
			\end{equation}
			Then $\bF_{K}$ intertwines the action of $\mathbb{T}^{k}\times\R^{N-k}$ generated by \eqref{Kfangsheji} with the standard action on $M^{k}$ described in Example~\ref{biaozhunmuxing}.
			
			\item
			The push-forward $(\bF_{K})_{*}(\omega_{K})$ extends smoothly across the discriminant locus $\D_{K}$, hence it defines a smooth \kah metric on the image of $\bF_{K}$.  
			Moreover, this metric is Calabi-–Yau on the subset
			\[
			\bF_{K}\bigl(\widetilde{\pi^{-1}(\B''_{K})}\bigr).
			\]
			
			\item
			By the induction hypothesis, a smooth structure has already been constructed on $\F_{k-1}\cap\B_{K}$.  
			We claim that the map $\bF_{K}$ is compatible with the previously defined smooth structure and hence furnishes a smooth structure along $\D_{K}$.
		\end{enumerate}
	\end{proposition}
	
	We now describe the explicit form of the map $\bF_{K}$.  
	By Theorem~\ref{GibbonsHawking}, using ${V}_{K,ij}$ and ${W}_{K}$ as GH coefficients, there exists a connection $1$-form
	\[
	\vartheta_{K}=\sum_{i=1}^{N}\vartheta_{K,i}\otimes e_{i}
	\quad\text{on }\pi^{-1}(\B_{K}\setminus\D)
	\]
	satisfying
	\[
	\mathrm{d}\vartheta_{K,i}=F_{K,i}
	\qquad\text{and}\qquad
	\vartheta_{K,i}(X_{j})=\delta_{ij},
	\]
	where $X_{j}$ is the smooth vector field generated by $e_{j}$ and
	\begin{equation*}
		F_{K,i}:=\sqrt{-1}\biggl(
		\frac{1}{2}\frac{\partial{W}_{K}}{\partial\mu_{i}}\,\mathrm{d}\eta\wedge \mathrm{d}\bar\eta
		+\frac{\partial{V}_{K,ij}}{\partial\eta}\,\mathrm{d}\mu_{j}\wedge \mathrm{d}\eta
		-\frac{\partial{V}_{K,ij}}{\partial\bar\eta}\,\mathrm{d}\mu_{j}\wedge \mathrm{d}\bar\eta
		\biggr).
	\end{equation*}
	Here we emphasize that, throughout this paper, $F$ denotes the curvature of a metric on a principal bundle, whereas the bold symbol $\boldsymbol{F}$ is reserved for morphisms between vector bundles.

	With respect to the affine basis \eqref{Kfangsheji}, let $\tilde X_{i}$ be the vector field generated by ${e}_{K,i}$.  
	A dual basis of connection forms is then
	\[
	\tilde\vartheta_{K,K}
	:=\vartheta_{K,K}
	-A_{KK'}A_{K'}^{-1}\vartheta_{K,K'},
	\qquad
	\tilde\vartheta_{K,K'}
	:=\vartheta_{K,K'},
	\]
	so that $\tilde\vartheta_{K,i}(\tilde X_{j})=\delta_{ij}$ and $\vartheta_{K}=\sum_{i}\tilde\vartheta_{K,i}\otimes{e}_{K,i}$.
	
	According to Theorem~\ref{GibbonsHawking}, a basis of $(1,0)$-forms for the complex structure $J_{K}$ is
	\[
	\xi_{K,i}:=\sum_{j}{V}_{K,ij}\,\mathrm{d}\mu_{j}+\sqrt{-1}\,\tilde\vartheta_{K,i},
	\qquad
	\mathrm{d}\eta.
	\]
	Performing the linear change
	\[
	\tilde\xi_{K,K}:=\xi_{K,K}-A_{KK'}A_{K'}^{-1}\xi_{K,K'},
	\qquad
	\tilde\xi_{K,K'}:=\xi_{K,K'},
	\qquad
	\mathrm{d}\eta,
	\]
	we obtain
	\[
	\tilde\xi_{K,i}(\tilde X_{j})=\sqrt{-1}\delta_{ij}.
	\]
	
	On $\C^{k+1}\times\R^{2(N-k)}$ the forms
	\[
	dw_{0},\dots,dw_{k},\quad \mathrm{d} x_{j}+\sqrt{-1}\,\mathrm{d} y_{j}\ (j=1,\dots,N-k)
	\]
	are closed and holomorphic.  
	Since $\bF_{K}$ is holomorphic, their pull-backs are closed holomorphic $1$-forms on $\widetilde{\pi^{-1}(\B_{K}\setminus\D)}$.  
	Comparing their values on the vector fields $\tilde X_{i}$ yields
	\begin{equation}\label{eq:dxidyi}
		\mathrm{d} x_{i}+\sqrt{-1}\,\mathrm{d} y_{i}=\tilde\xi_{K,i},
		\qquad
		i=k+1,\dots,N.
	\end{equation}
	Away from $w_{i}=0$ we may write
	\begin{equation}\label{eq:dlogwi}
		\begin{aligned}
			\mathrm{d}\log w_{i}&=\tilde\xi_{K,i}+\gamma_{K,i}\,\mathrm{d}\eta,
			&&i=1,\dots,k,\\
			\mathrm{d}\log w_{0}&=-\sum_{i=1}^{k}\tilde\xi_{K,i}+\gamma_{K,0}\,\mathrm{d}\eta,
		\end{aligned}
	\end{equation}
	where the functions $\gamma_{K,i}$ are chosen so that the right-hand sides are closed.  
	Equations \eqref{eq:dxidyi}--\eqref{eq:dlogwi} determine the local expression of $\bF_{K}$, the same argument as in Subsection~\ref{subsectiontheimageofthehol} shows that $\bF_{K}$ is globally defined once we fix the initial constants so that
	\[
	w_{0}\cdots w_{k}=\eta.
	\]
	
	The Calabi--Yau condition on $\B_{K}''$ reads
	\[
	\det(V_{K,ij})=W_{K}.
	\]
	For $k=1$ (i.e.\ $|K|=2$) the data $V_{K,ij}$, $W_{K}$ describe the product of a $2$-dimensional Taub--NUT metric and a flat factor near $\D_{ij}$, hence all properties listed in Proposition~\ref{prop:smooth-structure-DK} are satisfied by $V_{ij}^{(1)}$, $W^{(1)}$.

	With the above notation, we now complete the supplement to Proposition~\ref{inductionhyp}\,(II).  
	First, we introduce the weighted Sobolev norms induced by $g_{K}$.  
	Following the approach of \cite{li2023syz}, for a $\mathbb{T}^{k}\times\R^{N-k}$-invariant tensor $T$ on ${\pi^{-1}(\B_{K})}$ we define its weighted Sobolev norm with weight function
	\[
	\ell_{1}=1+\dist_{A}(\,\cdot\,,\D).
	\]
	For $0<\alpha<1$ set the normalized H\"older semi-norm
	\[
	[T]_{\alpha}
	=\sup_{p}\ell_{1}(p)^{\alpha}
	\sup_{p'\in B_{g_{K}}(p,\ell_{1}(p)/10)}
	\frac{|T(p)-T(p')|_{g_{K}}}{d_{g_{K}}(p,p')^{\alpha}},
	\]
	where $T(p)$ and $T(p')$ are compared via parallel transport along minimal geodesics.  
	For integers $k\ge 2$ the weighted norm of $T$ on $\pi^{-1}(\B_{K})$ is then
	\[
	\|T\|_{C^{k,\alpha}_{g_{K}}}
	=\sum_{j=0}^{k}\|\ell_{1}^{j}\nabla_{g_{K}}^{j}T\|_{L^{\infty}}
	+[\ell_{1}^{k}\nabla_{g_{K}}^{k}T]_{\alpha}.
	\]
	
	\begin{proposition}\label{buchongassumation}
		Using the notation of Proposition~\ref{prop:smooth-structure-DK}, the smooth \kah structure induced by $V_{K,ij}$ and $W_{K}$ satisfies:
		\begin{enumerate}[label={(\roman*)}]
			\item 
			Let $\{X_{i}\}_{i=1}^{N}$ be the fundamental vector fields on $\pi^{-1}(\B_{K})$ corresponding to the $\mathbb{T}^{N}$-action.  
			Then, with respect to the above weighted norm,
			\[
			\|X_{i}\|_{C^{k,\alpha}_{g_{K}}}\le C.
			\]
			The holomorphic volume form $\Omega_{K}$ induced by the GH ansatz is also bounded:
			\[
			\|\Omega_{K}\|_{C^{k,\alpha}_{g_{K}}}\le C.
			\]
			
			\item
			The curvature tensor of $g_{K}$ is bounded:
			\[
			\|\mathrm{Rm}_{g_{K}}\|_{C^{k,\alpha}_{g_{K}}}\le C.
			\]
			
			\item
			For the holomorphic map $(w_{0},w_{1},\dots,w_{k})$ defined in Proposition~\ref{prop:smooth-structure-DK}, set
			\begin{gather*}
				|w| = \Bigl(\sum_{i=0}^{k}|w_{i}|^{2}\Bigr)^{1/2},\\[4pt]
				|\vec{\mu}_{K}|_{G_{K}}
				= \bigl(|\mu_{K}|^{2}+\det A\,|\eta|^{2}\bigr)^{1/2}
				= \bigl(\mu_{K}^{\mathbb{T}}G_{K}\mu_{K}+\det A\,|\eta|^{2}\bigr)^{1/2}.
			\end{gather*}
			Then, whenever $|\vec{\mu}_{K}|_{G_{K}}>R$,
			\begin{equation}\label{duishuzengzhangK}
				K_{1}e^{K_{2}|\mu_{K}|}\,|\vec{\mu}_{K}|_{G_{K}}
				\le |w|\le
				K_{3}e^{K_{4}|\mu_{K}|}\,|\vec{\mu}_{K}|_{G_{K}},
			\end{equation}
			where $R>0$ and $K_{i}>0$ depend only on $n$, $\lambda$, and $\Lambda$.
		\end{enumerate}
	\end{proposition}
	
	One readily checks that these assumptions hold for $|K|=2$, thanks to the fine properties of the $2$-dimensional Taub--NUT metric.  
	For item (iii), take $K=\{0,1\}$; a direct computation gives
	\[
	|w_{1}|^{2}
	=e^{C}e^{2G_{\{0,1\}}\mu_{1}}\bigl(\mu_{1}+\sqrt{\mu_{1}^{2}+|\eta|^{2}}\bigr),
	\
	|w_{0}|^{2}
	=e^{-C}e^{-2G_{\{0,1\}}\mu_{1}}\bigl({-}\mu_{1}+\sqrt{\mu_{1}^{2}+|\eta|^{2}}\bigr).
	\]
	
	Item (i) will be used in Subsection~\ref{subsectionasymptoticprop} to prove the metric deviation estimate, (ii) ensures that Tian--Yau--Hein package applies, and (iii) assists in the proof of Lemma~\ref{zzengzhangxing}.  
	Propositions~\ref{prop:smooth-structure-DK} and \ref{buchongassumation} thus complete the supplementary description of Proposition~\ref{inductionhyp}, these properties themselves form part of our induction hypothesis.
	
	\section{Holomorphic Viewpoint}\label{sectionholpoint}
	From now on we fix an index set $I$ with $|I|=n+1$, where $2\le n\le N-1$.  
	Owing to the symmetry of the discriminant locus discussed in Section~\ref{subsectionthegeometry}, we may assume without loss of generality that
	\[
	I=\{0,1,\dots,n\}.
	\]
	By the induction hypothesis (Proposition~\ref{inductionhyp}) we already have GH coefficients $V_{ij}^{(n-1)}$, $W^{(n-1)}$ defined on $\F_{n-1}$ satisfying all properties listed in that proposition.  
	Our next task is to perform a “surgery” on these coefficients inside $\B_{I}$ (see Section~\ref{sectionsurgery}).  
	The present section prepares the ground for that surgery by showing that the GH structure determined by $V_{ij}^{(n-1)}$, $W^{(n-1)}$ enjoys properties analogous to those established in Proposition~\ref{prop:smooth-structure-DK}.

	\subsection{Gibbons--Hawking Structure on the Model Space}\label{subsectionstructureonmodel}
	Our surgery ultimately takes place inside $\B_{I}$,  
	we therefore focus on the topological properties of $\B_{I}$ and,  
	using $V_{ij}^{(n-1)}$ and $W^{(n-1)}$, build a model space near $\B_{I}$.
	
	On $\F_{n-1}\cap\B_{I}$ we set
	\begin{align}
		P_{I,ij}=A_{ij}+p_{I,ij}=V_{ij}^{(n-1)}-h_{I,ij},
		\qquad
		Q_{I}=\det A+q_{I}=W^{(n-1)}-h_{I}.
	\end{align}
	Proposition~\ref{inductionhyp} shows that $P_{I,ij}$ and $Q_{I}$ depend only on $\mu_{I}$ and $\eta$.  
	Moreover, $P_{I,ij}$ is positive-definite and $Q_{I}>0$.
	
	Fix $p=(\mu_{1},\dots,\mu_{n},\eta)\in\B_{I}$ and let $p_{I}$ be its projection to $\D_{I}$.  
	The computations in Subsection~\ref{subsectionthegeometry} give
	\[
	p_{I}=(0,\dots,0,\nu_{I,n+1},\dots,\nu_{I,N},0),
	\]
	where $\nu_{I}$ is an $(N-n)$-vector with $\nu_{I,j}>0$ for $j\in I^{\prime}$ and
	\[
	\nu_{I,I^{\prime}}=\mu_{I^{\prime}}+A_{I^{\prime}}^{-1}A_{I^{\prime}I}\mu_{I}.
	\]
	
	Choose $t_{j}>0$ ($j\in I^{\prime}$) and regard $P_{I,ij}$ and $Q_{I}$ as functions on
	\[
	\F_{n-1}\cap\B_{I}\cap\{\,p\in B:\nu_{I,j}=t_{j}\,\}.
	\]
	Because they depend only on $(\mu_{I},\eta)$, they are independent of the $t_{j}$.  
	Letting $t_{j}\to+\infty$ extends $P_{I,ij}$ and $Q_{I}$ to the whole space  
	$\mathbb{R}^{n}_{\mu_{I}}\times\mathbb{C}_{\eta}$ with a compact subset removed.  
	The obstruction to a global extension is the requirement  
	$\operatorname{dist}_{A}(p,\partial\D_{I})>C_{n-1}$ for $p\in\F_{n-1}$.  
	Hence, as functions of $(\mu_{I},\eta)$, they are defined only on
	\[
	\bigl\{\,q=(\mu_{I},\eta)\in\mathbb{R}^{n}\times\mathbb{C}:
	\operatorname{dist}_{G_{I}}(q,O)>C_{n-1}\,\bigr\}.
	\]
	
	On $\B_{I}$ the action generated by $e_{I^{\prime}}$ is non-degenerate.  
	To simplify the singular bundle geometry we pass to the model space $M^{n}$ of Example~\ref{biaozhunmuxing}.  
	Regarding $P_{I,ij}$ and $Q_{I}$ as functions on the base
	\[
	\B^{n}=\mathbb{R}^{n}_{\mu_{I}}\times\mathbb{C}_{\eta}
	\times\mathbb{R}^{N-n}_{\mu_{I^{\prime}}},
	\]
	this corresponds to trivialising the action generated by $e_{I^{\prime}}$,  
	i.e.\ we treat it as the $\mathbb{R}^{N-n}$-action.  
	All subsequent statements are proved on the singular bundle  
	$\pi^{n}\colon M^{n}\to\B^{n}$ and then transferred back to $\B_{I}$.
	
	Equip $\B^{n}$ with the metric $g_{A}$ and set
	\begin{align}
		\mathcal{T}_{I}=\{\,p\in\B^{n}:
		\operatorname{dist}_{A}(p,\D^{n}_{I})>C_{n-1}\,\}.
	\end{align}
	On $\mathcal{T}_{I}$ the functions $P_{I,ij}$ and $Q_{I}$ satisfy the integrability condition  
	and a distributional equation analogous to~\eqref{cherncon}.  
	Theorem~\ref{GibbonsHawking} therefore yields a connection form
	\[
	\vartheta_{I}=(\vartheta_{I,1},\dots,\vartheta_{I,N})
	\]
	on $(\pi^{n})^{-1}(\mathcal{T}_{I}\setminus\D^{n})$ with ${\rm d}\vartheta_{I}=F_{I}$, where
	$F_{I}=F_{I,i}\otimes e_{i}$ and
	\begin{equation*}
		F_{I,i}= \sqrt{-1} \left(
		\frac{1}{2}\frac{\partial Q_{I}}{\partial\mu_{j}}\,\mathrm{d}\eta\wedge \mathrm{d}\bar{\eta}
		+\frac{\partial P_{I,ij}}{\partial\eta}\,\mathrm{d}\mu_{j}\wedge \mathrm{d}\eta
		-\frac{\partial P_{I,ij}}{\partial\bar{\eta}}\,\mathrm{d}\mu_{j}\wedge \mathrm{d}\bar{\eta}
		\right).
	\end{equation*}
	Since $P_{I,ij}$ and $Q_{I}$ depend only on $(\mu_{I},\eta)$,  
	and since by Proposition~\ref{inductionhyp} $P_{I,kl}=A_{kl}$ whenever $k\notin I$ or $l\notin I$,  
	we have $F_{I,j}=0$ for $j\in I^{\prime}$.  
	Up to a gauge transformation we may therefore assume
	\[
	\vartheta_{I,j}=\mathrm{d}\theta_{I,j},\qquad j\in I^{\prime}.
	\]
	We now obtain the Gibbons--Hawking structure
	\begin{align}\label{modelGHstructure}
		\bigl((\pi^{n})^{-1}(\mathcal{T}_{I}\setminus\D^{n}),
		J_{I},\Omega_{I},g_{I},\omega_{I}\bigr)
	\end{align}
	with GH metric
	\begin{align*}
		g_{I}=P_{I,ij}\,\mathrm{d}\mu_{i}\otimes \mathrm{d}\mu_{j}
		+P^{-1}_{I,ij}\,\vartheta_{I,i}\otimes\vartheta_{I,j}
		+Q_{I}\lvert \mathrm{d}\eta\rvert^{2}.
	\end{align*}
	
	Choose an affine basis of $\mathfrak{t}^{n}\times\mathfrak{r}^{N-n}$:
	\[
	e_{I,I}:=e_{I},\qquad
	e_{I,I^{\prime}}:=e_{I^{\prime}}
	+A_{I^{\prime}}^{-1}A_{I^{\prime}I}e_{I}
	\]
	and the corresponding moment coordinates
	\[
	\nu_{I,I}=\mu_{I},\qquad
	\nu_{I,I^{\prime}}
	=\mu_{I^{\prime}}+A_{I^{\prime}}^{-1}A_{I^{\prime}I}\mu_{I}.
	\]
	The dual connection forms are
	\[
	\tilde\vartheta_{I,I}
	=\vartheta_{I,I}-A_{II^{\prime}}A_{I^{\prime}}^{-1}\vartheta_{I,I^{\prime}},
	\qquad
	\tilde\vartheta_{I,I^{\prime}}=\vartheta_{I,I^{\prime}}.
	\]
	On $(\pi^{n})^{-1}(\mathcal{T}_{I}\setminus\D^{n})$ the metric $g$ splits orthogonally as $g=g_{1}+g_{2}$, where
	\[
	\begin{aligned}
		g_{1}={}&{\rm d}\nu_{I,I}^{\top}(G_{I,ij}+p_{I,ij})\,{\rm d}\nu_{I,I}
		+\tilde\vartheta_{I,I}^{\top}(G_{I,ij}+p_{I,ij})^{-1}\tilde\vartheta_{I,I}\\[2pt]
		&\qquad+\bigl(\det A+q_{I}\bigr)\lvert{\rm d}\eta\rvert^{2}
	\end{aligned}
	\]
	and
	\[
	g_{2}={\rm d}\nu_{I,I^{\prime}}^{\top}A_{I^{\prime}}\,{\rm d}\nu_{I,I^{\prime}}
	+\tilde\vartheta_{I,I^{\prime}}^{\top}A_{I^{\prime}}^{-1}
	\tilde\vartheta_{I,I^{\prime}}.
	\]
	Notice that $p_{I,ij}$ and $q_{I}$ depend only on $(\mu_{I},\eta)$.  
	The pair $(\nu_{I,I^{\prime}},\theta_{I,I^{\prime}})$ gives an isomorphism from  
	$(\pi^{n})^{-1}(\mathcal{T}_{I}\setminus\D^{n})$ onto $\mathbb{R}^{2(N-n)}$.
	
	We next show that the model GH structure in~\eqref{modelGHstructure} splits as a product.  
	Using
	\[
	e_{I,I},\; e_{I,I^{\prime}}
	\]
	as a basis of the Lie algebra $\mathfrak{t}^{n}\times\mathfrak{r}^{N-n}$  
	induces an orthogonal decomposition of $M^{n}$.  
	The first factor
	\[
	\pi^{n}_{1}\colon M^{n}_{1}\to B^{n}_{1}
	\]
	is a singular $\mathbb{T}^{n}$-bundle over  
	$\mathbb{R}^{n}_{\mu_{I}}\times\mathbb{C}_{\eta}$ with fibre generated by $e_{I,I}$,  
	while the second factor
	\[
	\pi^{n}_{2}\colon M^{n}_{2}\to B^{n}_{2}
	\]
	is an $\mathbb{R}^{N-n}$-bundle over  
	$\mathbb{R}^{N-n}_{\nu_{I,I^{\prime}}}$ with fibre generated by $e_{I,I^{\prime}}$.  
	Consequently, $\mathcal{T}_{I}$ decomposes as
	\[
	\mathcal{T}_{I}
	=\bigl\{p\in\mathbb{R}^{n}_{\mu_{I}}\times\mathbb{C}_{\eta}:
	\operatorname{dist}_{G_{I}}(p,O)>C_{n-1}\bigr\}
	\times\mathbb{R}^{N-n}_{\nu_{I,I^{\prime}}}
	=:\mathcal{T}_{I,1}\times\mathbb{R}^{N-n}_{\nu_{I,I^{\prime}}}.
	\]
	
	Equip $\mathcal{T}_{I,1}\subset B^{n}_{1}$ with the GH data
	\begin{align}
		P_{ij}=G_{I,ij}+p_{I,ij},\qquad Q=Q_{I}=\det A+q_{I},
	\end{align}
	and choose the connection form $\tilde\vartheta_{I,I}$.  
	The resulting metric is exactly the $g_{1}$ given above.
	
	\begin{remark}
		For any geometric object defined on $\B^{n}$ that is compatible with the above splitting,  
		we append the subscript $1$ to denote its projection onto the first factor.  
		Examples include $\mathcal{T}_{I,1}$, $\D^{n}_{1}$, $\D^{n}_{K,1}$, etc.
	\end{remark}
	
	For the holomorphic complex structure $J_{I}$, Theorem~\ref{GibbonsHawking} provides the basis of holomorphic $(1,0)$-forms
	\[
	\xi_{I,i}=(A_{ij}+p_{I,ij})\,{\rm d}\mu_{j}
	+\sqrt{-1}\,\vartheta_{I,i},\qquad
	i=1,\dots,n,
	\]
	and $g_{I}$ is a \kah metric with respect to $J_{I}$.
	
	Restricting to the first factor and inserting the GH coefficients $P_{ij}$ and $Q$ yields a complex structure $J_{1}$ whose holomorphic $(1,0)$-forms are spanned by
	\[
	\tilde\xi_{I,i}=\sum_{j\in I}(G_{I,ij}+p_{I,ij})\,{\rm d}\nu_{I,j}
	+\tilde\vartheta_{I,i},\qquad
	i\in I.
	\]
	Since
	$\tilde\xi_{I,I}
	=\xi_{I,I}-A_{II^{\prime}}A_{I^{\prime}}^{-1}\xi_{I,I^{\prime}}$,  
	the structures $J_{1}$ and $J_{I}$ are compatible.
	
	On the second factor the restriction of $J_{I}$ satisfies  
	$\tilde\xi_{I,I^{\prime}}=\xi_{I,I^{\prime}}$,  
	which constitutes a holomorphic basis and induces a complex structure $J_{2}$ on $M^{n}_{2}$.  
	Consequently,
	\[
	J_{I}=J_{1}\oplus J_{2},
	\]
	i.e.\ $J_{I}$ is compatible with the product decomposition.  
	Finally, each $g_{i}$ is \kah with respect to $J_{i}$ ($i=1,2$).
	
	The holomorphic volume form $\Omega_{I}$ is
	\[
	\Omega_{I}
	=\bigwedge_{j=1}^{N}(-\sqrt{-1}\xi_{j})\wedge{\rm d}\eta.
	\]
	We introduce the corresponding forms on the two factors:
	\[
	\Omega_{1}
	=\bigwedge_{i\in I}(-\sqrt{-1}\tilde\xi_{i})\wedge{\rm d}\eta,
	\qquad
	\Omega_{2}
	=\bigwedge_{i\in I^{\prime}}(-\sqrt{-1}\tilde\xi_{i}).
	\]
	Then $\Omega_{1}$ is the holomorphic volume form induced by the GH data $P_{ij}$ and $Q$, and
	\[
	\Omega_{I}=(-1)^{N-n}\Omega_{1}\wedge\Omega_{2}.
	\]
	
	The \kah form of $g_{I}$ with respect to $J_{I}$ is
	\[
	\omega_{I}
	={\rm d}\mu_{i}\wedge\vartheta_{i}
	+\frac{\sqrt{-1}}{2}Q_{I}\,{\rm d}\eta\wedge{\rm d}\bar\eta,
	\]
	while those of $g_{1}$ and $g_{2}$ are
	\begin{align*}
		\omega_{1}
		&=\frac{\sqrt{-1}}{2}\bigl(P^{-1}_{ij}\tilde\xi_{i}\wedge\bar{\tilde\xi}_{j}
		+Q\,{\rm d}\eta\wedge{\rm d}\bar\eta\bigr)
		=\sum_{i=1}^{N}{\rm d}\mu_{i}\wedge\tilde\vartheta_{i}
		+\frac{\sqrt{-1}}{2}Q\,{\rm d}\eta\wedge{\rm d}\bar\eta,\\[4pt]
		\omega_{2}
		&=\frac{\sqrt{-1}}{2}A^{-1}_{I^{\prime},ij}
		\tilde\xi_{i}\wedge\bar{\tilde\xi}_{j}
		=\sum_{i=n+1}^{N}{\rm d}\nu_{I,i}\wedge{\rm d}\theta_{I,i}.
	\end{align*}
	Hence $\omega=\omega_{1}+\omega_{2}$.  Using
	\[
	\omega_{2}^{N-n}
	=\det(A_{I^{\prime}}^{-1})\frac{(N-n)!}{2^{N-n}}
	(\sqrt{-1})^{(N-n)^{2}}\Omega_{2}\wedge\bar\Omega_{2},
	\]
	we see that if
	\[
	\omega_{1}^{n+1}
	=\det A_{I^{\prime}}\frac{(n+1)!}{2^{n+1}}
	(\sqrt{-1})^{(n+1)^{2}}\Omega_{1}\wedge\bar\Omega_{1},
	\]
	then $\omega$ is Calabi-–Yau with respect to $\Omega$, i.e.\
	\begin{align*}
		\omega^{N+1}
		=(\omega_{1}+\omega_{2})^{N+1}
		&=\frac{(N+1)!}{(n+1)!(N-n)!}
		\omega_{1}^{n+1}\wedge\omega_{2}^{N-n}\\[4pt]
		&=\frac{(N+1)!}{2^{N+1}}
		(\sqrt{-1})^{(N+1)^{2}}\Omega_{I}\wedge\bar\Omega_{I}.
	\end{align*}
	
	\subsection{Holomorphic Map from the Model Space}\label{subsectionholomorphicmap}
	In this section we study the complex manifold
	\[
	\bigl((\pi^{n})^{-1}(\mathcal{T}_{I}\setminus\D^{n}),J_{I}\bigr).
	\]
	Our goal is to prove that it is biholomorphic to
	\[
	(\mathbb{C}^{n+1}\setminus B_{R})\times\mathbb{R}^{2(N-n)}.
	\]
	Parallel to Section~2.4 of~\cite{li2023syz}, we shall construct an explicit holomorphic map from $(\pi^{n})^{-1}(\mathcal{T}_{I})$ into $\mathbb{C}^{n+1}\times\mathbb{R}^{2(N-n)}$; the previous surgeries have to be taken into account.  The explicit formulae will be useful in Subsection~\ref{subsectionlogarithmicgrowth}.  Other approaches to proving biholomorphicity can be found in~\cite{lebrun1991complete,zharkov2004limiting}.
	
	We first look for closed holomorphic $1$-forms.
	Since ${\rm d}\xi_{I,i}=0$ for every $i\in I^{\prime}$, it suffices to consider $\tilde\xi_{I,i}$ ($i\in I$) and ${\rm d}\eta$.  A direct computation gives
	\[
	{\rm d}\bigl(\tilde\xi_{i}+\gamma_{i}\,{\rm d}\eta\bigr)
	=\sqrt{-1}\Bigl[
	-\Bigl(\frac{\partial\gamma_{i}}{\partial\bar\eta}
	+\frac12\frac{\partial q_{I}}{\partial\mu_{i}}\Bigr)\,{\rm d}\eta\wedge{\rm d}\bar\eta
	+\Bigl(\frac{\partial\gamma_{i}}{\partial\mu_{j}}
	-2\frac{\partial p_{I,ij}}{\partial\eta}\Bigr)\,{\rm d}\mu_{j}\wedge{\rm d}\eta
	\Bigr].
	\]
	Because $p_{I,ij}$ and $q_{I}$ depend only on $(\mu_{I},\eta)$ and vanish whenever $i\notin I$ or $j\notin I$, we can choose functions $\gamma_{i}(\mu_{I},\eta)$ such that
	\begin{align}
		\frac{\partial\gamma_{i}}{\partial\mu_{j}}
		=2\frac{\partial p_{I,ij}}{\partial\eta},
		\qquad
		\frac{\partial\gamma_{i}}{\partial\bar\eta}
		=-\frac12\frac{\partial q_{I}}{\partial\mu_{i}},
		\qquad
		1\le i\le n.
	\end{align}
	Similarly, requiring
	$-\sum_{i=1}^{n}\tilde\xi_{i}+\gamma_{0}\,{\rm d}\eta$
	to be closed yields
	\begin{align}
		\frac{\partial\gamma_{0}}{\partial\mu_{j}}
		=-2\sum_{i=1}^{n}\frac{\partial p_{I,ij}}{\partial\eta},
		\qquad
		\frac{\partial\gamma_{0}}{\partial\bar\eta}
		=-\frac{n}{2}\frac{\partial q_{I}}{\partial\mu_{i}}.
	\end{align}
	The integrability conditions for both systems are satisfied.
	
	Henceforth we restrict to the $(n+2)$-dimensional domain $\mathcal{T}_{I,1}$.
	For fixed $i\in I$, the coefficient $p_{I,ij}$ is smooth on $\mathcal{T}_{I,1}\setminus\D^{n}_{i,1}$ because the action generated by $e_{i}$ and $e_{i}-e_{j}$ ($j\in J$) is non-degenerate there; the sum $\sum_{j=1}^{n}p_{I,ij}$ is smooth on $\mathcal{T}_{I,1}\setminus\D^{n}_{0,1}$ for the same reason.
	One verifies that $\D^{n}_{i}$ is an open Hausdorff-$n$-subset of $\mathcal{T}_{I,1}$ bounded by codimension-$3$ loci and that $\mathcal{T}_{I,1}\setminus\D^{n}_{i,1}$ is simply connected.  For $(\mu_{I},\eta)\in\mathcal{T}_{I,1}$ we can therefore define
	\begin{equation}\label{gammadefine}
		\begin{aligned}
			\gamma_{0}(\mu_{I},\eta)
			&=\lim_{\mu^{\prime}_{j}\to-\infty}
			-2\int_{(\mu^{\prime}_{1},\dots,\mu^{\prime}_{n})}^{(\mu_{1},\dots,\mu_{n})}
			\sum_{k=1}^{n}\sum_{i=1}^{n}
			\frac{\partial p_{I,ik}}{\partial\eta}(s_{1},\dots,s_{n},\eta)\,{\rm d}s_{k},\\[4pt]
			\gamma_{i}(\mu_{I},\eta)
			&=\lim_{\substack{\mu^{\prime}_{i}\to+\infty\\[1pt]
					\mu^{\prime}_{i}-\mu^{\prime}_{j}\to+\infty}}
			2\int_{(\mu^{\prime}_{1},\dots,\mu^{\prime}_{n})}^{(\mu_{1},\dots,\mu_{n})}
			\sum_{k=1}^{n}
			\frac{\partial p_{I,ik}}{\partial\eta}(s_{1},\dots,s_{n},\eta)\,{\rm d}s_{k}.
		\end{aligned}
	\end{equation}
	
	The integrals in~\eqref{gammadefine} are improper, we verify their convergence.
	By Proposition~\ref{inductionhyp} the surgeries occur only inside certain $\B_{J}$.
	Take $\gamma_{1}$ as an example: any surgery that changes the values of $p_{I,1j}$ (hence of $V_{1j}^{(1)}$) takes place in $\B_{J}$ with $1\in J$ and $|J|\ge 3$.
	According to the definition of $\B_{J}$, the limit
	\[
	\mu_{1}\to+\infty,\quad \mu_{1}-\mu_{j}\to+\infty
	\]
	moves away from all such $\B_{J}$.
	Consequently, in the region relevant for convergence we may replace $p_{I,1j}$ by the smooth function $v_{I,1j}$; the same argument applies to every $\gamma_{i}$.
	Because
	\[
	\|\nabla_{A}^{k}\alpha_{ij}\|
	\le\frac{C}{\operatorname{dist}_{A}^{k+1}(\,\cdot\,,\D^{n}_{ij})},
	\qquad 0\le i,j\le n,\; i\ne j,
	\]
	the improper integrals defining $\gamma_{i}$ converge uniformly.
	
	Therefore
	\begin{align*}
		\frac{\partial\gamma_{i}}{\partial\bar\eta}
		&=-\frac12\frac{\partial q_{I}}{\partial\mu_{i}}
		+\lim_{\mu^{\prime}_{i}-\mu^{\prime}_{j}\to+\infty}
		\lim_{\mu^{\prime}_{i}\to+\infty}
		\frac12\frac{\partial q_{I}}{\partial\mu_{i}},\\[4pt]
		\frac{\partial\gamma_{0}}{\partial\bar\eta}
		&=-\frac12\sum_{i=1}^{n}\frac{\partial q_{I}}{\partial\mu_{i}}
		+\lim_{\mu^{\prime}_{j}\to-\infty}
		\frac12\sum_{i=1}^{n}\frac{\partial q_{I}}{\partial\mu_{i}}.
	\end{align*}
	To evaluate the limits we again consider $\gamma_{1}$.
	The distributional equation satisfied by $P_{ij}$ and $Q$ on $\mathcal{T}_{I}$ gives
	\[
	\frac{\partial^{2}q_{I}}{\partial\mu_{1}\partial\mu_{j}}
	+4\frac{\partial^{2}p_{I,1j}}{\partial\eta\partial\bar\eta}=0,
	\]
	which holds whenever $\mu_{1}$ and $\mu_{1}-\mu_{j}$ are sufficiently large.
	A line-integral argument shows that
	\[
	\lim_{\substack{\mu_{1}\to+\infty\\ \mu_{1}-\mu_{j}\to+\infty}}
	\frac{\partial q_{I}}{\partial\mu_{1}}
	\quad\text{exists and depends only on }\eta.
	\]
	Fixing $\eta$ and observing that $\partial q_{I}/\partial\mu_{1}$ vanishes in the surgery-free region, we conclude that the limit is identically zero.
	Hence
	\[
	\frac{\partial\gamma_{i}}{\partial\bar\eta}
	=-\frac12\frac{\partial q_{I}}{\partial\mu_{i}},\qquad
	\frac{\partial\gamma_{0}}{\partial\bar\eta}
	=-\frac12\sum_{i=1}^{n}\frac{\partial q_{I}}{\partial\mu_{i}}.
	\]
	
	We now compute the value of $\sum\gamma_{i}$.
	By construction, each $\tilde\xi_{i}+\gamma_{i}\,{\rm d}\eta$ is a closed holomorphic $1$-form, hence so is their sum
	\[
	\Bigl(\sum_{i=0}^{n}\gamma_{i}\Bigr){\rm d}\eta=f\,{\rm d}\eta.
	\]
	Thus $f_{\mu_{i}}=0$ and $f_{\bar\eta}=0$, so $f=f(\eta)$ is holomorphic in $\eta$.
	To obtain the explicit form of $f$ we evaluate it for large $|\eta|$ (to avoid the cylindrical excisions that affect $\mathcal{T}_{I}$ when $|\eta|$ is small) and then appeal to the uniqueness of holomorphic functions.
	Since $f$ is independent of $\mu_{i}$, we fix $\eta$ and send the $\mu_{i}$ to suitable limits.
	
	Fix $\eta$ and let $\mu_{i}\to-\infty$, $\mu_{i}-\mu_{j}\to-\infty$ ($i<j$).
	The contribution of $\gamma_{0}$ to the limit is zero.
	For $1\le k\le n$ we consider the ray family
	\[
	l_{s}(t)=s\vec d+t\vec b,\qquad
	b_{i}<0,\; b_{i}-b_{j}<0\;(i<j),\qquad
	d_{k}>0,\; d_{k}-d_{l}>0\;(l\ne k).
	\]
	Because $\lim_{s\to+\infty}\gamma_{k}\circ l_{s}(t)=0$ for fixed $t$, we have
	\[
	\gamma_{k}\circ l_{s}(t_{0})-\gamma_{k}\circ l_{s}(1)
	=\int_{0}^{t_{0}}\frac{{\rm d}}{{\rm d}t}(\gamma_{k}\circ l_{s}(t))\,{\rm d}t.
	\]
	One verifies that, for sufficiently large $s$, the terms in the above limit are unaffected by the previous surgeries.
	Sending $t_{0}\to+\infty$ and $s\to+\infty$, and using the decay of
	$\sum_{j}\partial p_{I,kj}/\partial\eta$ and $\partial p_{I,kj}/\partial\eta$,
	the Lebesgue dominated-convergence theorem shows that the right-hand side tends to $0$ for $2\le k\le n$.
	When $k=1$ one term survives:
	\[
	\lim_{s\to+\infty}\int_{-\infty}^{0}2b_{1}\sum_{j=1}^{n}
	\frac{\partial p_{I,1j}}{\partial\eta}\circ l_{s}(t)\,{\rm d}t.
	\]
	For sufficiently large $s$ we have $p_{I,1j}=v_{I,1j}$, so the integral becomes
	\[
	\int_{-\infty}^{+\infty}
	\frac{|b_{1}|\det A_{\{2,\dots,n\}}\bar\eta}
	{4\bigl((b_{1}\mu_{1})^{2}+\det A_{\{2,\dots,n\}}|\eta|^{2}\bigr)^{3/2}}\,
	{\rm d}\mu_{1}
	=\frac{1}{\eta}.
	\]
	Hence
	\[
	\sum_{i=0}^{n}\gamma_{i}=\frac{1}{\eta},
	\qquad\text{so}\qquad
	\sum_{i=0}^{n}{\rm d}\log z_{i}={\rm d}\log\eta.
	\]
	Choosing the constant of integration appropriately, we obtain
	\[
	\prod_{i=0}^{n}z_{i}=\eta.
	\]
	The coordinates $z_{i}$ are initially defined on
	$\mathcal{T}_{I,1}\setminus\D^{n}_{i,1}$,
	the above relation allows us to extend $z_{i}$ continuously to zero on $\D^{n}_{i}$.
	We shall verify in the next subsection that $z_{i}$ also vanishes on the boundary of $\D^{n}_{i}$.

	\subsection{Smooth and Logarithmic Growth}\label{subsectionlogarithmicgrowth}
	Continuing with the notation of the previous section, we show that
	\[
	|z_{i}|\to 0\quad\text{as we approach }\partial\D^{n}_{i}.
	\]
	By the induction hypothesis (Proposition~\ref{inductionhyp}) and Proposition~\ref{prop:smooth-structure-DK}, the smooth structure near the relevant loci is already available. Hartogs's lemma then implies that the holomorphic structure $J_{1}$ is smooth.
	
	Fix $K\subset I$ with $|K|=k+1$.  Using the symmetry discussed in Subsection~\ref{subsectionthegeometry}, we may assume
	\[
	K=\{0,1,\dots,k\}.
	\]
	Introduce the Gibbons--Hawking coefficients
	\[
	P_{K,ij}=A_{ij}+p_{K,ij}=P_{I,ij}+h_{I,ij}-h_{K,ij},\qquad
	Q_{K}=\det A+q_{K}=Q_{I}+h_{I}-h_{K}.
	\]
	By Remark~\ref{HKdefine}, on $\H_{K}\cap\F_{n-1}$ these data coincide with $V_{K,ij}$ and $W_{K}$ defined in Subsection~\ref{subsectionsmoothstructure}; hence $P_{K,ij}$ and $Q_{K}$ induce a smooth structure of type $\mathbb{C}^{k+1}\times\mathbb{R}^{2(N-k)}$.
	
	\begin{remark}
		There is a minor difference compared with the earlier use of $V_{K,ij}$ and $W_{K}$: the previous GH structure depended on $\mu_{I^{\prime}}$, but inside the conical region $\B_{I}$ both $V_{K,ij}$ and $W_{K}$ are independent of $\mu_{I^{\prime}}$.
		Consequently, the structure induced by $P_{K,ij}$ and $Q_{K}$ is obtained by taking the ``conical limit'' in the opening direction of $\B_{I}$.
	\end{remark}
	
	On $\B^{n}$ we define, for every $J\subset I$ with $|J|\ge 2$, the open set
	\begin{align*}
		\B^{n}_{J}&=\bigl\{p\in\B^{n}\bigm|
		C_{0}\operatorname{dist}_{A}(p,\D^{n}_{J})<
		\operatorname{dist}_{A}(p,\partial\D^{n}_{J})\bigr\},\\
		\B^{n}_{a}&=\bigl\{p\in\B^{n} \bigm|
		2 C_{0}^{N-1}\operatorname{dist}_{A}(p,\D^{n})>
		\operatorname{dist}_{A}(p,\D^{n}_{I})\bigr\},
	\end{align*}
	and similarly introduce $(\B^{n}_{J})^{\prime}$, $(\B^{n}_{J})^{\prime\prime}$.
	Formally, when $I=\{i\}$, we set $\B^{n}_{i}=\D^{n}_{i}$ for $i=0,1,\dots,N$.
	For $K\subset I$ with $|K|\ge 2$, parallel to Remark~\ref{HKdefine}, we set
	\[
	\H^{n}_{K}=\mathcal{T}_{I}\cap \bigl(\B^{n}_{K}\setminus
	\bigcup_{K\subsetneqq J\subsetneqq I}(\B^{n}_{J})^{\prime}\bigr).
	\]
	Thus, on $\H^{n}_{K}$ we have
	$\operatorname{dist}_{A}(\,\cdot\,,\D^{n}_{J})\sim
	\operatorname{dist}_{A}(\,\cdot\,,\D^{n}_{I})$
	whenever $K\subsetneqq J\subsetneqq I$.
	Notice that
	$\mathcal{T}_{I}\setminus \B^{n}_{a}
	=\bigcup_{K\subsetneqq I}\H^{n}_{K}$.
	
	By the induction hypothesis we have a basis of holomorphic $(1,0)$-forms with respect to $J_{K}$:
	\[
	\xi_{K,i}=P_{K,ij}\,{\rm d}\mu_{j}+\sqrt{-1}\,\vartheta_{K,i},
	\]
	where the connection $1$-forms satisfy
	\[
	{\rm d}\vartheta_{K,i}=F_{K,i}=
	\sqrt{-1}\Bigl(\frac12\frac{\partial Q_{K}}{\partial\mu_{j}}\,
	{\rm d}\eta\wedge{\rm d}\bar\eta
	+\frac{\partial P_{K,ij}}{\partial\eta}\,{\rm d}\mu_{j}\wedge{\rm d}\eta
	-\frac{\partial P_{K,ij}}{\partial\bar\eta}\,
	{\rm d}\mu_{j}\wedge{\rm d}\bar\eta\Bigr).
	\]
	Consequently, on $\H^{n}_{K}$,
	\[
	\xi_{I,i}=\xi_{K,i}+\rho_{i},\qquad
	\rho_{i}=(h_{I,ij}-h_{K,ij})\,{\rm d}\mu_{j}
	+\sqrt{-1}(\vartheta_{I,i}-\vartheta_{K,i}).
	\]
	Since $h_{I,ij}-h_{K,ij}$ depends only on $(\mu_{I},\eta)$, it is a smooth harmonic function on $\H^{n}_{K}$.
	The closed forms
	\[
	{\rm d}\log w_{0}=\sum_{j=1}^{k}\tilde\xi_{K,j}+\gamma_{K,0}\,{\rm d}\eta,
	\qquad
	{\rm d}\log w_{i}=\tilde\xi_{K,i}+\gamma_{K,i}\,{\rm d}\eta,
	\quad i=1,\dots,k,
	\]
	are built from
	\[
	\tilde\xi_{K,K}=\xi_{K,K}-A_{KK^{\prime}}A_{K^{\prime}}^{-1}\xi_{K,K^{\prime}},
	\qquad
	\tilde\xi_{K,K^{\prime}}=\xi_{K,K^{\prime}}.
	\]
	
	For $i\in I^{\prime}$ both $\xi_{I,i}$ and $\xi_{K,i}$ are closed $1$-forms with the same real part $A_{ij}\,{\rm d}\mu_{j}$, their imaginary parts differ by a gauge transformation.
	We choose the gauge so that $\vartheta_{I,i}=\vartheta_{K,i}$, hence $\rho_{I^{\prime}}=0$.
	
	Put $J=I\setminus K$.  Since
	$G_{I}=A_{I}-A_{II^{\prime}}A_{I^{\prime}}^{-1}A_{I^{\prime}I}$, we have
	\[
	G_{I}=\begin{pmatrix}
		s_{K} & s_{KJ}\\[2pt]
		s_{JK} & s_{J}
	\end{pmatrix}.
	\]
	where
	\begin{align*}
		s_{K}&=A_{K}-A_{KI^{\prime}}A_{I^{\prime}}^{-1}A_{I^{\prime}K},&
		s_{KJ}&=A_{KJ}-A_{KI^{\prime}}A_{I^{\prime}}^{-1}A_{I^{\prime}J},\\[2pt]
		s_{J}&=A_{J}-A_{JI^{\prime}}A_{I^{\prime}}^{-1}A_{I^{\prime}J},&
		s_{JK}&=A_{JK}-A_{JI^{\prime}}A_{I^{\prime}}^{-1}A_{I^{\prime}K}.
	\end{align*}
	By definition,
	\[
	{\rm d}\log z_{J}=\xi_{I,J}-A_{JI^{\prime}}A_{I^{\prime}}^{-1}\xi_{I,I^{\prime}}
	+\gamma_{I,J}\,{\rm d}\eta.
	\]
	Hence, for every $j\in J$,
	\begin{align*}
		{\rm d}\log|z_{J}|=s_{JK}\,{\rm d}\mu_{K}+s_{J}\,{\rm d}\mu_{J}+p_{I,JK}\,{\rm d}\mu_{K}
		+\operatorname{Re}(\gamma_{I,J}\,{\rm d}\eta).
	\end{align*}
	Therefore, by Proposition~\ref{inductionhyp} (II), (III) and the definition of $\H^{n}_{K}$, on $\mathcal{T}_{I}\cap\H^{n}_{K}$ both $p_{I,jk}$ and $\gamma_{I,j}$ are smooth and satisfy
	\[
	|p_{I,jk}|\le C \operatorname{dist}_{A}^{-1}(\,\cdot\,,\D^{n}_{I}),\qquad
	|\gamma_{I,j}|\le C \operatorname{dist}_{A}^{-1}(\,\cdot\,,\D^{n}_{I}).
	\]
	Integrating these estimates gives
	\begin{equation}\label{duishuzengzhang1}
		|z_{j}|=\exp\bigl( s_{JK}\mu_{K}+s_{J}\mu_{J}+H_{J}\bigr)
		\qquad\text{on }\mathcal{T}_{I}\cap\H^{n}_{K},
	\end{equation}
	where
	\[
	|H_{j}|\le C\log\bigl(1+\operatorname{dist}_{A}(\,\cdot\,,\D^{n}_{I})\bigr), \quad j\in J
	\]
	with a uniform constant $C$.
	
	We next analyse $z_{K}$.  First,
	\begin{align*}
		{\rm d}\log z_{K}
		&=\xi_{I,K}-A_{KI^{\prime}}A_{I^{\prime}}^{-1}\xi_{I,I^{\prime}}
		+\gamma_{I,K}\,{\rm d}\eta\\
		&=\xi_{K,K}-A_{KI^{\prime}}A_{I^{\prime}}^{-1}\xi_{K,I^{\prime}}
		+\rho_{K}+(\gamma_{I,K}-\gamma_{K,K})\,{\rm d}\eta.
	\end{align*}
	Recall that
	\[
	{\rm d}\log w_{K}=\xi_{K,K}-A_{KK^{\prime}}A_{K^{\prime}}^{-1}\xi_{K,K^{\prime}}
	+\gamma_{K,K}\,{\rm d}\eta,
	\]
	and observe the block decomposition
	\[
	A_{KK^{\prime}}A_{K^{\prime}}^{-1}=
	\begin{pmatrix}
		s_{KJ}s_{J}^{-1} &
		-s_{KJ}s_{J}^{-1}A_{JI^{\prime}}A_{I^{\prime}}^{-1}
		+A_{KI^{\prime}}A_{I^{\prime}}^{-1}
	\end{pmatrix}.
	\]
	Consequently,
	\begin{equation*}
		\begin{aligned}
			{\rm d}\log z_{K}={}&{\rm d}\log w_{K}
			+\rho_{K}+(\gamma_{I,K}-\gamma_{K,K})\,{\rm d}\eta\\
			&+s_{KJ}s_{J}^{-1}\xi_{K,J}
			-s_{KJ}s_{J}^{-1}A_{JI^{\prime}}A_{I^{\prime}}^{-1}\xi_{K,I^{\prime}}.
		\end{aligned}
	\end{equation*}
	We obtain
	\begin{align*}
		{\rm d}\log|z_{K}|={\rm d}&\log|w_{K}|
		+ s_{KJ}\,{\rm d}\mu_{J}+s_{KJ}s_{J}^{-1}s_{JK}\,{\rm d}\mu_{K}\\*
		&+\operatorname{Re}\rho_{K}
		+\operatorname{Re}\bigl((\gamma_{I,K}-\gamma_{K,K})\,{\rm d}\eta\bigr).
	\end{align*}
	The right-hand side is a closed $1$-form, and $\rho_{K}$ is smooth on $\H^{n}_{K}$, hence $(\gamma_{I,K}-\gamma_{K,K})\,{\rm d}\eta$ has smooth exterior derivative and decays at the same rate as $\rho_{K}$.
	Thus, modulo a closed $1$-form of the form $f\,{\rm d}\eta$ with $f$ holomorphic in $\eta$, the difference $\rho_{K}+(\gamma_{I,K}-\gamma_{K,K})\,{\rm d}\eta$ is smooth.
	Because both ${\rm d}\log z_{K}$ and ${\rm d}\log w_{K}$ are partially defined at $\eta=0$, the function $f$ is smooth and holomorphic on the whole $\eta$-plane.
	Using
	\[
	\operatorname{Re}\rho_{k}=\sum_{j=1}^{n}(h_{I,kj}-h_{K,kj})\,{\rm d}\mu_{j}
	\]
	and the estimate
	\[
	|\alpha_{I,ij}-\alpha_{K,ij}|
	\le\frac{C}{\operatorname{dist}_{A}(\,\cdot\,,\D^{n}_{I})}
	\qquad\text{on }\H^{n}_{K}.
	\]
	We integrate as before, for any $k\in K$, $k\ne 0$, we obtain on $\mathcal{T}_{I}\cap\H^{n}_{K}$,
	\begin{equation}\label{duishuzengzhang2}
		\begin{aligned}
			|z_{K}|&=|w_{K}|\exp\bigl( s_{KJ}\mu_{J}+s_{KJ}s_{J}^{-1}s_{JK}\mu_{K}+H_{K}\bigr),\\
			|z_{0}|&=|w_{0}|\exp\Bigl(-\sum_{l=1}^{k}(s_{lJ}\mu_{J}
			+s_{lJ}s_{J}^{-1}s_{JK}\mu_{K})-\sum_{l=1}^{k} H_{l}\Bigr),
		\end{aligned}
	\end{equation}
	where
	\[
	|H_{k}|\le C\log\bigl(1+\operatorname{dist}_{A}(\,\cdot\,,\D^{n}_{I})\bigr),\quad k\in K
	\]
	with a uniform constant $C$.
	Since $|w_{K}|\to 0$ as we approach $\D^{n}_{K}$, \eqref{duishuzengzhang2} implies $|z_{K}|\to 0$.
	By Hartogs's lemma the map constructed above is smooth with respect to the original smooth structure on $\mathcal{T}_{I}$, the holomorphic structure $J_{1}$ is therefore smooth.
	The equivalence constants in \eqref{duishuzengzhang1} and \eqref{duishuzengzhang2} are uniform as long as we stay a definite distance away from $\D^{n}_{K}$, this fact will be used later.
	
	The following lemma is crucial in Subsection~\ref{subsctionsurgeryoriginal}.
	
	\begin{lemma}\label{zzengzhangxing}
		Let $|z|=\Bigl(\sum_{i=0}^{n}|z_{i}|^{2}\Bigr)^{1/2}$ and
		\[
		|\vec{\mu}_{I}|_{G_{I}}
		=\bigl(|\mu_{I}|^{2}+\det A\,|\eta|^{2}\bigr)^{1/2}
		=\Bigl(\sum_{1\le i,j\le n}G_{I,ij}\mu_{i}\mu_{j}
		+\det A\,|\eta|^{2}\Bigr)^{1/2}.
		\]
		If $|\vec{\mu}_{I}|_{G_{I}}>R$ in $\mathcal{T}_{I,1}$, then
		\begin{equation}\label{duishuzengzhang}
			K_{1}e^{K_{2}|\mu_{I}|}\,|\vec{\mu}_{I}|_{G_{I}}
			\le |z|\le
			K_{3}e^{K_{4}|\mu_{I}|}\,|\vec{\mu}_{I}|_{G_{I}},
		\end{equation}
		where $R>0$ and $K_{i}>0$ are uniform constants depending only on $n$, $\lambda$, and $\Lambda$.
	\end{lemma}
	
	\begin{proof}
		We first work on $\B^{n}_{a}$.  Fix $p=(\mu_{I},\eta)\in\B^{n}_{a}$ and set $|\vec{\mu}_{I}|_{G_{I}}=R_{0}>0$.  Assume $\sqrt{\det A}\,|\eta|\ge\epsilon R_{0}$ with a small constant $\epsilon>0$ to be chosen later.  Put $x_{i}=\sum_{l=1}^{n}G_{I,il}\mu_{l}$, then $x_{I}$ is an $n$-dimensional column vector.  As in \eqref{duishuzengzhang1} we have
		\[
		|z_{i}|\sim\exp(x_{i}+H_{i}),\quad i=0,1,\dots,n,
		\]
		where
		\[
		|H_{i}|\le C\log\bigl(1+|\vec{\mu}_{I}|_{G_{I}}\bigr).
		\]
		Taking $R_{0}$ sufficiently large, one readily verifies \eqref{zzengzhangxing}.
		
		\smallskip
		Next suppose $\sqrt{\det A}\,|\eta|<\epsilon R_{0}$.  Then
		\[
		\det A\,|\eta|^{2}\le\frac{\epsilon^{2}}{1-\epsilon^{2}}\,
		\operatorname{dist}_{A}^{2}(q,\D_{I}^{n}).
		\]
		Let $q=(\mu_{I},0)$ be the orthogonal projection of $p$ onto $\{\eta=0\}$.  By definition of $\B^{n}_{a}$,
		\[
		4C_{0}^{2(N-1)}\bigl(\operatorname{dist}_{A}^{2}(q,\D^{n})
		+\det A\,|\eta|^{2}\bigr)
		>\operatorname{dist}_{A}(q,\D^{n}_{I})+\det A\,|\eta|^{2}.
		\]
		Hence, choosing $\epsilon$ small enough, we obtain
		\begin{equation}\label{eqzengzhang1}
			C\operatorname{dist}_{A}(q,\D^{n})>\operatorname{dist}_{A}(q,\D^{n}_{I})
		\end{equation}
		with a uniform constant $C>1$.  Consequently $q$ lies in the interior of some $\D^{n}_{i}$, without loss of generality we take $i=0$.
		
		Since
		\[
		x_{I}^{\top}\mu_{I}=\mu_{I}^{\top}G_{I}\mu_{I}\ge(1-\epsilon^{2})R_{0}^{2},
		\]
		and \eqref{eqzengzhang1} cuts out finitely many disjoint conical regions in $\mathbb{R}^{n}_{\mu_{I}}$, the vector $x_{I}$ lies in the dual cone $\mathcal{C}^{*}$ of the cone $\mathcal{C}$ corresponding to the interior of $\D^{n}_{0}$.  Thus the largest component of the unit vector $\frac{x_{I}}{|x_{I}|}$ is bounded below by a uniform positive constant, say $\frac{x_{1}}{|x_{I}|}>C>0$.  Observing
		\[
		x_{I}^{\top}x_{I}=\mu_{I}^{\top}G_{I}^{2}\mu_{I}\ge C(\lambda,\Lambda)R_{0},
		\]
		we get $x_{1}>C R_{0}$.  By definition of $\B^{n}_{a}$ we have $\operatorname{dist}_{A}(\,\cdot\,,\D^{n})\sim\operatorname{dist}_{A}(\,\cdot\,,\D^{n}_{I})$, so under our current assumption
		\[
		|z_{i}|=\exp(x_{i}+H_{i}),\quad i\in I,\; i\ne 0,
		\]
		with
		\[
		|H_{i}|\le C\log\bigl(1+|\vec{\mu}_{I}|_{G_{I}}\bigr).
		\]
		For $R_{0}$ large enough this gives $|z|>|z_{1}|\ge\exp(C R_{0})$ with a uniform constant $C>0$, while $|z_{i}|\le\exp(C R_{0})$ for $i\ne 0$.  Since $|z_{0}|=\frac{|\eta|}{\prod_{i\ne 0}|z_{i}|}$, estimate \eqref{duishuzengzhang} follows immediately.
		
		Since $\mathcal{T}_{I}\setminus\B^{n}_{a}=\bigcup_{K\subsetneqq I}\H^{n}_{K}$, it suffices to prove the lemma on a single $\H^{n}_{K}$ with $K=\{0,1,\dots,k\}$.  For $w_{K}$, recall that Proposition~\ref{buchongassumation} shows that estimate~\eqref{duishuzengzhangK} yields the similar conclusion as~\eqref{duishuzengzhang}.
		
		Let $p_{K}$ be the foot of the perpendicular from $p$ to $\D^{n}_{K}$.  Then
		\[
		\operatorname{dist}_{A}^{2}(p,\D^{n}_{I})
		=\operatorname{dist}_{A}^{2}(p,p_{K})+\operatorname{dist}_{A}^{2}(p_{K},\D^{n}_{I}),
		\]
		so upper bounds on $|z|$ follow from \eqref{duishuzengzhang1} and \eqref{duishuzengzhang2}.  Set $\operatorname{dist}_{A}(p,\D^{n}_{I})=|\vec{\mu}_{I}|_{G_{I}}=R_{0}>0$.  If $\operatorname{dist}_{A}(p_{K},\D^{n}_{I})>\delta R_{0}$ for a small constant $\delta>0$ to be chosen, then \eqref{duishuzengzhang1} gives $|z|>|z_{J}|\ge\exp(CR_{0})$ exactly as in the case of $\B^{n}_{a}$.  Otherwise
		\[
		\operatorname{dist}_{A}(p,p_{K})
		>\sqrt{1-\delta^{2}}\,R_{0}
		>\frac{\sqrt{1-\delta^{2}}}{\delta}\operatorname{dist}_{A}(p_{K},\D^{n}_{I}).
		\]
		By \eqref{duishuzengzhang2} the exponential terms are uniformly controlled by $|w_{i}|$, hence we may choose $\delta>0$ sufficiently small (uniformly) and combine this with \eqref{duishuzengzhangK} to ensure $|z|>|z_{K}|\ge C|w_{K}|\ge\exp(CR_{0})$.  This completes the proof of the lemma.
	\end{proof}

	\subsection{The Image of the Holomorphic Map}\label{subsectiontheimageofthehol}
	Above we constructed a map from $\mathcal{T}_{I}$ to $\mathbb{C}^{n+1}\times\mathbb{R}^{2(N-n)}$, by the definition of $\mathcal{T}_{I}$ and $\nu_{I,I^{\prime}}$, the projection onto $\mathbb{R}^{2(N-n)}$ is bijective.
	In this section we prove that the projection onto $\mathbb{C}^{n+1}$ is injective and that its image is the whole space with a compact subset removed.
	
	We have the holomorphic bundle map
	\[
	\mathbf{F}\colon\mathcal{T}_{I,1}\longrightarrow\mathbb{C}^{n+1},\qquad
	\mathbf{F}=(z_{0},\dots,z_{n}),
	\]
	and we consider the holomorphic vector fields
	\[
	J(X_{i})+\sqrt{-1}X_{i},\qquad i\in I,
	\]
	where $X_{i}$ is the fundamental vector field generated by $e_{i}$.
	These fields preserve $\eta$ and act only in the fibre direction of $\mu_{I}$.
	Thus we obtain the holomorphic bundle diagram
	\[
	\begin{tikzcd}[column sep=large, row sep=large]
		\mathcal{T}_{I,1}
		\arrow[rr,"\mathbf{F}"]
		\arrow[dr,"\eta"']
		&&
		\mathbb{C}^{n+1}
		\arrow[dl,"\pi_{\eta}"]
		\\
		&
		\mathbb{C}
	\end{tikzcd}
	\]
	where $\pi_{\eta}=z_{0}\cdots z_{n}$.
	We shall show that $\mathbf{F}$ is a holomorphic bundle isomorphism by studying the holomorphic action generated by these vector fields.
	Since $X_{i}$ acts only along the fibre, the fibre-wise action is transparent.
	
	Let $E_{i}(t,\,\cdot\,)$ denote the flow generated by $-J(X_{i})$ ($i\in I$).
	In the remainder of this section all sums run from $1$ to $n$.
	The vector field $-J_{1}(X_{i})$ has the explicit expression
	\[
	-J_{1}(X_{i})=P^{-1}_{ij}\Bigl(\frac{\partial}{\partial\mu_{j}}
	-\tilde\vartheta_{I,k}\Bigl(\frac{\partial}{\partial\mu_{j}}\Bigr)
	X_{k}\Bigr).
	\]
	Since $\vartheta_{I,j}(J_{1}(X_{i}))=0$, the flow $E_{i}$ descends to an action on the base $\B^{n}$.
	The commutativity $[J_{1}(X_{i}),J_{1}(X_{j})]=0$ implies that the flows $E_{i}$ generate an $\mathbb{R}^{n}$-action, which we denote by $E$.
	On the base we have
	\[
	\frac{{\rm d}}{{\rm d}t}\mu_{i}\circ E_{j}=V^{-1}_{ij}.
	\]
	For fixed $\eta=\eta_{0}\ne 0$ define
	\[
	G\colon\mathcal{T}_{I,1}\cap\{\eta=\eta_{0}\}\longrightarrow\mathbb{R}^{n},\qquad
	\frac{\partial G_{i}}{\partial\mu_{j}}=V_{ij}(\mu_{I},\eta_{0}).
	\]
	The positive-definiteness and commutativity of $V_{ij}$ guarantee that the domain is simply connected, hence, by the Frobenius theorem, $G$ is well-defined.
	Choosing the constant of integration so that
	\[
	G\circ\mu_{I}\circ E=\operatorname{id},
	\]
	we conclude that $\mu_{I}\circ E$ is injective.
	
	When $\eta=0$ the same argument applies on $\mathcal{T}_{I,1}\cap\D^{n}_{i}$.
	On the locus $\D^{n}_{K}$ we consider the action generated by $-J(X_{j})$ for $j\in I\setminus K$ and again obtain injectivity of $\mu_{I}\circ E$.
	
	\begin{remark}
		The existence of $G$ relies essentially on the positive-definiteness and commutativity of the matrix $V_{ij}$.
	\end{remark}
	
	On the other hand, the definition of the $z_{i}$ gives
	\begin{align*}
		\mathcal{L}_{X_{j}}z_{i}
		&={\rm d}z_{i}(X_{j})=z_{i}\tilde\xi_{I,i}(X_{j})
		=\sqrt{-1}z_{i}\delta_{ij},\qquad i\ne 0,\\[2pt]
		\mathcal{L}_{X_{j}}z_{0}
		&={\rm d}z_{0}(X_{j})
		=z_{0}\Bigl(-\sum_{i=1}^{n}\tilde\xi_{I,i}\Bigr)(X_{j})
		=-\sqrt{-1}z_{0}.
	\end{align*}
	Hence $\mathbf{F}$ intertwines the singular $\mathbb{T}^{n}$-action generated by $e_{I}$ with the standard diagonal $T^{n}$-action on $\mathbb{C}^{n+1}$.
	Thus, the $\mathbb{C}^{n}$-action corresponding to $\mathbf{F}_{*}(-J(X_{i}))$ is
	\[
	(t_{1},\dots,t_{n})\cdot(z_{0},\dots,z_{n})
	=(e^{-\sum t_{i}}z_{0},e^{t_{1}}z_{1},\dots,e^{t_{n}}z_{n}).
	\]
	Consequently, the $\mathbb{C}^{n}$-action generated by $J(X_{i})+\sqrt{-1}X_{i}$ is transitive on each fibre $\eta=\eta_{0}$ in $\mathbb{C}^{n+1}$.
	This also shows that the flow generated by $E_{i}$ is complete: indeed, estimate~\eqref{duishuzengzhang} confines the flow to a compact set within any finite time interval.
	Observe that the pairs $X_i, J(X_j)$ commute and that the $X_i$ are mutually orthogonal.
	Together with estimate~\eqref{duishuzengzhang}, this implies that the image of $\mathbf{F}$ is $\mathbb{C}^{n+1}$ with a large ball $B_R$ removed.
	
	We also observe that $\mathbf{F}$ pushes the holomorphic volume form forward to the standard one on $\mathbb{C}^{n+1}$.
	The holomorphic volume form $\Omega$ is uniquely determined by
	\[
	\Omega_{1}\Bigl(X_{1},\dots,
	X_{n},\cdot\Bigr)={\rm d}\eta.
	\]
	But
	\begin{align*}
		(\sqrt{-1})^{n}{\rm d}z_{0}\wedge\dots\wedge{\rm d}z_{n}
		\Bigl(X_{1},\dots,
		X_{n},\cdot\Bigr)
		&=\sum_{i=0}^{n}z_{0}\dots\widehat{z_{i}}\dots z_{n}\,{\rm d}z_{i}={\rm d}\eta,
	\end{align*}
	where the first equality follows from the functional equation $\prod_{i=0}^{n}z_{i}=\eta$.
	Hence
	\[
	\Omega=-{\rm d}z_{0}\wedge\dots\wedge{\rm d}z_{n},
	\]
	and $\mathbf{F}$ is locally biholomorphic.
	We thus obtain the following statement.
	
	\begin{proposition}\label{smoothstructurenearBI}
		There exists a smooth injective bundle map $\mathbf{F}_{I}$ fitting into the commutative diagram
		\[
		\begin{tikzcd}[column sep=large, row sep=large]
			(\pi^{n})^{-1}(\mathcal{T}_{I})
			\arrow[rr,"\mathbf{F}_{I}"]
			\arrow[dr,"\eta"']
			&&
			\mathbb{C}^{n+1}\times \mathbb{R}^{2(N-n)}
			\arrow[dl,"\pi^{n}_{\eta}"]
			\\
			&
			\mathbb{C}
		\end{tikzcd}
		\]
		where $\pi^{n}_{\eta}$ is the projection defined in Example~\ref{biaozhunmuxing}.
		The map is given explicitly by
		\[
		\mathbf{F}_{I}=(z_{0},\dots,z_{n},\nu_{I,I^{\prime}},\theta_{I,I^{\prime}}).
		\]
		It is holomorphic on $(\pi^{n})^{-1}(\mathcal{T}_{I})$, and its image is
		\[
		(\mathbb{C}^{n+1}\setminus B_{R})\times\mathbb{R}^{2(N-n)},
		\]
		where $B_{R}\subset\mathbb{C}^{n+1}$ is the Euclidean ball of radius $R$ centered at the origin.
		Moreover, $\mathbf{F}_{I}$ intertwines the $\mathbb{T}^{n}\times\mathbb{R}^{N-n}$ action generated by $e_{I,I}$ and $e_{I,I^{\prime}}$ with the standard $\mathbb{T}^{n}\times\mathbb{R}^{N-n}$ action on $\mathbb{C}^{n+1}\times\mathbb{R}^{2(N-n)}$, and satisfies
		\[
		\mathbf{F}_{I}^{*}\Omega_{\text{Nor}}=\Omega_{I}.
		\]
	\end{proposition}

	\section{Surgery near the Discriminant Locus}\label{sectionsurgery}	
	
	\subsection{Smooth Extension of \texorpdfstring{$g_{1}$}{g1}}\label{subsctionsurgeryoriginal}
	
	By Proposition~\ref{smoothstructurenearBI}, the metric $g_{1}$ is a smooth \kah metric on $\C^{n+1}\setminus B_{R}$.
	We now construct a smooth \kah metric $\tilde g_{1}$ on the whole $\C^{n+1}$ such that $g_{1}=\tilde g_{1}$ outside a larger ball $B_{R_{1}}$ with $R_{1}\gg R$.
	
	For notational simplicity we drop the subscript $G_{I}$ and write
	\[
	|\vec\mu_{I}|=|\vec\mu_{I}|_{G_{I}};
	\]
	this is a smooth $\mathbb{T}^{n}$-invariant function on $\C^{n+1}\setminus B_{R}$.
	One verifies the uniform bound,	the following inequality can be readily obtained from the discussion in Subsection~\ref{subsectionasymptoticprop}.
	\begin{equation}\label{vecmuIbudengshi}
		\bigl|\mathrm{d}|\vec\mu_{I}|\bigr|_{g_{1}}
		+|\vec\mu_{I}|\cdot\bigl| dd^{c}|\vec\mu_{I}|\bigr|_{g_{1}}
		\le C.
	\end{equation}
	
	\begin{lemma}
		There exists a smooth \kah metric $\tilde g_{1}$ on $\C^{n+1}$ satisfying $\tilde g_{1}=g_{1}$ on $\C^{n+1}\setminus B_{2R_{1}}$ for some $R_{1}\gg R$.
	\end{lemma}
	
	\begin{proof}
		Since $\omega_{1}$ is a smooth $\mathbb{T}^{n}$-invariant \kah form on $\C^{n+1}\setminus B_{R}$, the $\partial\bar\partial$-lemma provides a $\mathbb{T}^{n}$-invariant function $\varphi$ on the same domain such that $\omega_{1}= dd^{c}\varphi$.
		Fix a radial cut-off function $\chi^{\prime}\in C^{\infty}_{c}(\C^{n+1})$ with
		\[
		\chi^{\prime}\equiv 1\ \text{on }\C^{n+1}\setminus B_{3R},\qquad
		\operatorname{supp}\chi^{\prime}\subset \C^{n+1}\setminus B_{2R}.
		\]
		Then the $(1,1)$-form $ dd^{c}(\chi^{\prime}\varphi)$ is supported in $\C^{n+1}\setminus B_{2R}$ and agrees with $\omega_{1}$ outside $B_{3R}$.
		On the annulus $B_{3R}\setminus B_{2R}$ the positivity may be lost, however, we still have
		\[
		 dd^{c}(\chi^{\prime}\varphi)\ge -\frac{K}{2}\omega_{0},
		\]
		where $\omega_{0}$ is the standard flat \kah form on $\C^{n+1}$ and $K>0$ is a uniform constant.
		
		Next we construct a smooth closed $\mathbb{T}^{n}$-invariant $(1,1)$-form $\omega$ on $\C^{n+1}$ satisfying
		\[
		\omega\ge K\omega_{0}\quad\text{on }B_{3R}
		\]
		and $\operatorname{supp}\omega\subset B_{2R_{1}}$ for some $R_{1}>4R$, while
		\[
		\omega+\omega_{1}>0 \quad\text{on }B_{2R_{1}}\setminus B_{4R}.
		\]
		Setting
		\[
		\tilde\omega=\omega+ dd^{c}(\chi^{\prime}\varphi),
		\]
		we obtain the desired \kah form on $\C^{n+1}$.
		
		Let $t=|z|^{2}$. For a radially symmetric function $f(t)$ (hence $f$ is $\mathbb{T}^{n}$-invariant), one computes
		\begin{align*}
			 dd^{c} f ={}& \sqrt{-1}\,\sum_{i}f^{\prime}(t)\,\mathrm{d} z_{i}\wedge d\bar z_{i}
			+\sqrt{-1}\,\sum_{i,j}f^{\prime\prime}(t)\bar z_{i}z_{j}\,\mathrm{d} z_{i}\wedge d\bar z_{j},
		\end{align*}
		where the eigenvalues of the above $(1,1)$-form are $f^{\prime}(t)$ and $f^{\prime}(t)+t f^{\prime\prime}(t)$.
		
		Let $\eta(s)$ be a cutoff function defined by
		\[
		\eta(s)=
		\begin{cases}
			1, & s\in[0,1),\\[2mm]
			0, & s\in[2,+\infty).
		\end{cases}
		\]
		Then $\eta^{\prime}$ and $\eta^{\prime\prime}$ are uniformly bounded.
		Let $M=3R$ be a positive constant and $R_{1}>M+1$ a parameter to be determined.
		Set $\chi(z)=\eta\Bigl(\dfrac{|\vec\mu_{I}|}{R_{1}}\Bigr)$ and define
		\[
		\omega= dd^{c}(\chi^{2}f).
		\]
		Then $\omega$ is $\mathbb{T}^{n}$-invariant and satisfies
		\begin{align*}
			\omega={}& \chi^{2} dd^{c} f+f dd^{c}\chi^{2}
			+\sqrt{-1}\chi\,\p\chi\wedge\bar\p f
			+\sqrt{-1}\chi\,\p f\wedge\bar\p\chi.
		\end{align*}
		Note that
		\[
		\p\chi=\frac{1}{R_{1}}\eta^{\prime}\Bigl(\frac{|\vec\mu_{I}|}{R_{1}}\Bigr)\p|\vec\mu_{I}|,
		\qquad
		\p f=f^{\prime}(t)\bar z_{i}\,\mathrm{d} z_{i}.
		\]
		By the Cauchy inequality we therefore obtain
		\begin{align*}
			\omega\ge{}& \chi^{2}\Bigl( dd^{c} f
			-\frac{(f^{\prime}(t))^{2}}{R_{1}^{2(1-\epsilon)}}
			z_{j}\bar z_{i}\,\mathrm{d} z_{i}\wedge d\bar z_{j}\Bigr)\\[2mm]
			& -\frac{\Bigl(\eta^{\prime}\bigl(\tfrac{|\vec\mu_{I}|}{R_{1}}\bigr)\Bigr)^{2}}{R_{1}^{2\epsilon}}
			\sqrt{-1}\p|\vec\mu_{I}|\wedge\bar\p|\vec\mu_{I}|
			+f dd^{c}\chi^{2}.
		\end{align*}
		
		On the decay region of $\chi$, where $R_{1}<|\vec{\mu}_{I}|<2R_{1}$, if $f$ satisfies the growth condition
		\begin{equation}\label{fcond1}
			|\vec{\mu}_{I}|^{-2}|f|\ll 1,
		\end{equation}
		then, by virtue of \eqref{vecmuIbudengshi}, the last two terms in the lower bound for $\omega$ can be absorbed by $\omega_{1}$ once $R_{1}$ is chosen sufficiently large.
		
		To control the first term we require $ dd^{c} f>0$ and, more importantly,
		\begin{equation}\label{fcond2}
			(t f^{\prime})^{\prime}
			-\frac{1}{t}\Bigl(\frac{1}{|\vec\mu_{I}|}\Bigr)^{2(1-\epsilon)}(t f^{\prime})^{2}>0.
		\end{equation}
		By~\eqref{duishuzengzhang} there exists a uniform constant $C$ such that
		\begin{equation}\label{shuajianduishuzengzhang}
			\log t\le C|\vec\mu_{I}|,
		\end{equation}
		so that \eqref{fcond2} becomes
		\[
		(t f^{\prime})^{\prime}
		-\frac{C}{t}\Bigl(\frac{1}{\log t}\Bigr)^{2(1-\epsilon)}(t f^{\prime})^{2}>0.
		\]
		
		Set $H=t f^{\prime}$ and $H^{\prime}=h$.
		We construct $h$ as follows:
		\[
		h(t)=
		\begin{cases}
			K, & 0\le t<M,\\[2mm]
			\dfrac{2\log 2\cdot K}{(t-M+2)\log(t-M+2)}, & t\ge M,
		\end{cases}
		\]
		and smooth $h$ on $(M-1,M+1)$ without affecting the main conclusions.
		Hence
		\[
		H(t)=
		\begin{cases}
			Kt, & 0\le t\le M-1,\\[2mm]
			KM+2\log 2\cdot K\log\log(t-M+2), & t\ge M+1.
		\end{cases}
		\]
		
		Because $ dd^{c} f>0$ by construction, inequality~\eqref{fcond2} is equivalent to
		\[
		\begin{aligned}
			-\frac{C}{t}\Bigl(\frac{1}{\log t}\Bigr)^{2(1-\epsilon)}
			&\bigl(KM+2\log 2\cdot K\log\log(t-M+2)\bigr)^{2}\\[2mm]
			&+\frac{2\log 2\cdot K}{(t-M+2)\log(t-M+2)}>0
		\end{aligned}
		\]
		on the region $R_{1}<|\vec\mu_{I}|<2R_{1}$.
		Since $t$ grows at least linearly in $|\vec\mu_{I}|$ (again by~\eqref{duishuzengzhang}), the left-hand side is positive for $R_{1}$ sufficiently large.
		
		For $t>M+1$ we integrate to obtain
		\[
		f(t)=KM\log t+2\log 2\cdot K\log(t-M+2)\bigl(\log\log(t-M+2)-1\bigr)+C.
		\]
		Thus $|f|$ is controlled by $\log t$ on the decay region, and~\eqref{fcond1} is satisfied when $R_{1}$ is taken large enough.  The lemma is proved.
	\end{proof}
	
	\begin{remark}
		The very existence of the metric $\tilde{g}_{1}$ relies on Lemma~\ref{zzengzhangxing}, and especially on estimate~\eqref{duishuzengzhang}.
		Since the latter is independent of the metric, the same smooth extension argument can be carried out for any $\mathbb{T}^{n}$-invariant \kah metric on $\mathbb{C}^{n+1}$ for which~\eqref{vecmuIbudengshi} holds.
	\end{remark}
	
	Consider the moment map of \(\tilde{\omega}_{1}\).
	On \(\mathbb{C}^{n+1}\setminus B_{2R_{1}}\) we have \(\tilde{\omega}_{1}=\omega_{1}\), hence \(\mathrm{d}\tilde{\mu}_{i}=\mathrm{d}\mu_{i}\).
	Because \(\mathbb{C}^{n+1}\setminus B_{2R_{1}}\) is simply connected, we can normalize the constants so that
	\(\tilde{\mu}_{i}=\mu_{i}\) on this region.
	Our gluing construction does not change the holomorphic volume form, so $\eta$ remains the same.
	For the $\mathbb{T}^{n}$-action on $\mathbb{C}^{n+1}$ the induced moment map is necessarily injective, since the action is transitive, the argument in Section~\ref{subsectiontheimageofthehol} therefore applies verbatim.
	
	Now consider the continuous map
	\(\widetilde{MP}=(\tilde{\mu}_{I},\eta)\colon\mathbb{C}^{n+1}\longrightarrow\mathbb{R}^{n}\times\mathbb{C}\).
	By the preceding remarks $\widetilde{MP}$ is injective, and it agrees with the map $MP$ associated with $\omega_{1}$ outside a compact set.
	Since $MP$ maps $\mathbb{C}^{n+1}\setminus B_{2R_{1}}$ bijectively onto the exterior region in $\mathbb{R}^{n}\times\mathbb{C}$, we conclude that $\widetilde{MP}$ is a bijection and that its discriminant locus coincides with that of $MP$.
	
	Observe that the GH coefficients are given by
	\[
	\tilde{P}^{-1}_{ij}=\tilde{\omega}\!\left(X_{i},JX_{j}\right).
	\]
	Consequently, on the region of $\mathbb{R}^{n}\times\mathbb{C}$ that corresponds to $\mathbb{C}^{n+1}\setminus B_{2R_{1}}$, we have $\tilde{P}_{ij}=P_{ij}$; likewise $\tilde{Q}=Q$.
	In other words, we have ``extended'' the GH coefficients $P_{ij}$ and $Q$ from the original domain to the whole ball $\mathcal{B}^{n}_{1}=\mathbb{R}^{n}\times\mathbb{C}$. For simplicity we continue to denote these extended coefficients by $P_{ij}$ and $Q$.
	
	Similarly, for the \kah metric $\tilde{\omega}+\omega_{2}$ on $\C^{n+1}\times \R^{2(N-n)}$, we obtain new coefficients $P_{I,ij}$ and $Q_{I}$ defined on all of $\mathcal{B}^{n}$.
	The GH structure induced by $P_{I,ij}$ and $Q_{I}$ is isomorphic to $\mathbb{C}^{n+1}\times\mathbb{R}^{2(N-n)}$, as dictated by the direct-product property.
	Clearly, the new GH coefficients $P_{I,ij}$ and $Q_{I}$ satisfy the commutativity condition~\eqref{assumationjiaohuan} and the distributional equation~\eqref{assumationcherncon}, and depend only on $\mu_{I}$ and $\eta$.
	If the image of\/ $\mathbb{C}^{n+1}\setminus B_{2R_{1}}$ under $\widetilde{MP}$ is contained in the subset
	\[
	\bigl\{q\in\mathcal{B}^{n}_{1}\mid \dist_{G_{I}}(q,O)>C\bigr\},
	\]
	then the preimage of this subset under the map $(\pi^{1})^{-1}$ in $\mathcal{B}^{n}$ is
	\[
	\bigl\{p\in\mathcal{B}^{n}\mid \dist_{A}(p,\D_{I})>C\bigr\}.
	\]
	This shows that we have extended the GH coefficients $P_{I,ij}$ and $Q_{I}$ within a ``cylindrical'' region of\/ $\mathcal{B}^{n}$.

	\subsection{Asymptotic Properties of  \texorpdfstring{$g_{I}$}{g(I)}}\label{subsectionasymptoticprop}
	For the smooth \kah metric \( g_{I} \) on \( \C^{n+1}\times \R^{2(N-n)} \), we show in this section that on various regions it can be approximated by the corresponding model metrics.
	To facilitate later references to PDE results, and in the spirit of \cite{li2023syz}, we also introduce global weighted Sobolev norms on \( \C^{n+1}\times \R^{2(N-n)} \).
	
	To streamline the discussion, we set the weight functions
	\begin{equation*}
		\ell_{i}=1+
		\dist_{A}\mkern-2mu
		\Bigl(\mspace{2mu}\cdot\mspace{4mu},
		\smash[b]{%
			\bigcup_{\substack{%
					\scriptstyle J\subseteq I\\[-1pt]
					\scriptstyle\lvert J\rvert=i+1}}%
			\D^{n}_{J}%
		}\Bigr),
		\qquad 1\le i\le n.
	\end{equation*}
	so that
	\begin{equation*}
		\ell_{1}\le \ell_{2}\le \dots \le \ell_{n}.
	\end{equation*}
	Let $\rho>1$ be a weight decay function, in our setting $\rho$ will typically be a negative power of $\ell_{i}$.
	For any $\mathbb{T}^{n}\times \R^{N-n}$-invariant tensor field $T$ we define the weighted Sobolev norm $\|\cdot\|_{C^{k,\alpha}}$ on $\C^{n+1}\times \R^{2(N-n)}$ in the usual way.
	
	On the singular bundle
	\[
	\pi^{n}\colon (\pi^{n})^{-1}(\B^{n}\setminus \D^{n})
	\longrightarrow \B^{n}\setminus \D^{n},
	\]
	endowed with the GH parameters $P_{I,ij}$ and $Q_{I}$, the discussion in the previous section shows that the induced GH structure
	\[
	\bigl((\pi^{n})^{-1}(\B^{n}\setminus \D^{n}),\;
	J_{I},\Omega_{I},g_{I},\omega_{I}\bigr)
	\]
	can be compactified to a smooth \kah structure
	\[
	(\C^{n+1}\times \R^{2(N-n)},\;
	J_{\mathrm{Nor}},\Omega_{\mathrm{Nor}},g_{I},\omega_{I}).
	\]
	
	Fix $R>0$ and define the region
	\[
	\T_{R}=\bigl\{p\in \B^{n}\mid
	\dist_{A}(p,\D^{n}_{I})>R\bigr\}.
	\]
	By the final discussion in Subsection~\ref{subsctionsurgeryoriginal}, there exists an $R_{1}>0$ such that inside $\T_{2R_{1}}$ the coefficients $P_{I,ij}$ and $Q_{I}$ agree with their smooth extensions.
	On $\B^{n}\setminus \T_{3R_{1}}$ the metric splits as $g_{I}=g_{1}+g_{2}$, where $g_{1}$ is a smooth metric on $\C^{n+1}$ and
	\[
	g_{2}=A_{I^{\prime},ij}\,\mathrm{d}\nu_{I,i}\otimes \mathrm{d}\nu_{I,j}
	+A^{-1}_{I^{\prime},ij}\,\vartheta_{I,i}\otimes \vartheta_{I,j}.
	\]
	Hence on $\B^{n}\setminus \T_{3R_{1}}$ we define $\|\cdot\|_{C^{k,\alpha}}$ to be the standard Sobolev norm $\|\cdot\|_{C^{k,\alpha}_{\mathrm{Nor}}}$.

	We now restrict our discussion to the region $\T_{2R_{1}}$.
	First, on $\T_{2R_{1}}\cap\H^{n}_{K}$ with $|K|=k+1$, we input the GH coefficients
	\[
	P_{K,ij}=A_{ij}+p_{K,ij}=P_{I,ij}+h_{I,ij}-h_{K,ij},\quad
	Q_{K}=\det A+q_{K}=Q_{I}+h_{I}-h_{K}.
	\]
	Let $g_{K}$ denote the induced GH metric.
	Proposition~\ref{prop:smooth-structure-DK} tells us that $P_{K,ij}$ and $Q_{K}$ induce a smooth structure of type $\C^{k+1}\times\R^{2(N-k)}$, and $g_{K}$ is smooth with respect to this structure.
	As described in Subsection~\ref{subsectionsmoothstructure}, $g_{K}$ can actually be defined on the larger region $\T_{2R_{1}}\cap\B^{n}_{K}$ by using $V^{(k)}_{ij}-h_{K,ij}$ and $W^{(k)}-h_{K}$.
	In particular, on $\H^{n}_{K}$ the metric obtained this way coincides with the one defined by the above $P_{K,ij}$ and $Q_{K}$.
	Therefore, in what follows we regard $g_{K}$ as a smooth metric on $\pi^{-1}(\T_{2R_{1}}\cap\B^{n}_{K})$.
	
	In Subsection~\ref{subsectionsmoothstructure} we have already defined the weighted Sobolev norm determined by $g_{K}$.
	We introduce the $\rho$-weighted Sobolev norm with respect to $g_{K}$.
	To simplify notation, for any $\mathbb{T}^{n}\times\R^{N-n}$-invariant tensor field $T$ on $\pi^{-1}(\T_{2R_{1}}\cap\B^{n}_{K})$ we write symbolically
	\[
	\|T\|_{C^{k,\alpha}_{g_{K}}}\le C\rho
	\]
	to mean
	\[
	\|T\|_{C^{k,\alpha}_{g_{K}}}
	=\sum_{j=0}^{k}\bigl\|\rho^{-1}\ell_{1}^{j}\nabla_{g_{K}}^{j}T\bigr\|_{L^{\infty}}
	+\bigl[\rho^{-1}\ell_{1}^{k}\nabla_{g_{K}}^{k}T\bigr]_{\alpha}\le C.
	\]
	
	On $\H^{n}_{K}$ the connection 1-forms satisfy $\mathrm{d}\vartheta_{K,I^{\prime}}=\mathrm{d}\vartheta_{I,I^{\prime}}=0$, so we may choose $\vartheta_{K,I^{\prime}}=\vartheta_{I,I^{\prime}}$.
	In general, since
	\[
	\mathrm{d}(\vartheta_{I,i}-\vartheta_{K,i})=F_{I,i}-F_{K,i},
	\]
	using the expressions \eqref{GibbonsHawkingcurvature} for $F_{I}$ and $F_{K}$ we obtain
	\begin{align*}
		F_{I,i}-F_{K,i}
		=\sqrt{-1}&\biggl(
		\frac{1}{2}\frac{\partial(h_{I}-h_{K})}{\partial\mu_{j}}\,\mathrm{d}\eta\wedge \mathrm{d}\bar{\eta}\\[1mm]
		&{}+\frac{\partial(h_{I,ij}-h_{K,ij})}{\partial\eta}\,\mathrm{d}\mu_{j}\wedge \mathrm{d}\eta
		-\frac{\partial(h_{I,ij}-h_{K,ij})}{\partial\bar{\eta}}\,\mathrm{d}\mu_{j}\wedge \mathrm{d}\bar{\eta}
		\biggr).
	\end{align*}
	Both $h_{I,ij}-h_{K,ij}$ and $h_{I}-h_{K}$ are smooth harmonic functions on $\H^{n}_{K}$, Lemma~\ref{alphajianjin} provides their size estimates.
	By Proposition~\ref{buchongassumation} and the relations
	\[
	\mathrm{d}\mu_{i}=-\iota_{X_{i}}\omega_{K},\qquad
	\mathrm{d}\eta=(\sqrt{-1})^{N}\Omega_{K}(X_{1},\dots,X_{N},\,\cdot\,),
	\]
	we have
	\[
	\|\mathrm{d}\mu_{i}\|_{C^{k,\alpha}_{g_{K}}}\le C,\qquad
	\|\mathrm{d}\eta\|_{C^{k,\alpha}_{g_{K}}}\le C.
	\]
	Notice that on $\H^{n}_{K}$
	\[
	\dist_{A}(\,\cdot\,,\partial\D_{K})\sim\dist_{A}(\,\cdot\,,\D^{n}_{I}).
	\]
	Hence
	\[
	\|F_{I,i}-F_{K,i}\|_{C^{k,\alpha}_{g_{K}}}\le C\ell_{n}^{-2}.
	\]
	Applying the Poincar\'e lemma we can choose $\vartheta_{K}$ so that
	\[
	\|\vartheta_{I,i}-\vartheta_{K,i}\|_{C^{k,\alpha}_{g_{K}}}\le C\ell_{n}^{-1}.
	\]
	
	\begin{proposition}\label{devationK}
		On $\T_{2R_{1}}\cap\H^{n}_{K}$ we have the deviation estimates
		\[
		\begin{cases}
			\|g_{I}-g_{K}\|_{C^{k,\alpha}_{g_{K}}}\le\ell_{n}^{-1},&
			\|\omega_{I}-\omega_{K}\|_{C^{k,\alpha}_{g_{K}}}\le\ell_{n}^{-1},\\[1mm]
			\|J_{I}-J_{K}\|_{C^{k,\alpha}_{g_{K}}}\le\ell_{n}^{-1},&
			\|\Omega_{I}-\Omega_{K}\|_{C^{k,\alpha}_{g_{K}}}\le\ell_{n}^{-1}.
		\end{cases}
		\]
		In particular, if $R_{1}$ is chosen sufficiently large, then
		\[
		|g_{I}-g_{K}|\ll 1,\quad
		|\omega_{I}-\omega_{K}|\ll 1,\quad
		|\Omega_{I}-\Omega_{K}|\ll 1,\quad
		|J_{I}-J_{K}|\ll 1.
		\]
	\end{proposition}
	
	\begin{proof}
		We illustrate the estimate for $J_{I}-J_{K}$, the others are similar.
		Away from the singular locus we choose a basis of holomorphic $1$-forms
		\[
		P_{K,ij}\,\mathrm{d}\mu_{j}+\sqrt{-1}\,\vartheta_{K,i},\qquad \mathrm{d}\eta,
		\]
		whose dual basis is
		\[
		\frac{1}{2}\bigl(P_{K,ij}^{-1}E_{K,\mu_{j}}-\sqrt{-1}\,X_{i}\bigr),\qquad E_{K,\eta},
		\]
		where
		\[
		E_{K,\mu_{i}}=\frac{\partial}{\partial\mu_{i}}
		-\vartheta_{K,j}\Bigl(\frac{\partial}{\partial\mu_{i}}\Bigr)X_{j},\qquad
		E_{K,\eta}=\frac{\partial}{\partial\eta}
		-\vartheta_{K,j}\Bigl(\frac{\partial}{\partial\eta}\Bigr)X_{j}.
		\]
		Thus
		\[
		J_{K}=\sqrt{-1}\,(\mathrm{d}\eta\otimes E_{K,\eta}-\mathrm{d}\bar{\eta}\otimes E_{K,\bar{\eta}})
		+P_{K,ij}\,\mathrm{d}\mu_{i}\otimes X_{j}
		-P_{K,ij}^{-1}\,\vartheta_{K,i}\otimes E_{K,\mu_{j}}.
		\]
		An identical expression holds with $I$ in place of $K$.
		Hence
		\[
		E_{I,\eta}-E_{K,\eta}
		=-\bigl((\vartheta_{I,i}-\vartheta_{K,i})(\partial/\partial\eta)\bigr)X_{i},
		\]
		and the first two terms in $J_{I}-J_{K}$ decay at the required rate.
		For the last term we write
		\begin{align*}
			&P_{I,ij}^{-1}\,\vartheta_{I,i}\otimes E_{I,\mu_{j}}
			-P_{K,ij}^{-1}\,\vartheta_{K,i}\otimes E_{K,\mu_{j}}\\[1mm]
			&\quad=P_{K,ij}^{-1}\,\vartheta_{K,i}\otimes(E_{I,\mu_{j}}-E_{K,\mu_{j}})
			+(P_{I,ij}^{-1}\,\vartheta_{I,i}-P_{K,ij}^{-1}\,\vartheta_{K,i})\otimes E_{I,\mu_{j}}.
		\end{align*}
		Notice that $P_{K,ij}^{-1}\,\vartheta_{K,i}=-J_{K}(\mathrm{d}\mu_{j})$, thus the last line satisfies the desired estimate.
		For the second term we have
		\begin{align*}
			&(P_{I,ij}^{-1}\,\vartheta_{I,i}-P_{K,ij}^{-1}\,\vartheta_{K,i})\otimes E_{I,\mu_{j}}\\[1mm]
			&\quad=(\vartheta_{I,i}-P_{I,ij}P_{K,jt}^{-1}\,\vartheta_{K,t})\otimes P_{I,is}^{-1}E_{I,\mu_{s}}\\[1mm]
			&\quad=(\vartheta_{I,i}-\vartheta_{K,i}
			-(h_{I,ij}-h_{K,ij})P_{K,jt}^{-1}\,\vartheta_{K,t})\otimes P_{I,is}^{-1}E_{I,\mu_{s}},
		\end{align*}
		and the required estimate follows.
		
		To estimate $\Omega_{I}-\Omega_{K}$, observe that
		\[
		\Omega_{K}=\bigwedge_{j=1}^{N}(-\sqrt{-1}\,\xi_{K,j})\wedge \mathrm{d}\eta,
		\qquad
		\xi_{K,i}=P_{K,ij}\,\mathrm{d}\mu_{j}+\sqrt{-1}\,\vartheta_{K,i}.
		\]
		Hence $\Omega_{I}-\Omega_{K}$ can be expanded into a linear combination of terms of the form
		\[
		\bigwedge_{j=1}^{k}(-\sqrt{-1}\,(\xi_{I,j}-\xi_{K,j}))
		\wedge\bigwedge_{j=k+1}^{N}(-\sqrt{-1}\,\xi_{K,j})\wedge \mathrm{d}\eta,
		\]
		where the first factor is decaying and the remaining factors are bounded in the weighted norm by Proposition~\ref{buchongassumation}.
	\end{proof}
	
	\begin{remark}\label{devationremark1}
		In the later gluing construction we will see that, on $\T_{2R_{1}}\cap\B^{n}_{K}$,
		\[
		\|g_{I}-g_{K}\|_{C^{k,\alpha}_{g_{K}}}
		\le\ell_{1}^{-\epsilon}\,\ell_{|K|}^{-1+\epsilon}
		\]
		for a small constant $\epsilon>0$.
	\end{remark}
	
	We next consider $\T_{R_{1}}\cap\B^{n}_{a}$.
	Note that $\T_{R_{1}}\cap\B^{n}_{a}$ can be decomposed into finitely many simply-connected components, we still denote one of them by $\T_{R_{1}}\cap\B^{n}_{a}$.
	Since this region is topologically trivial, we can input the GH coefficients $A_{ij}$, $\det A$ to construct a GH metric $g_{\mathrm{flat}}$.
	For any $\mathbb{T}^{n}\times\R^{N-n}$-invariant tensor field $T$ on $\pi^{-1}(\T_{2R_{1}}\cap\B^{n}_{a})$ we introduce the normalized H\"older seminorm
	\[
	[T]_{\alpha}=\sup_{p}\ell_{1}(p)^{\alpha}
	\sup_{p'\in B_{g_{\mathrm{flat}}}(p,\ell_{1}(p)/10)}
	\frac{|T(p)-T(p')|_{g_{\mathrm{flat}}}}{d_{g_{\mathrm{flat}}}(p,p')^{\alpha}}.
	\]
	We write symbolically
	\[
	\|T\|_{C^{k,\alpha}_{g_{\mathrm{flat}}}}\le C\rho
	\]
	to mean
	\[
	\|T\|_{C^{k,\alpha}_{g_{\mathrm{flat}}}}
	=\sum_{j=0}^{k}\bigl\|\rho^{-1}\ell_{1}^{j}
	\nabla_{g_{\mathrm{flat}}}^{j}T\bigr\|_{L^{\infty}}
	+\bigl[\rho^{-1}\ell_{1}^{k}
	\nabla_{g_{\mathrm{flat}}}^{k}T\bigr]_{\alpha}\le C.
	\]
	
	On $\T_{2R_{1}}\cap\B^{n}_{a}$ we have
	\[
	\dist_{A}(\,\cdot\,,\partial\D)\sim\dist_{A}(\,\cdot\,,\D^{n}_{I}).
	\]
	Thus, arguing as before, we can choose $\vartheta_{\mathrm{flat}}$ so that
	\[
	\|\vartheta_{I,i}-\vartheta_{\mathrm{flat},i}\|
	_{C^{k,\alpha}_{g_{\mathrm{flat}}}}\le C\ell_{n}^{-1}.
	\]
	In particular we set $\vartheta_{\mathrm{flat},I^{\prime}}=\vartheta_{I,I^{\prime}}$.
	
	\begin{proposition}\label{devationf}
		On $\T_{2R_{1}}\cap\B^{n}_{a}$ we have the deviation estimates
		\[
		\begin{cases}
			\|g_{I}-g_{\mathrm{flat}}\|_{C^{k,\alpha}_{g_{\mathrm{flat}}}}
			\le\ell_{n}^{-1},&
			\|\omega_{I}-\omega_{\mathrm{flat}}\|
			_{C^{k,\alpha}_{g_{\mathrm{flat}}}}\le\ell_{n}^{-1},\\[1mm]
			\|J_{I}-J_{\mathrm{flat}}\|_{C^{k,\alpha}_{g_{\mathrm{flat}}}}
			\le\ell_{n}^{-1},&
			\|\Omega_{I}-\Omega_{\mathrm{flat}}\|
			_{C^{k,\alpha}_{g_{\mathrm{flat}}}}\le\ell_{n}^{-1}.
		\end{cases}
		\]
		In particular, if $R_{1}$ is chosen sufficiently large, then
		\[
		|g_{I}-g_{\mathrm{flat}}|\ll 1,\quad
		|\omega_{I}-\omega_{\mathrm{flat}}|\ll 1,\quad
		|\Omega_{I}-\Omega_{\mathrm{flat}}|\ll 1,\quad
		|J_{I}-J_{\mathrm{flat}}|\ll 1.
		\]
	\end{proposition}
	
	\begin{remark}\label{devationremark2}
		The metric $g_{\mathrm{flat}}$ can also be defined on the larger region
		\[
		\T_{2R_{1}}\cap\bigl\{\dist_{A}(\,\cdot\,,\D)>R^{\prime}\bigr\}.
		\]
		By (IV) of Proposition~\ref{inductionhyp} we have
		\[
		\|g_{I}-g_{K}\|_{C^{k,\alpha}_{g_{K}}}\le\ell_{1}^{-1},
		\]
		so that, for $R_{1}$ large enough,
		\[
		\|g_{I}-g_{K}\|_{C^{k,\alpha}_{g_{K}}}\ll 1.
		\]
	\end{remark}
	
	One verifies that on the overlaps of the regions above the corresponding norms are equivalent.
	Thus, for any $\mathbb{T}^{n}\times\R^{N-n}$-invariant tensor $T$ on $\C^{n+1}\times\R^{2(N-n)}$, we endow the space with the weighted Sobolev norm $\|T\|_{C^{k,\alpha}}$;
	the concise bound $\|T\|_{C^{k,\alpha}}\le C\rho$ is then unambiguously
	interpreted within this convention.
	Notice that on $\B^{n}\setminus\T_{2R_{1}}$ the weight $\rho$ is uniformly bounded below, so $\|T\|_{C^{k,\alpha}}\le C\rho$ reduces to $\|T\|_{C^{k,\alpha}_{\mathrm{Nor}}}\le C$.
	By the metric deviation estimates in Propositions~\ref{devationK} and~\ref{devationf}, this norm is equivalent to the weighted norm defined by $g_{I}$.
	
	We claim that the norm $\|\cdot\|_{C^{k,\alpha}}$ naturally extends to $\mathbb{T}^{n}$-invariant tensors on $\C^{n+1}$.
	This follows immediately from the product structure of $g_{I}$.
	Let us explain how $g_{K}$ and $g_{\mathrm{flat}}$ restrict to $(\pi^{1})^{-1}(\B^{n}_{1})$.
	Take $K=\{0,1,\dots,k\}$.
	On the base we set $\nu_{I,s}=t_{s}$ for $s\in I^{\prime}$, on the fibre we fix the frame $e_{I}+A_{I^{\prime}}^{-1}A_{I^{\prime}I}e_{I}$.
	One checks that
	\[
	\nu_{I,I^{\prime}}=\nu_{K,I^{\prime}}
	+A_{I^{\prime}}^{-1}A_{I^{\prime}J}\nu_{K,J},\qquad
	e_{I,I^{\prime}}=e_{K,I^{\prime}}
	+A_{I^{\prime}}^{-1}A_{I^{\prime}J}e_{K,J}.
	\]
	Since $P_{K,ij}$ and $Q_{K}$ induce the smooth structure $\C^{k+1}\times\R^{2(N-k)}$, restricting to $H$ amounts to choosing a section of $\R^{2(N-k)}$, so this smooth structure descends to $H$ without changing the holomorphic part.
	Notice that $p_{K,ij}=0$ whenever $i\in I^{\prime}$ or $j\in I^{\prime}$, hence
	\begin{align*}
		g_{K}&=g_{2}+(G_{I,ij}+p_{K,ij})\,\mathrm{d}\mu_{i}\otimes \mathrm{d}\mu_{j}
		+Q_{K}\,|\mathrm{d}\eta|^{2}\\[1mm]
		&\quad+(G_{I}+p_{K})^{-1}_{ij}
		\,(\vartheta_{K}-A_{II^{\prime}}A_{I^{\prime}}^{-1}\vartheta_{K,I^{\prime}})_{i}
		\otimes(\vartheta_{K}-A_{II^{\prime}}A_{I^{\prime}}^{-1}\vartheta_{K,I^{\prime}})_{j}\\[2mm]
		&=g_{2}+g_{K,1},
	\end{align*}
	where $g_{2}$ is the same as in the decomposition $g_{I}=g_{1}+g_{2}$, because we have chosen $\vartheta_{K,I^{\prime}}=\vartheta_{I,I^{\prime}}$.
	Thus $g_{K}$ also splits as a product on $\C^{n+1}\times\R^{2(N-n)}$, and we can restrict it to $\C^{n+1}$ by simply taking $g_{K,1}$.
	Conversely, $g_{K,1}$ can be viewed as the GH metric on $H$ induced by
	\[
	P_{ij}+h_{I,ij}-h_{K,ij},\qquad
	Q_{K}+h_{I}-h_{J}.
	\]
	By the preceding discussion in Subsection~\ref{subsectionstructureonmodel}, $g_{K}$ is Calabi--Yau if and only if $g_{K,1}$ is Calabi--Yau.
	
	For $\mathbb{T}^{n}$-invariant tensors on $\C^{n+1}$ we can use the weighted Sobolev norm defined above by regarding them as tensors on $\C^{n+1}\times\R^{2(N-n)}$ that depend only on the $\C^{n+1}$ factor (We can naturally restrict $\ell_{i}$ to $\B^{n}_{1}$).
	One verifies that for such tensors this norm is equivalent to the analogous weighted Sobolev norm defined by $g_{\mathrm{flat},1}$ or $g_{K,1}$, and the corresponding estimates in Propositions~\ref{devationK} and~\ref{devationf} continue to hold.

	\subsection{Gluing Calabi--Yau Metrics}\label{subsectiongluingcalabiyau}
	In this section we perform the gluing construction.  
	We first recall the solvability and a priori estimate for the complex Monge--Ampère equation.
	
	\begin{theorem}\label{MAkejiexing}
		There exists a complete metric
		\[
		\omega_{\C^{n+1}}
		=\omega_{1}+\sqrt{-1}\,\partial\bar\partial\varphi
		\quad\text{on }\C^{n+1}
		\]
		satisfying
		\[
		\omega_{\C^{n+1}}^{n+1}
		=\frac{(n+1)!}{2^{n+1}}\,\sqrt{-1}^{(n+1)^{2}}
		\Omega_{\mathrm{Nor}}\wedge\overline{\Omega}_{\mathrm{Nor}},
		\]
		together with the metric-deviation estimates
		\begin{equation}\label{decayofphi}
			\begin{aligned}
				\|\mathrm{d}\varphi\|_{C^{k+1,\alpha}}
				&\le C\ell_{1}^{-\epsilon}\ell_{n}^{-1+\epsilon},\\[1mm]
				\|\sqrt{-1}\,\partial\bar\partial\varphi\|_{C^{k,\alpha}}
				&\le C\ell_{1}^{-1-\epsilon}\ell_{n}^{-1+\epsilon}.
			\end{aligned}
		\end{equation}
		The meaning of these inequalities was explained in the previous section.
		Here $\varphi$ is a $\mathbb{T}^{n}$-invariant function, $0<\epsilon\ll 1$ is an arbitrarily small given number, and the constants depend only on $n,\alpha,\epsilon$.
		This metric inherits all the symmetries of $\omega_{1}$.
	\end{theorem}
	
	We will prove Theorem~\ref{MAkejiexing} in Section~\ref{subsectionpde}, here we assume its validity and fix $0<\epsilon\ll 1$.
	
	\begin{lemma}\label{phic0guji}
		The function $\varphi$ above satisfies the $C^{0}$ estimate
		\[
		|\varphi|\le C\ell_{n}^{\epsilon}.
		\]
	\end{lemma}
	
	\begin{proof}
		By Theorem~\ref{MAkejiexing} we have a weighted estimate for $\mathrm{d}\varphi$.
		Since $\varphi$ is $\mathbb{T}^{n}$-invariant, we regard $\mathrm{d}\varphi$ as a $1$-form on $\B^{n}$ and argue region by region.
		On $\B^{n}_{a}$ the weighted norm gives a bound for $|\mathrm{d}\varphi|_{g_{A}}$, so the desired estimate follows by line integration.
		On $\H^{n}_{K}$ we use the product decomposition of $g_{K}$ and the fact that $\dist_{A}(p,O)\sim\dist_{A}(p_{K},O)$, where $p_{K}$ is the projection of $p$ onto $\D^{n}_{K}$ with respect to $g_{A}$.
		Combining this with (IV) of Proposition~\ref{inductionhyp} yields the claim.
	\end{proof}
	
	By the definition of $\B^{n}$ in Subsection~\ref{subsectionstructureonmodel} we can naturally regard $\B_{I}\subset\B^{n}$ and the geometry of the fibration directions is identical.
	For $p\in\B_{I}$ we have $\nu_{I,j}>0$ ($j\in I^{\prime}$), and \eqref{affinejulihengdengshi} gives
	\[
	\dist_{A}^{2}(p,\partial\D_{I})
	=\dist_{A}^{2}(p,\D_{I})
	+\dist_{A}^{2}(p_{I},\partial\D_{I}).
	\]
	Hence
	\begin{equation}\label{julidengjiaaffien2}
		\sqrt{1-\frac{1}{C_{0}^{2}}}\,
		\dist_{A}(p,\partial\D_{I})
		\le\dist_{A}(p_{I},\partial\D_{I})
		\le\dist_{A}(p,\partial\D_{I}).
	\end{equation}
	Since $p_{I}=(0,\dots,0,\nu_{I,n+1},\dots,\nu_{I,N},0)$, the quantity $\dist_{A}(p_{I},\partial\D_{I})$ is a Lipschitz function of $\nu_{I,I^{\prime}}$ alone.
	We need the following lemma.
	
	\begin{lemma}\label{juliheaffinedengjia}
		In $\B_{I}$, for any $J\supseteq I$ we have
		\begin{equation}\label{julidengjiaaffinehanshu}
			\rho_{I,J}\le\dist_{A}(p_{I},\D_{J})\le\hat{C}\rho_{I,J},
		\end{equation}
		where $\hat{C}>1$ is a uniform constant depending only on $n$ and $\Lambda/\lambda$, and
		\[
		\rho_{I,J}
		=\sqrt{\nu_{I,K}^{\top}
			\bigl(A_{I^{\prime},K}
			-A_{I^{\prime},KJ^{\prime}}A_{I^{\prime},J^{\prime}}^{-1}
			A_{I^{\prime},J^{\prime}K}\bigr)\nu_{I,K}}
		\quad\text{with }\ J=I\cup K.
		\]
	\end{lemma}
	
	\begin{proof}
		If the projection of $p_{I}$ onto $\D_{J}$ lies in the interior of $\D_{J}$, then $\rho_{I,J}=\dist_{A}(p_{I},\D_{J})$.
		In general, we perform a linear change of coordinates for $\nu_{I,I^{\prime}}$ so that $g_{A_{I^{\prime}}}$ becomes the standard metric.
		This transforms the unbounded polyhedral cone $\mathfrak{C}=\{\nu_{I,j}>0\}$ into another unbounded polyhedral cone $\mathfrak{C}^{\prime}$ whose dihedral angles are uniformly strictly less than $\pi$.
		In the new coordinates $\rho_{I,J}$ is the distance from a point in $\mathfrak{C}$ to the affine hyperplane corresponding to $\D_{J}$. Therefore, we can obtain \eqref{julidengjiaaffinehanshu}
	\end{proof}
	
	We now perform the gluing.
	Fix a constant $C^{\prime}>0$ (to be chosen large) and work on the region
	\[
	\B_{I}\cap\{\dist_{A}(p,\partial\D_{I})>C^{\prime}\}.
	\]
	Choose a function on $\mathbb{R}$ satisfying
	\[
	\chi(x)=
	\begin{cases}
		1,\quad |x|<\frac{3}{8},\\[2mm]
		0,\quad |x|\ge\frac{1}{2},
	\end{cases}
	\]
	so that all derivatives of $\chi$ are uniformly bounded.  
	Set
	\begin{align}\label{glumetric3}
		\omega_{\mathrm{glu}}=\omega_{1}+dd^{c}(f\varphi)+\omega_{2},
	\end{align}
	where
	\begin{align}
		f=\prod_{I\subsetneqq J}\chi\Bigl(C_{0}\frac{|\vec{\mu}_{I}|}{\rho_{I,J}}\Bigr).
	\end{align}
	By virtue of \eqref{julidengjiaaffinehanshu}, the function f  is smooth.
	
	If $p\in\B_{I}^{\prime\prime}$, then \eqref{julidengjiaaffien2} and \eqref{julidengjiaaffinehanshu} give
	\begin{align*}
		|\vec{\mu}_{I}|
		&<\frac{1}{4CC_{0}}\dist_{A}(p,\partial\D_{I})\\[1mm]
		&\le\frac{1}{4C\sqrt{C_{0}^{2}-1}}\dist_{A}(p_{I},\partial\D_{I})
		\le\frac{1}{4\sqrt{C_{0}^{2}-1}}\rho_{I,J}.
	\end{align*}
	Thus for every $J\supsetneqq I$,
	\[
	C_{0}\frac{|\vec{\mu}_{I}|}{\rho_{I,J}}<\frac{3}{8},
	\qquad\text{so}\quad f\equiv 1\quad\text{on }\B_{I}^{\prime\prime}.
	\]
	Similarly, on $\B_{I}\setminus\B_{I}^{\prime}$,
	\begin{align*}
		|\vec{\mu}_{I}|
		&\ge\frac{1}{2C_{0}}\dist_{A}(p,\partial\D_{I})\\[1mm]
		&\ge\frac{1}{2C_{0}}\dist_{A}(p_{I},\partial\D_{I})
		\ge\frac{1}{2C_{0}}\min_{I\subsetneqq J}\rho_{I,J},
	\end{align*}
	so there exists at least one $J$ with
	\[
	C_{0}\frac{|\vec{\mu}_{I}|}{\rho_{I,J}}>\frac{1}{2},
	\qquad\text{hence}\quad f\equiv 0\quad\text{on }\B_{I}\setminus\B_{I}^{\prime}.
	\]
	
	Consequently, $\omega_{\mathrm{glu}}$ is a smooth closed $(1,1)$-form on
	\[
	\pi^{-1}\bigl(\B_{I}\cap\{\dist_{A}(p,\partial\D_{I})>C^{\prime}\}\bigr),
	\]
	and
	\[
	\omega_{\mathrm{glu}}=\omega_{\C^{n+1}}+\omega_{2}\quad\text{in }\B_{I}^{\prime\prime},
	\qquad
	\omega_{\mathrm{glu}}=\omega_{I}\quad\text{in }\B_{I}\setminus\B_{I}^{\prime}.
	\]
	We will show that $\omega_{\mathrm{glu}}$ is a \kah form provided $C^{\prime}$ is sufficiently large, it suffices to analyze the region where $f$ decays.
	
	By the $C^{2}$-estimate \eqref{decayofphi} for $\varphi$,
	\[
	\Bigl(1-\frac{C}{|\vec{\mu}_{I}|^{1-\epsilon}}\Bigr)\omega_{1}
	<\omega_{1}+dd^{c}\varphi
	<\Bigl(1+\frac{C}{|\vec{\mu}_{I}|^{1-\epsilon}}\Bigr)\omega_{1}.
	\]
	In the region where $f$ decays there exists $J\supsetneqq I$ such that
	\begin{align}\label{gluingguji3}
		\frac{C^{\prime}}{4C_{0}}
		<\frac{1}{4C_{0}}\rho_{I,J}
		<|\vec{\mu}_{I}|
		<\frac{3}{4C_{0}}\rho_{I,J}.
	\end{align}
	Choosing $1\gg\delta>0$ and $C^{\prime}$ sufficiently large, we obtain
	\begin{align}\label{gluingguji1}
		\omega_{1}+f\,dd^{c}\varphi
		=f(\omega_{1}+dd^{c}\varphi)+(1-f)\omega_{1}
		>(1-\delta)\omega_{1}.
	\end{align}
	
	We next control $df\wedge \mathrm{d}^{c}\varphi$ and $\mathrm{d}\varphi\wedge \mathrm{d}^{c}f$.
	Recall that
	\[
	\omega_{2}
	=\frac{\sqrt{-1}}{2}A_{I^{\prime},ij}\xi_{i}\wedge\bar{\xi}_{j}
	=\sum_{i=n+1}^{N}\mathrm{d}\nu_{I,i}\wedge \mathrm{d}\theta_{i},
	\qquad
	\xi_{i}=A_{I^{\prime},ij}^{-1}\,\mathrm{d}\nu_{I,j}+\sqrt{-1}\,\mathrm{d}\theta_{i}.
	\]
	Hence
	\[
	\|\mathrm{d}\nu_{I,j}\|_{g_{2}}=1,\qquad dd^{c}\nu_{I,j}=0.
	\]
	By the explicit form of $\rho_{I,J}$, we have
	\begin{equation}\label{gluingguji2}
		|\mathrm{d}\rho_{I,J}|_{g_{2}}
		+\rho_{I,J}|dd^{c}\rho_{I,J}|_{g_{2}}
		\le C.
	\end{equation}
	Combining \eqref{vecmuIbudengshi}, the $C^{1}$-estimate \eqref{decayofphi} for $\varphi$, and the Cauchy inequality, we can absorb the negative contributions of $df\wedge \mathrm{d}^{c}\varphi$ and $\mathrm{d}\varphi\wedge \mathrm{d}^{c}f$ by \eqref{gluingguji1}.
	Similarly, for $\varphi\,dd^{c}f$ the derivatives of $f$ produce terms like
	\[
	dd^{c}\rho_{I,J},\quad
	\mathrm{d}\rho_{I,J}\wedge \mathrm{d}^{c}|\vec{\mu}_{I}|,\quad
	\mathrm{d}^{c}\rho_{I,J}\wedge \mathrm{d}|\vec{\mu}_{I}|,\quad
	dd^{c}|\vec{\mu}_{I}|,
	\]
	which are controlled by \eqref{gluingguji2} and Lemma~\ref{phic0guji}.
	Thus
	\[
	\omega_{\mathrm{glu}}>0\quad\text{when }C^{\prime}\text{ is sufficiently large.}
	\]
	
	As in Subsection~\ref{subsctionsurgeryoriginal}, we consider the moment map
	\[
	(MP_{\mathrm{glu}},\eta)\colon
	\C^{n+1}\times\R^{2(N-n)}\longrightarrow\R^{n}\times\C.
	\]
	Outside the outer-cone region $\B_{I}\setminus\B_{I}^{\prime}$, the form $\omega_{\mathrm{glu}}$ coincides with $\omega_{I}$. Hence after a suitable choice of constants, their moment maps agree and the corresponding GH coefficients are identical.
	On $\B_{I}^{\prime}$ the map $(MP_{\mathrm{glu}},\eta)$ is injective by the same argument as in Subsection~\ref{subsctionsurgeryoriginal}, and it surjects onto $\B_{I}^{\prime}$ because on $\B_{I}^{\prime\prime}$ we have $\omega_{\mathrm{glu}}=\omega_{\C^{n+1}}+\omega_{2}$, which splits as a product and whose $\omega_{2}$-factor has linear moment map.

	The GH coefficients induced by $(MP_{\mathrm{glu}},\eta)$ are originally defined on $\R^{n}\times\C_{\eta}\times\R^{N-n}$ and are then pushed forward to $\R^{N}\times\C_{\eta}\cong\B_{I}$ via an affine transformation.
	Denote the original coefficients by $P^{\mathrm{glu}}_{ij}$, $Q^{\mathrm{glu}}$, and the transformed ones by $P^{\mathrm{glu}}_{I,ij}$, $Q^{\mathrm{glu}}_{I}$.
	Set
	\[
	V^{\mathrm{glu}}_{ij}=P^{\mathrm{glu}}_{I,ij}+h_{I,ij},\qquad
	W^{\mathrm{glu}}=Q^{\mathrm{glu}}_{I}+h_{I}.
	\]
	We have thus obtained GH coefficients $V^{\mathrm{glu}}_{ij}$, $W^{\mathrm{glu}}$ on
	\[
	\B_{I}\cap\bigl\{\dist_{A}(\,\cdot\,,\partial\D_{I})>C^{\prime}\bigr\}.
	\]
	It remains to verify that these coefficients satisfy the conditions listed in Proposition~\ref{inductionhyp}.

	First, $V^{\mathrm{glu}}_{ij}$ and $W^{\mathrm{glu}}$ are not necessarily positive a priori. We only know that $P^{\mathrm{glu}}_{I,ij}$ is positive definite and $Q^{\mathrm{glu}}_{I}>0$, because $\omega_{\mathrm{glu}}>0$ and the congruence transformation preserves positivity.
	On $\B_{I}$, however, $h_{I,ij}$ and $h_{I}$ are smooth uniformly bounded harmonic functions, so the $2$-form
	\[
	h_{I,ij}\,\mathrm{d}\mu_{i}\otimes \mathrm{d}\mu_{j}+h_{I}\,|\mathrm{d}\eta|^{2}
	\]
	is smooth.
	By the deviation estimate \eqref{decayofphi} in Theorem~\ref{MAkejiexing}, the weighted norm defined by $g_{\mathrm{glu}}$ is equivalent to that defined by $g_{I}$,  and both $\|\mathrm{d}\mu_{i}\|_{g_{\mathrm{glu}}}$ and $\|\mathrm{d}\eta\|_{g_{\mathrm{glu}}}$ are bounded.
	Since $h_{I,ij}$ and $h_{I}$ decay linearly with respect to $\dist_{A}(\,\cdot\,,\partial\D_{I})$, choosing $C^{\prime}$ sufficiently large yields positivity.
	If $I\subset J$, then on
	\[
	\B_{J}\cap\{\dist_{A}(p,\partial\D_{I})>C^{\prime}\}
	\]
	the same argument together with Lemma~\ref{alphajianjin} shows that
	\[
	P^{\mathrm{glu}}_{I,ij}+h_{I,ij}-h_{J,ij},\qquad
	Q^{\mathrm{glu}}_{I}+h_{I}-h_{J}
	\]
	are positive definite for sufficiently large $C^{\prime}$.

	Choosing $C_{n}>C^{\prime}$, we conclude that on
	\[
	\B_{I}\cap\{\dist_{A}(\,\cdot\,,\partial\D_{I})>C^{\prime}\}
	\]
	the coefficients $V^{\mathrm{glu}}_{ij}$ and $W^{\mathrm{glu}}$ satisfy all the positivity conditions listed in (I) and the relevant assumptions in (II) and (III) of Proposition~\ref{inductionhyp}.
	The integrability condition follows from the harmonicity of $h$ and the commutation relations, while the Chern-class condition is inherited from $P^{\mathrm{glu}}_{ij}$, $Q^{\mathrm{glu}}$ and the smooth harmonic property of $h$.
	Performing the surgery on all $\B_{L}$ with $|L|=n+1$ and taking $C_{n}$ sufficiently large. We obtain the GH coefficients $V^{(n)}_{ij}$, $W^{(n)}$ together with their domain of definition $\F_{n}$, these surgery regions do not interfere with each other.
	Thus $V^{(n)}_{ij}$ and $W^{(n)}$ satisfy the positivity assumptions in (I), (II) and (III).
	
	The existence of the map $\mathbf{F}_{I}$ required in (II), i.e.\ Proposition~\ref{prop:smooth-structure-DK}, is guaranteed by the construction of $V^{\mathrm{glu}}_{ij}$ and $W^{\mathrm{glu}}$.
	As noted in Remark~\ref{devationremark1}, for $K\subset I$ the smooth structure near $\B_{I}\cap\D_{K}$ is already in place. The smoothness of $\mathbf{F}_{I}$ with respect to this structure follows from the continuity discussion in Subsection~\ref{subsectionlogarithmicgrowth} together with the Hartogs lemma.
	By the deviation estimate \eqref{decayofphi} for $g_{\mathrm{glu}}$, the holomorphic volume form $\Omega_{\mathrm{Nor}}$ and the fundamental vector fields $X_{i}$ of the $\mathbb{T}^{n}\times\R^{N-n}$-action are bounded in the weighted norm. The same holds for the curvature tensor, thereby establishing (i) and (ii) of Proposition~\ref{buchongassumation}.
	
	Now let $I\subset J$, we prove that on $\F_{n}\cap\B_{I}\cap\B_{J}$ the coefficients
	\[
	P^{\mathrm{glu}}_{I,ij}+h_{I,ij}-h_{J,ij}
	\quad\text{and}\quad
	Q^{\mathrm{glu}}_{I}+h_{I}-h_{J}
	\]
	depend only on $\mu_{J}$ and $\eta$.
	By symmetry we may assume $J=\{0,1,\dots,n,n+1,\dots,n+k\}$.
	It suffices to show that $P^{\mathrm{glu}}_{ij}$ and $Q^{\mathrm{glu}}$ depend only on $\mu_{J}$ and $\eta$, which is determined by the gluing function $f$.
	Specifically, we need to verify that on $\B_{J}\cap\F_{n-1}$, for any $I\subset L$ with $L\subsetneqq J$,
	\[
	\chi\Bigl(C_{0}\frac{|\vec{\mu}_{I}|}{\rho_{I,L}}\Bigr)=1
	\ \Leftrightarrow\
	C_{0}\frac{|\vec{\mu}_{I}|}{\rho_{I,L}}<\frac{3}{8}.
	\]
	By Lemma~\ref{juliheaffinedengjia},
	\[
	\rho_{I,L}\ge\frac{1}{\hat{C}}\dist_{A}(\,\cdot\,,\D_{L});
	\]
	hence it is enough to show that on $\B_{J}\cap\B_{J}$,
	\[
	\hat{C}C_{0}\frac{\dist_{A}(p,\D_{I})}{\dist_{A}(p,\D_{L})}<\frac{3}{8}.
	\]
	The definition of $\B_{J}$ gives
	\[
	C_{0}\dist_{A}(p,\D_{J})<\dist_{A}(p,\D_{L}),
	\]
	and the definition of $\B_{I}$ gives
	\[
	C_{0}\dist_{A}(p,\D_{I})<\dist_{A}(p,\D_{J});
	\]
	consequently,
	\[
	\hat{C}C_{0}\frac{\dist_{A}(p,\D_{I})}{\dist_{A}(p,\D_{L})}<\frac{\hat{C}}{C_{0}}.
	\]
	Choosing $C_{0}$ sufficiently large uniformly yields the desired inequality.
	
	Moreover, for the metric $g_{I}$ before gluing, the corresponding GH coefficients satisfy
	\[
	V^{(n-1)}_{kl}-h_{I,kl}=A_{kl}
	\quad\text{whenever }k\notin I\text{ or }l\notin I.
	\]
	Thus, to affect $P^{\mathrm{glu}}_{ij}$ with $k\notin I$ or $l\notin I$, the gluing term $dd^{c}(f\varphi)$ must involve $\nu_{I,k}$, because
	\[
	Jd\nu_{I,k}=-A^{-1}_{kl}\vartheta_{l}.
	\]
	Near $\B_{J}$ the gluing function $f$ only involves $\nu_{I,J}$, so when $dd^{c}(f\varphi)$ is evaluated on $e_{k},e_{l}$ with $k\notin J$ or $l\notin J$, the result is zero.
	Under the affine transformation that produces $P^{\mathrm{glu}}_{I,ij}$, $Q^{\mathrm{glu}}_{I}$ from $P^{\mathrm{glu}}_{ij}$, $Q^{\mathrm{glu}}$, this property is preserved, and since $h_{I,kl}=h_{J,kl}$ whenever $k\notin J$ or $l\notin J$, we finally obtain
	\[
	V^{(n)}_{kl}-h_{I,kl}=A_{kl}
	\quad\text{whenever }k\notin I\text{ or }l\notin I
	\]
	on $\F_{n}\cap\B_{J}$.
	Furthermore, the special form of $dd^{c}(f\varphi)$ implies that for $J=I\cup\{j\}$,
	\[
	|V^{(n)}_{ij}-h_{I,ij}-A_{ij}|
	\le\frac{C}{\dist_{A}(\,\cdot\,,\D_{J})}
	\quad\text{for all }i\in I
	\]
	on $\F_{n}\cap\B_{J}$.
	
	We now verify (IV) of Proposition~\ref{inductionhyp} on the region
	\[
	\F_{n}\cap\{\dist_{A}(\,\cdot\,,\D)>1\}\cap\B_{I}.
	\]
	By the induction hypothesis, the GH coefficients of $g_{I}$ already satisfy the required estimate there.
	We first relate the coefficients of $g_{I}$ and $g_{\mathrm{glu}}$:
	\begin{equation*}
		\begin{aligned}
			(V^{\mathrm{glu}})^{-1}_{ij}\circ\mu^{\mathrm{glu}}
			&=\omega_{\mathrm{glu}}\bigl(JX_{i},
			X_{j}\bigr)
			=\bigl(\omega+dd^{c}f\varphi\bigr)\bigl(JX_{i},
			X_{j}\bigr)\\[-1pt]
			&=V^{-1}_{ij}\circ\mu+H^{ij}\circ\mu
			=V^{-1}_{ik}(\delta_{kj}+V_{kt}H^{tj})\circ\mu\\[-1pt]
			&=V^{-1}_{ik}(\delta_{kj}+H^{j}_{k})\circ\mu,
		\end{aligned}
	\end{equation*}
	so that
	\begin{equation}\label{gluingdevV}
		V^{\mathrm{glu}}_{ij}\circ\mu^{\mathrm{glu}}
		=V_{ik}\bigl(E+(H^{t}_{l})_{N\times N}\bigr)^{-1}_{kj}\circ\mu,
	\end{equation}
	where the decay term $H^{j}_{k}$ satisfies
	\[
	\|H^{j}_{k}\|_{C^{k,\alpha}}\le C\ell_{1}^{-1-\epsilon}\ell_{n}^{-1+\epsilon}.
	\]
	
	Similarly,
	\begin{equation}\label{gluingdevW}
		\begin{aligned}
			W^{\mathrm{glu}}\circ\mu^{\mathrm{glu}}
			&=\frac{\omega_{\mathrm{glu}}^{n}}{\Omega\wedge\bar\Omega}\,
			\det(V^{\mathrm{glu}}_{ij})\circ\mu^{\mathrm{glu}}\\[1mm]
			&=\frac{\omega^{n}}{\Omega\wedge\bar\Omega}\,
			\det(V_{ij})(1+H)\circ\mu\\[1mm]
			&=W(1+H)\circ\mu,
		\end{aligned}
	\end{equation}
	with
	\[
	\|H\|_{C^{k,\alpha}}\le C\ell_{1}^{-1-\epsilon}\ell_{n}^{-1+\epsilon}.
	\]
	
	Notice that
	\begin{equation}\label{gluingdevmu}
		\mu^{\mathrm{glu}}_{i}
		=\mu_{i}+\iota_{X_{i}} \mathrm{d}^{c}f\varphi
		=\mu_{i}+h_{i},
	\end{equation}
	where
	\[
	\|h_{i}\|_{C^{k+1,\alpha}}\le C\ell_{1}^{-\epsilon}\ell_{n}^{-1+\epsilon}.
	\]
	Thus, for sufficiently large $C_{n}$, the weights $1+|\mu^{\mathrm{glu}}|$ and $1+|\mu|$ are uniformly equivalent, and the desired estimate in (IV) follows.
	The same argument shows that inequality \eqref{duishuzengzhang} also holds for $\mu^{\mathrm{glu}}$.
	
	From \eqref{gluingdevmu} we also see that on the bundle
	\[
	\pi^{\mathrm{glu}}\colon\C^{n+1}\times\R^{2(N-n)}
	\longrightarrow\R^{n}\times\C\times\R^{N-n}
	\]
	induced by $g_{\mathrm{glu}}$ we can still input the original GH coefficients $P_{I,ij}$, $Q_{I}$ to obtain a metric $g^{\prime}$ satisfying
	\[
	g^{\prime}=\Phi_{*}g,
	\]
	where $\Phi$ is a diffeomorphism of $\C^{n+1}\times\R^{2(N-n)}$. In the surgry region $\Phi-\mathrm{id}$ is measured by $h_{i}$.  Hence
	\[
	\|\Phi-\mathrm{id}\|_{C^{k,\alpha}}\le C\ell_{1}^{-\epsilon}\ell_{n}^{-1+\epsilon},
	\]
	and $\Phi=\mathrm{id}$ outside the surgery region.
	Hence
	\[
	\|g^{\prime}-g_{\mathrm{glu}}\|
	\le\|\Phi_{*}g-g\|+\|g-g_{\mathrm{glu}}\|
	\le C\ell_{1}^{-\epsilon}\ell_{n}^{-1+\epsilon},
	\]
	so that on $\T_{2R_{1}}\cap\B^{n}_{K}$,
	\[
	\|g_{\mathrm{glu}}-g_{K}\|\le C\ell_{1}^{-\epsilon}\ell_{n}^{-1+\epsilon}.
	\]
	
	Moreover, item \textup{(iii)} in Proposition~\ref{buchongassumation} has already been established in Lemma~\ref{zzengzhangxing}, so all the hypotheses listed in Section~\ref{inductionhyp} have now been verified.  
	Notice that, throughout the preceding discussion, we have assumed $n\le N-1$. We finally obtain the GH coefficients $V^{(N-1)}_{ij}$, $W^{(N-1)}$, defined on $\F_{N-1}$.  
	Since $\F_{N-1}$ is the complement of a large ball in $\B$, one further application of the above surgery produces a Taub--NUT type metric on $\C^{N+1}$, thereby completing the proof of Theorem~\ref{maintheorem}. Here, the boundedness of curvature follows from the deviation estimate for $g_{I}$ given in Subsection~\ref{subsectionasymptoticprop}, while the decay property away from the locus is ensured by the fact that the curvature of $g_{\mathrm{flat}}$ satisfies $|\mathrm{Rm}_{\mathrm{flat}}|\sim O(\ell_{1}^{-3})$.
	
	\subsection{Solving PDEs}\label{subsectionpde}
	We now prove Theorem~\ref{MAkejiexing}.  First, since $P_{I,kl}=0$ whenever $k\notin I$ or $l\notin I$, the volume-form error of $g_{1}$ is the same as that of $g$:
	\begin{equation*}
		E=\frac{\det A+q_{I}}{\det(A_{ij}+p_{I,ij})}-1.
	\end{equation*}
	We have the following lemma.
	
	\begin{lemma}\label{decayvolumeerror}
		The volume-form error $E$ of the metric $g_{1}$ on $\C^{k+1}$ satisfies
		\begin{equation*}
			\|E\|_{C^{k,\alpha}}\le C\ell_{1}^{-1-\epsilon}\ell_{n}^{-1+\epsilon}
		\end{equation*}
		in the weighted norm sense of Subsection~\ref{subsectionasymptoticprop}.
	\end{lemma}
	
	\begin{proof}
		We argue region by region.
		
		\emph{Region 1:}  The totally unaffected part
		\[
		\B^{n}\setminus\bigcup_{\substack{K\subseteq I\\ |K|\ge 3}}(\B^{n}_{K})^{\prime}.
		\]
		Here $p_{I,ij}=v_{I,ij}$ and $q_{I}=w$,  all non-constant terms are smooth harmonic and decay linearly towards the corresponding locus.
		Our choice of region ensures
		\[
		\dist_{A}(\,\cdot\,,\D^{n}_{I})\le C\dist_{A}(\,\cdot\,,\D^{n}_{K}),
		\]
		so we may regard the decay as being with respect to $\ell_{n}$.
		The leading term of the denominator is $\det A$, while the numerator is the linear part of the expansion of the denominator.  Hence the numerator decays quadratically and
		\[
		\|E\|_{C^{k,\alpha}}\le C\ell_{n}^{-2}.
		\]
		
		\emph{Region 2:}  On $\H_{K}^{\prime\prime}$ our prior assumption gives
		\[
		\det(A_{ij}+p_{K,ij})=\det A+q_{K}.
		\]
		Since
		\[
		p_{I,ij}=p_{K,ij}+h_{K,ij}-h_{I,ij},\qquad
		q_{I}=q_{K}+h_{K}-h_{I},
		\]
		and $h_{K,ij}-h_{I,ij}$, $h_{K}-h_{I}$ are smooth harmonic on $\H^{n}_{K}$ and decay linearly in $\ell_{n}$, we obtain
		\begin{align*}
			E&=\frac{q_{K}+h_{K}-h_{I}}{\det\!\bigl(P_{K,ij}+h_{K,ij}-h_{I,ij}\bigr)}-1\\[1mm]
			&=\frac{1+(\det P_{K}^{-1})(h_{K}-h_{I})}
			{\det\!\bigl(E+P_{K}^{-1}(h_{K}-h_{I})\bigr)}-1.
		\end{align*}
		Note that $P^{-1}_{K,ij}$ is uniformly bounded on $\H_{K}$,  combining this with the assumptions in Proposition~\ref{buchongassumation} yields
		\[
		|E|\le C\ell_{1}^{-1}\ell_{n}^{-1}.
		\]
		
		\emph{Region 3:}  On $\H_{K}^{\prime}\setminus\H_{K}^{\prime\prime}$ the situation is similar, but we no longer have the Calabi--Yau identity.
		Our surgery procedure near $\D_{K}$ is identical to that near $\D_{I}$, so we have identities analogous to \eqref{gluingdevV} and \eqref{gluingdevW}:
		\[
		(1+E^{\prime})\det(A_{ij}+p_{K,ij})=\det A+q_{K},
		\]
		where $\|E^{\prime}\|_{C^{k,\alpha}}\le C\ell_{1}^{-1-\epsilon}\ell_{k}^{-1+\epsilon}$.
		On $\H_{K}^{\prime}\setminus\H_{K}^{\prime\prime}$ we have $\ell_{n}\sim\ell_{k}$, so
		\[
		\|E\|_{C^{k,\alpha}}\le C\ell_{1}^{-1-\epsilon}\ell_{n}^{-1+\epsilon}.
		\]
		This completes the proof.
	\end{proof}
	
	\subsubsection{Tian--Yau--Hein Package}
	The analytical toolkit now referred to as the Tian--Yau--Hein package emerged from efforts to extend Yau's solution of the Calabi conjecture \cite{yau1978ricci} to non-compact manifolds.
	Tian--Yau \cite{tian1990complete,tian1991complete} produced the first complete Ricci-flat \kah metrics on the complement of an anti-canonical divisor by solving a complex Monge–Ampère equation whose right-hand side decays faster than any inverse polynomial in the distance to the divisor.
	Hein \cite{hein2010gravitational,hein2011weighted} recast these estimates into a weighted H\"older framework, proving an isomorphism theorem for the Laplacian on algebraically constructed ends and thus converting the earlier existence argument into a quantitative inverse-function machine that converts any `small' perturbation of the volume form into a Ricci-flat metric with prescribed asymptotics.

	The arguments below rely on the work of Hein (cf.\ Chapters 3 and 4 of \cite{hein2010gravitational}) and on \cite[Section~2.7]{li2023syz}.  We verify that the metric $g_{1}$ on $\C^{n+1}$ satisfies the hypotheses of the Tian--Yau--Hein package.
	
	First, $g_{1}$ must admit a so-called $C^{k,\alpha}$ quasi-atlas with $k\ge 3$ (see \cite[Definition~4.2]{hein2010gravitational}).  This ensures that Sobolev norms are well defined and guarantees the maximum principle in the non-compact setting.  The following lemma from \cite{tian1990complete} yields such a quasi-atlas.
	
	\begin{lemma}
		If $|\mathrm{Rm}|\le C$, then there exists a quasi-atlas which is $C^{1,\alpha}$ for every $\alpha$.  If, moreover, $\sum_{i=1}^{k}|\nabla^{i}\mathrm{Scal}|\le C$, then this quasi-atlas is even $C^{k+1,\alpha}$.
	\end{lemma}
	
	The required curvature bounds follow from Proposition~\ref{buchongassumation} and the asymptotic properties of $g_{1}$ established in Subsection~\ref{subsectionasymptoticprop}.
	
	Second, Hein's estimates rely on a weighted Sobolev inequality.  Here the weight is a smooth function $\rho(x)>1$ satisfying $|\nabla\rho|+\rho|\nabla^{2}\rho|\le C$ and equivalent outside a compact set to the distance function.  In our setting we take $\rho=\sqrt{|\vec{\mu}_{I}|^{2}+1}$,  equation \eqref{vecmuIbudengshi} and the asymptotics of $g_{1}$ in Subsection~\ref{subsectionasymptoticprop} guarantee that this $\rho$ meets the requirements.
	
	Since the volume growth of $g_{1}$ is $\operatorname{Vol}(B_{R})\sim R^{n+2}$, we have
	
	\begin{proposition}
		For $1\le p\le\frac{n+1}{n}$ the weighted Sobolev inequality
		\[
		\Bigl(\int_{\C^{n+1}}|u|^{2p}\rho^{np-n-2}\,\mathrm{d}\operatorname{Vol}\Bigr)^{1/p}
		\le C\int_{\C^{n+1}}|\nabla_{g_{1}}u|^{2}\,\mathrm{d}\operatorname{Vol}
		\]
		holds for every $\mathbb{T}^{n}$-invariant function $u$.  The constant depends only on the scale-invariant ellipticity bounds $\lambda,\Lambda$.
	\end{proposition}
	
	\begin{proof}
		The proof is identical to that of \cite[Proposition~2.15]{li2023syz}.
	\end{proof}
	
	Thus $g_{1}$ satisfies the hypotheses of the Tian--Yau--Hein package, and we obtain the following existence and decay statements:
	
	\begin{itemize}
		\item (Poisson equation)  Let $f\in C^{0,\alpha}$ satisfy $|f|\le C\rho^{-q}$ with $n+2>q>2$.  Then there exists a unique $C^{2,\alpha}$ solution of $\Delta_{g_{1}}u=f$ with
		\[
		|u|\le C\rho^{2-q+\epsilon},\qquad 0<\epsilon\ll q-2.
		\]
		
		\item (Complex Monge--Ampère equation)  Let $f\in C^{2,\alpha}$ satisfy $|f|\le C\rho^{-q}$ with $n+2>q>2$.  Then there exist $0<\alpha^{\prime}\le\alpha$ and $u\in C^{4,\alpha^{\prime}}$ solving
		\[
		(\omega_{1}+\sqrt{-1}\, \partial\bar\partial u)^{n+1}=e^{f}\omega_{1}^{n+1}
		\]
		with
		\[
		|u|\le C\rho^{2-q+\epsilon},\qquad 0<\epsilon\ll q-2.
		\]
	\end{itemize}
	
	To obtain solvability, the right-hand side must decay faster than $\rho^{-2}$, because Hein's Moser-iteration argument requires $u=O(\rho^{2-q})$ to be bounded.  Hence we cannot directly perturb the ansatz $g_{1}$ into a Calabi--Yau metric.  An important feature is that Hein's method respects compact group actions.  Applying the $\mathbb{T}^{n}$-equivariant version of the Poisson result, we obtain
	
	\begin{corollary}\label{heinkejieg1}
		Let $2<q<n+2$ and $0<\epsilon<q-2$.  There exists a bounded Green operator for $\mathbb{T}^{n}$-invariant functions
		\[
		G_{g_{1}}\colon\{f\in C^{0,\alpha}\mid f=O(|\vec{\mu}_{I}|^{-q})\}
		\longrightarrow\{u\in C^{2,\alpha}\mid u=O(|\vec{\mu}_{I}|^{2-q+\epsilon})\}
		\]
		such that $\Delta_{g_{1}}G_{g_{1}}f=f$.
	\end{corollary}
	
	Recall from Subsection~\ref{subsectionsmoothstructure} the model metric $g_{K}$ on a $\Z$-quotient of the Taub--NUT space $\C^{k+1}\times\R^{2(N-k)}$.  There is a splitting $g_{K}=g_{K,1}+g_{2}$ compatible with the product decomposition.  A variant of the preceding discussion yields weighted Sobolev inequalities and Green-function estimates for $g_{K,1}$:
	
	\begin{corollary}\label{heinkejiegK}
		Let $2<q<n+2$ and $0<\epsilon<q-2$.  There exists a bounded Green operator for $\mathbb{T}^{n}$-invariant functions on the model space with metric $g_{K,1}$:
		\[
		G_{K,1}\colon\{f\in C^{0,\alpha}\mid f=O(|\vec{\mu}_{I}|^{-q})\}
		\longrightarrow\{u\in C^{2,\alpha}\mid u=O(|\vec{\mu}_{I}|^{2-q+\epsilon})\}
		\]
		such that $\Delta_{g_{K,1}}G_{K,1}f=f$.
	\end{corollary}
	
	So the Green kernel of $g_{K,1}$ decays like $O(|\vec{\mu}|_{A}^{-n+\epsilon})$ at infinity.

	\subsubsection{Green Operator Estimates for $g_1$ }
		Next we discuss the solvability of the Laplace equation on $\C^{n+1}$
		\begin{align*}
			\Delta_{g_{1}}u=f,
		\end{align*}
		where $f$  is a smooth function that decays at a prescribed rate.  Assume
		\begin{align*}
			\|f\|_{C^{k,\alpha}}\le C\ell_{1}^{\delta}\ell_{n}^{\tau}
		\end{align*}
		for exponents $\delta,\tau$ to be determined.  Our approach follows Section~2.8 of \cite{li2023syz}, the present subsection merely verifies that, in higher dimensions, the decay of the volume-form error established in Lemma~\ref{decayvolumeerror}, together with the asymptotic properties of $g_{1}$ derived in Subsection~\ref{subsectionasymptoticprop}, already suffice to run the same argument.

		Roughly speaking, one approximates the Green kernel of $g_{1}$ by the Green kernels of the standard metric on each region.  For notational simplicity, we write $\dist(\cdot,\cdot)$ for the distance induced by $g_{G_{I}}$ on $\B^{n}_{1}$ and \emph{omit the subscript $1$ from symbols such as $\B^{n}_{K,1}$, $\D^{n}_{K,1}$}. Since $\B^{n}=\B^{n}_{1}\oplus\B^{n}_{2}$ and $g_{A}=g_{G_{I}}\oplus g_{A_{I^{\prime}}}$, this causes no confusion. The reader may regard the simplified notation as referring to the images of the corresponding regions of $\B^{n}$ under the projection $\pi^{n}_{1}$.
	\begin{lemma}\label{PDEaffine1}
		Let $-3<\delta<0$ and $\delta+\tau<0$.  Let $f$ be a $\mathbb{T}^{n}$-invariant function on $\mathbb{C}^{n+1}$ supported in $\{\dist(\,\cdot\,,\mathfrak{D}^{n})\ge 1\}$ with
		\[
		\|f\|_{C^{k,\alpha}}\le C\ell_{1}^{\delta}\ell_{n}^{\tau}.
		\]
		Then the second-order derivatives of the Euclidean potential $\Delta_{G_I}^{-1}f$ satisfy
		\[
		\bigl|\nabla_{G_{I}}^{2}\Delta_{G_I}^{-1}f\bigr|_{g_{G_I}}
		\le C\ell_{1}^{\delta}\ell_{n}^{\tau}.
		\]
		Moreover, if $\delta<-1$ and $\delta+\tau<-1$, then
		\[
		\bigl\|\nabla_{g_{1}}^{2}\Delta_{G_I}^{-1}f\bigr\|_{C^{k,\alpha}(\mathbb{C}^{n+1})}
		\le C\ell_{1}^{\delta}\ell_{n}^{\tau}.
		\]
		The constants depend only on $k,\alpha,\delta,\tau$ and the uniform ellipticity constants $\lambda,\Lambda$.
	\end{lemma}
	\begin{proof}
			We first assume that $f$ is compactly supported.  Let $w$ be the Newton potential of $f$:
			\[
			w(x)=\int_{\R^{n+2}}\Gamma(x-y)\,f(y)\,\mathrm{d} y
			=C\int_{\R^{n+2}}\frac{1}{|x-y|_{G_{I}}^{n}}\,f(y)\,\mathrm{d} y.
			\]
			Taking second derivatives gives
			\begin{align*}
				D_{ij}w(x)
				&=\int_{\R^{n+2}\setminus B_{r_{0}}(x)}D_{ij}\Gamma(x-y)\,f(y)\,\mathrm{d} y\\
				&\quad+\int_{B_{r_{0}}(x)}D_{ij}\Gamma(x-y)\bigl(f(y)-f(x)\bigr)\,\mathrm{d} y
				-c_{ij}f(x),
			\end{align*}
			where $c_{ij}$ are universal constants and $r_{0}>0$ is chosen sufficiently small.  The second integral is easily estimated.
			
			\emph{Step 1:}  $|y|$ far from $|x|$.  Let
			\[
			2^{n}\le|x|<2^{n+1},\qquad
			2^{k}\le|y|\le 2^{k+1},\qquad
			k\le n-2\ \text{or}\ k\ge n+2.
			\]
			Then
			\begin{align*}
				&\int_{2^{k}\le|y|\le 2^{k+1}}\frac{1}{|x-y|^{n+2}}\,|f(y)|\,\mathrm{d} y\\
				&\quad\le C\int_{2^{k}\le|y|\le 2^{k+1}}\ell_{1}^{\delta}|y|^{\tau}\,\mathrm{d} y
				\cdot\sup\Bigl\{\frac{1}{|x-y|^{n+2}}\Bigr\}\\[1mm]
				&\quad\le C^{\prime}(2^{k})^{\delta+\tau+n+2}
				\cdot\sup\Bigl\{\frac{1}{(2^{k}-2^{n+1})^{n+2}},
				\frac{1}{(2^{n}-2^{k+1})^{n+2}}\Bigr\}\\[2mm]
				&\quad\le\begin{cases}
					C^{\prime}2^{(\delta+\tau)k},&k\ge n+2,\\[1mm]
					C^{\prime}2^{(\delta+\tau)k}\cdot 2^{(n+2)(k-n)},&k\le n-2,
				\end{cases}
			\end{align*}
			with
			\[
			C^{\prime}=C\int_{1\le|s|\le 2}
			\bigl(2^{-k}+\dist(s,\D^{n})\bigr)^{\delta}\,\mathrm{d} s.
			\]
			The integral is finite because $\delta>-3$ (after rescaling to the unit annulus).  When $\delta+\tau<0$ the sum over $k$ is controlled by $\ell_{n}^{\delta+\tau}$.
			
			\emph{Step 2:}  $|y|\sim|x|$ but $|x-y|>r_{0}$.  Here
			\begin{align*}
				&\int_{y\sim x,\,|x-y|>r_{0}}\frac{1}{|x-y|^{n+2}}\,
				\ell_{1}^{\delta}\ell_{n}^{\tau}\,\mathrm{d} y\\[1mm]
				&\quad\le C\ell_{n}^{\tau}(x)
				\int_{r_{0}}^{\infty}\frac{1}{t^{n+2-\delta}}\,t^{n+1}\,\mathrm{d} t
				\int_{S^{n+1}}\mathrm{d}^{\delta}\Bigl(\frac{x}{t}+\sigma,\D\Bigr)\,\mathrm{d}\sigma
				\le C\ell_{n}^{\tau}(x)\ell_{1}^{\delta}(x),
			\end{align*}
			where the last integral is finite because $\delta>-3$ and $\delta<0$.
			
			\emph{Step 3:}  The ball $B(x,r_{0})$ with $r_{0}<\frac{1}{10}\ell_{1}(x)$.  Since $f$ is supported in $\{\dist(\,\cdot\,,\D)>1\}$, both $x$ and $y$ stay uniformly away from $\D$.  By Proposition~\ref{inductionhyp}\,(IV), $g_{1}$ is uniformly equivalent to $g_{\mathrm{flat}}$ in this region, so
			\[
			|f(y)-f(x)|\le C\ell_{1}(x)^{\delta-1}|x|^{\tau}\operatorname{dist}(x,y).
			\]
			Inserting this into
			\[
			\int_{|x-y|<\frac{1}{10}\ell_{1}(x)}\frac{1}{|x-y|^{n+2}}\,|f(y)-f(x)|\,\mathrm{d} y
			\]
			gives the desired bound.
			
			We have thus shown
			\[
			|\nabla^{2}_{g_{G_{I}}}\Delta_{G_{I}}^{-1}f|_{g_{G_{I}}}
			\le C\ell_{1}^{\delta}\ell_{n}^{\tau}.
			\]
			Integrating the second derivatives from infinity and using $\delta<-1$, $\delta+\tau<-1$, we also obtain
			\[
			|\mathrm{d}\Delta_{G_{I}}^{-1}f|_{g_{G_{I}}}
			\le C\ell_{1}^{\delta+1}\ell_{n}^{\tau}.
			\]
			
			To estimate the Hessian with respect to $g_{1}$, we expand
			\[
			\nabla^{2}_{g_{1}}\Delta_{G_{I}}^{-1}f
			=\sum\frac{\partial^{2}\Delta_{G_{I}}^{-1}f}{\partial\mu_{i}\partial\mu_{j}}
			\nabla_{g_{1}} \mathrm{d} \mu_{i}\otimes\nabla_{g_{1}} \mathrm{d} \mu_{j}
			+\sum\frac{\partial\Delta_{G_{I}}^{-1}f}{\partial\mu_{i}}
			\nabla^{2}_{g_{1}} \mathrm{d} \mu_{i}
			\]
			(and similar terms involving $\eta$-derivatives).  Using
			\[
			\|\mathrm{d}\mu_{i}\|_{C^{k,\alpha}}\le C,\qquad
			\|\mathrm{d}\eta\|_{C^{k,\alpha}}\le C,
			\]
			we obtain
			\[
			\|\nabla^{2}_{g_{1}}\Delta_{G_{I}}^{-1}f\|_{C^{k,\alpha}(\C^{n+1})}
			\le C\ell_{1}^{\delta}\ell_{n}^{\tau}.
			\]
			
			An approximation argument in the weak topology removes the compact-support assumption, so $\nabla^{2}_{g_{1}}\Delta_{G_{I}}^{-1}$ extends to a bounded linear operator between the weighted H\"older spaces.
	\end{proof}

	By (IV) of Proposition~\ref{inductionhyp}, at a definite distance from the locus, $g$ can be approximated by $g_{\mathrm{flat}}$. Hence we obtain
	
	\begin{lemma}\label{PDEaffine2}
		Under the hypotheses of Lemma~\ref{PDEaffine1},
		\[
		\|\Delta_{g_{1}}\Delta_{G_{I}}^{-1}f-f\|_{C^{k,\alpha}}
		\le C\ell_{1}^{\delta-1}\ell_{n}^{\tau}.
		\]
		In particular, for a sufficiently large constant $N_{1}$,
		\[
		\|\Delta_{g}\Delta_{G_{I}}^{-1}f-f\|_{C^{k,\alpha}
			(\{\dist(\,\cdot\,,\mathfrak{D}^{n})>N_{1}\})}
		\le\frac{C}{N_{1}}\ell_{1}^{\delta}\ell_{n}^{\tau}.
		\]
	\end{lemma}

	 Let $N_{0}=4\hat{C}C_{0}$ be the constant used in the definition of $B_{I}^{\prime\prime}$, so that $p\in B_{K}^{\prime\prime}$ iff $N_{0}\dist(\,\cdot\,,\D^{n}_{K})\le\dist(\,\cdot\,,\partial\D^{n}_{K})$.  Take a sequence of constants $\{N_{i}\}_{i=1}^{n+1}$ with $N_{i+1}\gg N_{i}$. The following lemma relies on the results of Corollary~\ref{heinkejiegK}.
	 
	 \begin{lemma}\label{PDElocus1}
	 	Let $\tau<2$ and $0<\epsilon\ll 1$.  Let $f$ be a $\mathbb{T}^{n}$-invariant function supported in
	 	\[
	 	\{\dist(\,\cdot\,,\partial\mathfrak{D}^{n}_{K})>2NN_{k},\;
	 	\dist(\,\cdot\,,\mathfrak{D}^{n}_{K})<2N_{k}\}
	 	\]
	 	inside the model space, with
	 	\[
	 	\|f\|_{C^{k,\alpha}}\le C\ell_{1}^{\delta}\ell_{n}^{\tau},
	 	\]
	 	where $C_{k+2}\gg C_{k+1}$.  Then
	 	\[
	 	\begin{cases}
	 		\|G_{K}f\|_{C^{k+2,\alpha}}
	 		\le C\ell_{k}^{-1+\epsilon}\ell_{n}^{\tau},
	 		&-1<\tau<1-\epsilon,\\[2mm]
	 		\|G_{K}f\|_{C^{k+2,\alpha}}
	 		\le C\ell_{n}^{-2+\epsilon},
	 		&\tau\le-1.
	 	\end{cases}
	 	\]
	 	The constant depends only on $C_{k+1},\delta,\epsilon,\tau,k,\alpha$ and the uniform ellipticity constants $\lambda,\Lambda$.  In particular, if
	 	\[
	 	\begin{cases}
	 		\text{either}\quad -1<\tau<1,\quad -3+2\epsilon<\delta\le 0,\\[1mm]
	 		\text{or}\quad -2+\epsilon<\tau\le-1,\quad -4+2\epsilon<\delta+\tau,
	 	\end{cases}
	 	\]
	 	then
	 	\[
	 	\|\nabla_{g_{K}}^{2}G_{K}f\|_{C^{k,\alpha}}
	 	\le C\ell_{1}^{\delta}\ell_{n}^{\tau}.
	 	\]
	 \end{lemma}
	 
	 \begin{proof}
	 	Assume first that $f$ is compactly supported.  Choose a lattice on the locus $\D^{n}_{K}$, which is an $(n-|K|+1)$-dimensional subset of $\R^{n}\times\C$ with $2\le|K|\le n$.  Decompose $\D^{n}_{K}$ into small cubes centred at
	 	\[
	 	x_{(m_{i})}=(0,\dots,0,m_{k+1},\dots,m_{n},0)\in\D_{K}.
	 	\]
	 	Break $f$ into a sum of $f_{(m_{i})}$ supported in the cylindrical regions
	 	\[
	 	\{m_{j}\lesssim\nu_{K,j}\lesssim m_{j}+1,\;k+1\le j\le n+1\}\cap \operatorname{supp}f.
	 	\]
	 	Taking $C_{k+2}$ sufficiently large, the support of $f$ lies inside $(\B_{K}^{n})^{\prime\prime}$, so $g_{K}$ splits as a product.  Hein's estimate gives
	 	\[
	 	|G_{K}f_{(m_{i})}|\lesssim
	 	\Bigl(\sum m_{i}^{2}\Bigr)^{\tau/2}
	 	\bigl(|x-x_{(m_{i})}|+1\bigr)^{\epsilon-n},
	 	\]
	 	and elliptic regularity yields
	 	\[
	 	\|G_{K}f_{(m_{i})}\|_{C^{k+2,\alpha}(B(x,\ell_{1}(x)/10))}
	 	\le C\Bigl(\sum_{j=k+1}^{n+1}m_{j}^{2}\Bigr)^{\tau/2}
	 	\bigl(|x-x_{(m_{i})}|+1\bigr)^{\epsilon-n}.
	 	\]
	 	Summing over $(m_{i})$ and replacing the sum by an integral, we obtain
	 	\begin{align*}
	 		&\|G_{K}f\|_{C^{k+2,\alpha}(B(x,\ell_{1}(x)/10))}\\[1mm]
	 		&\quad\le C\sum_{(m_{i})}
	 		\Bigl(\sum_{j=k+1}^{n+1}m_{j}^{2}\Bigr)^{\tau/2}
	 		\bigl(|x-x_{(m_{i})}|+1\bigr)^{\epsilon-n}\\[1mm]
	 		&\quad\le C\int_{1}^{\infty}\!\cdots\!\int_{1}^{\infty}
	 		\Bigl(\sum_{i=k+1}^{n+1}y_{i}^{2}\Bigr)^{\tau/2}
	 		\bigl(\ell_{k}(x)^{2}+\sum|\nu_{K,j}(x)-y_{j}|^{2}\bigr)^{\frac{\epsilon-n}{2}}
	 		\,\mathrm{d} y_{i}\\[2mm]
	 		&\quad\lesssim
	 		\begin{cases}
	 			(|x|_{G_{I}}+1)^{\tau}\ell_{k}(x)^{\epsilon-1},
	 			&-1<\tau<1-\epsilon,\\[1mm]
	 			(|x|_{G_{I}}+1)^{\epsilon-2},
	 			&\tau\le-1.
	 		\end{cases}
	 	\end{align*}
	 	Here we have used polar coordinates and observed that the worst convergence occurs when $|K|=2$.  Consequently,
	 	\[
	 	\begin{cases}
	 		\|G_{K}f\|_{C^{k+2,\alpha}}
	 		\le C\ell_{k}^{-1+\epsilon}\ell_{n}^{\tau},
	 		&-1<\tau<1-\epsilon,\\[2mm]
	 		\|G_{K}f\|_{C^{k+2,\alpha}}
	 		\le C\ell_{n}^{-2+\epsilon},
	 		&\tau\le-1,
	 	\end{cases}
	 	\]
	 	and the Hessian satisfies
	 	\[
	 	\begin{cases}
	 		\|\nabla_{g_{K}}^{2}G_{K}f\|_{C^{k,\alpha}}
	 		\le C\ell_{1}^{-2}\ell_{k}^{-1+\epsilon}\ell_{n}^{\tau},
	 		&-1<\tau<1-\epsilon,\\[2mm]
	 		\|\nabla_{g_{K}}^{2}G_{K}f\|_{C^{k,\alpha}}
	 		\le C\ell_{1}^{-2}\ell_{n}^{-2+\epsilon},
	 		&\tau\le-1.
	 	\end{cases}
	 	\]
	 	
	 	An approximation argument in the weak topology removes the compact-support assumption, and the proof is complete.
	 \end{proof}
	 
	 Define the cut-off function
	 \[
	 \chi_{K}=
	 \begin{cases}
	 	1,& \dist(\,\cdot\,,\partial\mathfrak{D}^{n}_{K})>3N_{0}N_{k+1}\;
	 	\text{and}\;\dist(\,\cdot\,,\mathfrak{D}^{n}_{K})<N_{k+1},\\[2mm]
	 	0,& \dist(\,\cdot\,,\partial\mathfrak{D}^{n}_{K})<2N_{0}N_{k+1}\;
	 	\text{or}\;\dist(\,\cdot\,,\mathfrak{D}^{n}_{K})>2N_{k+1}.
	 \end{cases}
	 \]
	 
	 \begin{lemma}\label{PDElocus2}
	 	Under the hypotheses of Lemma~\ref{PDElocus1}, if $C_{3}$ is sufficiently large, then
	 	\[
	 	\|\Delta_{g}(\chi_{K}G_{K}f)-f\|_{C^{k,\alpha}
	 		(\{\dist(\,\cdot\,,\partial\mathfrak{D}^{n}_{K})>3NN_{k+1}\})}
	 	\le\frac{C}{N_{k+1}^{\epsilon}}\ell_{1}^{\delta}\ell_{n}^{\tau}.
	 	\]
	 \end{lemma}
	 
	 \begin{proof}
	 	On the set $\{\dist(\,\cdot\,,\partial\mathfrak{D}^{n}_{K})>3NN_{k+1}\}$ the cut-off $\chi_{K}$ equals $1$ on the support of $f$, so we only need to estimate
	 	\[
	 	I=\|(\Delta_{g}-\Delta_{K})G_{K}f\|.
	 	\]
	 	By Remark~\ref{devationremark1},
	 	\[
	 	\|g_{K}-g\|_{C^{k,\alpha}}\le C\ell_{k}^{-1+\epsilon},
	 	\]
	 	and, in particular, on the support of $f$,
	 	\[
	 	\|g_{K}-g\|_{C^{k,\alpha}}\le\frac{C}{N_{k+1}^{1-\epsilon}}.
	 	\]
	 	Hence the error term is bounded by $\dfrac{C}{N_{k+1}^{1-\epsilon}}$.  On the set $\{\mathrm{d}\chi_{K}\neq 0\}$ we also have to estimate $\|\Delta_{g}(\chi_{K}G_{K}f)\|$, since the scaling factor of $\chi_{K}$ is $N_{k+1}^{-1}$, the desired bound follows from Lemma~\ref{PDElocus1}.
	 \end{proof}

	\begin{lemma}\label{PDEB0}
		Assume either
		\[
		-2\le\delta\le 0,\quad \tau>-2,
		\qquad\text{or}\qquad
		\delta\le-2,\quad \delta+\tau>-4,
		\]
		and let $0<\epsilon\ll 1$ depend on $\delta,\tau$.  If $f$ is supported in the ball
		\[
		\{\dist(\,\cdot\,,\D_{I})<2N_{n+1}\}
		\]
		with
		\[
		\|f\|_{C^{k,\alpha}(\mathbb{C}^{n+1})}\le C\ell_{1}^{\delta}\ell_{n}^{\tau},
		\]
		then
		\[
		\|G_{g_{1}}f\|_{C^{k,\alpha}}\le C\ell_{n}^{-2+\epsilon},
		\]
		and, in particular,
		\[
		\|\nabla_{g_{1}}^{2}G_{g_{1}}f\|_{C^{k,\alpha}}\le C\ell_{1}^{\delta}\ell_{n}^{\tau}.
		\]
	\end{lemma}
	
	\begin{proof}
		This follows directly from Corollary~\ref{heinkejieg1}.
	\end{proof}
	
	\subsubsection{Solving the Laplace Equation}
	Observe that the ranges of the weight exponents $\delta,\tau$ in all the lemmas above overlap, as in Li--Yang \cite{li2023syz}, the common admissible region is
	\begin{equation}\label{GoodRange}
		\{-3<\delta<-1,\;-3<\tau<1,\;-4<\delta+\tau<-1\}.
	\end{equation}
	We thus obtain the following solvability statement for the Laplace equation.
	
	\begin{lemma}(Proposition~2.23 of \cite{li2023syz})\label{laplaciansov}
		Let $(\delta,\tau)$ lie in the range \eqref{GoodRange}.  Given $f$ with
		\[
		\|f\|_{C^{k,\alpha}}\le C\ell_{1}^{\delta}\ell_{n}^{\tau},
		\]
		there exists a function $u$ solving $\Delta_{g}u=f$ and satisfying
		\[
		\|\mathrm{d} u\|_{C^{k+1,\alpha}}\le C\ell_{1}^{\delta+1}\ell_{n}^{\tau}.
		\]
	\end{lemma}
	
	\begin{proof}
		Choose a cut-off function
		\[
		\chi^{\prime}=
		\begin{cases}
			1,&\dist(\,\cdot\,,\mathfrak{D}^{n})\ge 2,\\[1mm]
			0,&\dist(\,\cdot\,,\mathfrak{D}^{n})\le 1.
		\end{cases}
		\]
		Then $\chi^{\prime}f$ is supported in $\{\dist(\,\cdot\,,\mathfrak{D}^{n})\ge 1\}$ and satisfies
		\[
		\|\chi^{\prime}f\|_{C^{k,\alpha}}\le C\ell_{1}^{\delta}\ell_{n}^{\tau}.
		\]
		
		By Lemma~\ref{PDEaffine1}, $u_{a}=\Delta_{G_I}^{-1}(\chi^{\prime}f)$ satisfies
		\[
		\|\nabla^{2}_{g_{1}}u_{a}\|_{C^{k,\alpha}}\le C\ell_{1}^{\delta}\ell_{n}^{\tau},
		\]
		and, for $N_{1}\gg 1$, Lemma~\ref{PDEaffine2} gives
		\[
		\|\Delta_{g_{1}}u_{a}-f\|_{C^{k,\alpha}(\{\dist(\,\cdot\,,\mathfrak{D}^{n})>N_{1}\})}
		\le\frac{C}{N_{1}}\ell_{1}^{\delta}\ell_{n}^{\tau}.
		\]
		
		For $K\subset I$ with $|K|=k+1$, $k\ge 1$, define
		\[
		\chi^{\prime}_{K}=
		\begin{cases}
			1,&\dist(\,\cdot\,,\partial\mathfrak{D}^{n}_{K})>3N_{0}N_{k},\;
			\dist(\,\cdot\,,\mathfrak{D}^{n}_{K})<N_{k},\\[2mm]
			0,&\dist(\,\cdot\,,\partial\mathfrak{D}^{n}_{K})<2N_{0}N_{k}\;
			\text{or}\;\dist(\,\cdot\,,\mathfrak{D}^{n}_{K})>2N_{k}.
		\end{cases}
		\]
		The scaling length of $\chi^{\prime}_{K}$ is approximately $N_{k}$.  Set
		\[
		f_{ij}=\chi^{\prime}_{ij}(f-\Delta_{g_{1}}u_{a}),
		\]
		so that $f_{ij}$ is supported in
		\[
		\{\dist(\,\cdot\,,\partial\mathfrak{D}^{n}_{ij})>2NN_{1},\;
		\dist(\,\cdot\,,\mathfrak{D}^{n}_{ij})<2N_{1}\}
		\]
		and satisfies
		\[
		\|f_{ij}\|_{C^{k,\alpha}}\le C\ell_{1}^{\delta}\ell_{n}^{\tau}.
		\]
		By Lemma~\ref{PDElocus1}, there exists $G_{\{i,j\}}f_{ij}$ with
		\[
		\|\nabla_{g_{\{i,j\}}}^{2}G_{\{i,j\}}f_{ij}\|_{C^{k,\alpha}}
		\le C\ell_{1}^{\delta}\ell_{n}^{\tau}.
		\]
		Put $u_{ij}=\chi_{ij}G_{\{i,j\}}f_{ij}$, where $\chi_{ij}$ is the cut-off defined before Lemma~\ref{PDElocus2}.  Then $u_{ij}$ is globally defined on $\C^{n+1}$ and, by Lemma~\ref{PDElocus2},
		\[
		\|\Delta_{g_{1}}u_{ij}-f_{ij}\|_{C^{k,\alpha}
			(\{\dist(\,\cdot\,,\partial\mathfrak{D}^{n}_{ij})>3NN_{2}\})}
		\le\frac{C}{N_{2}^{\epsilon}}\ell_{1}^{\delta}\ell_{n}^{\tau}.
		\]
		
		Repeating the procedure for $K\subset I$, $|K|=k+1$, we obtain $u_{K}$ with analogous estimates.  Finally, let
		\[
		\chi_{0}^{\prime}=
		\begin{cases}
			1,&\dist(\,\cdot\,,\mathfrak{D}^{n}_{I})\le N_{n+1},\\[1mm]
			0,&\dist(\,\cdot\,,\mathfrak{D}^{n}_{I})>2N_{n+1},
		\end{cases}
		\]
		and set
		\[
		f_{0}=\chi_{0}^{\prime}\Bigl(f-\Delta_{g_{1}}\bigl(u_{a}
		+\sum_{|K|=2}^{n}u_{K}\bigr)\Bigr).
		\]
		Then $f_{0}$ is supported in $\{\dist(\,\cdot\,,\mathfrak{D}^{n}_{I})<2N_{n+1}\}$ and satisfies
		\[
		\|f_{0}\|_{C^{k,\alpha}}\le C\ell_{1}^{\delta}\ell_{n}^{\tau}.
		\]
		By Lemma~\ref{PDEB0}, there exists $u_{0}=G_{g_{1}}f_{0}$ with
		\[
		\|\nabla_{g_{1}}^{2}u_{0}\|_{C^{k,\alpha}}\le C\ell_{1}^{\delta}\ell_{n}^{\tau},
		\qquad
		\Delta_{g_{1}}u_{0}=f_{0}.
		\]
		
		Let
		\[
		u=u_{0}+u_{a}+\sum_{|K|=2}^{n}u_{K}.
		\]
		Since $N_{k+1}\gg N_{k}$, we have
		\[
		\|\nabla_{g_{1}}^{2}u\|_{C^{k,\alpha}}\le C\ell_{1}^{\delta}\ell_{n}^{\tau},
		\qquad
		\|\Delta_{g_{1}}u-f\|_{C^{k,\alpha}}\le\frac{1}{N_{1}^{\epsilon}}\ell_{1}^{\delta}\ell_{n}^{\tau}.
		\]
		Thus $u$ is an approximate solution with the required bounds.  The iteration to obtain an exact solution follows the same argument as in Proposition~2.23 of Li--Yang \cite{li2023syz}, which we omit here.
	\end{proof}
	\subsubsection{Solving the Complex Monge--Amp{\`e}re Equation}
	Observe that the volume-form error $E$ satisfies
	\[
	\|E\|_{C^{k,\alpha}}\le C\ell_{1}^{-1+\epsilon}\ell_{n}^{-1-\epsilon}.
	\]
	Applying Lemma~\ref{laplaciansov}, we solve the Poisson equation
	\[
	\Delta_{g_{1}}u^{\prime}=-2E,\qquad
	\|\mathrm{d} u^{\prime}\|_{C^{k+1,\alpha}}\le C\ell_{1}^{-\epsilon}\ell_{n}^{-1+\epsilon},
	\]
	so that, in particular,
	\[
	\|\partial\bar\partial u^{\prime}\|_{C^{k,\alpha}}
	\le C\ell_{1}^{-1-\epsilon}\ell_{n}^{-1+\epsilon},\qquad
	\|(\partial\bar\partial u^{\prime})^{2}\|_{C^{k,\alpha}}
	\le C\ell_{1}^{-2-2\epsilon}\ell_{n}^{-2+2\epsilon}.
	\]
	Set $\omega_{1}^{\prime}=\omega_{1}+\sqrt{-1}\, \partial\bar\partial u^{\prime}$.  Then
	\begin{align*}
		(\omega_{1}^{\prime})^{n+1}
		&=(\omega_{1}+\sqrt{-1}\, \partial\bar\partial u^{\prime})^{n+1}\\[1mm]
		&=(\omega_{1})^{n+1}\Bigl(1+\frac{1}{2}\Delta_{g_{1}}u^{\prime}
		+O\!\bigl(|\partial\bar\partial u^{\prime}|^{2}\bigr)\Bigr),
	\end{align*}
	and the new volume-form error $E^{\prime}$ satisfies
	\[
	\|E^{\prime}\|_{C^{k,\alpha}}\le C \ell_{1}^{-2}\ell_{n}^{-2+2\epsilon}.
	\]
	Outside a compact set the modification to $\omega_{1}$ is $C^{0}$-small, so positivity of the K\"ahler metric is preserved.  Inside the compact region we argue as in Section~\ref{subsctionsurgeryoriginal}, using the deviation estimate for $u^{\prime}$ to control the growth of the moment map of $g_{1}^{\prime}$ via an analogue of Lemma~\ref{zzengzhangxing}.  We still denote the resulting K\"ahler metric by $\omega_{1}^{\prime}$,  it inherits all analytic properties of $\omega_{1}$.
	
	Applying Lemma~\ref{laplaciansov} again with background metric $g_{1}^{\prime}$, we solve
	\[
	\Delta_{g_{1}^{\prime}}u^{\prime\prime}=-2E^{\prime},\qquad
	\|\mathrm{d} u^{\prime\prime}\|_{C^{k+1,\alpha}}
	\le C \ell_{1}^{-1}\ell_{n}^{-2+2\epsilon},
	\]
	and set $\omega_{1}^{\prime\prime}=\omega_{1}^{\prime}+\sqrt{-1}\, \partial\bar\partial u^{\prime\prime}$.  The identity
	\[
	(\omega_{1}^{\prime}+\sqrt{-1}\, \partial\bar\partial u^{\prime\prime})^{n+1}
	=(\omega_{1}^{\prime})^{n+1}
	\Bigl(1+\frac{1}{2}\Delta_{g_{1}^{\prime}}u^{\prime\prime}
	+O\!\bigl(|\partial\bar\partial u^{\prime\prime}|^{2}\bigr)\Bigr)
	\]
	yields volume-form error $E^{\prime\prime}$ satisfies
	\[
	\|E^{\prime\prime}\|_{C^{k,\alpha}}\le C \ell_{1}^{-4}\ell_{n}^{-4+2\epsilon}.
	\]
	A further surgery in the compact region preserves the K\"ahler property.  The metric $\omega_{1}^{\prime\prime}$ now satisfies all the hypotheses required by Tian--Yau--Hein package.  Theorem~\ref{MAkejiexing} follows by Hein's estimates and elliptic bootstrap.

	\bibliographystyle{amsplain}   
	\bibliography{main}

\end{document}